\documentclass{amsart}

\usepackage{amsmath,amssymb,amsthm,mathrsfs}
\usepackage{fullpage}
\usepackage{longtable}
\usepackage{array}
\usepackage{hyperref}
\usepackage{color}
\usepackage[alphabetic]{amsrefs}
\usepackage[all]{xy}
\usepackage{tikz}
\usepackage{graphicx}

\allowdisplaybreaks

\title{Symmetric $\mathcal{A}$ actions on $\mathcal{A}(2)$}
\author{Robert R. Bruner}
\address{Department of Mathematics, Wayne State University, Detroit, MI, USA}
\email{robert.bruner@wayne.edu} \urladdr{http://www.rrb.wayne.edu/}
\newtheorem{theorem}{Theorem}[section]
\newtheorem{proposition}[theorem]{Proposition}

\newtheorem{corollary}[theorem]{Corollary}
\newtheorem{conjecture}[theorem]{Conjecture}
\theoremstyle{definition}
\newtheorem{definition}[theorem]{Definition}

\theoremstyle{remark}

\newtheorem{remark}[theorem]{Remark}

\DeclareMathOperator{\Hom}{Hom}

\DeclareMathOperator{\Img}{Im}

\newcommand{\longto}{\longrightarrow}

\newcommand{\modmod}{/\!/}

\newcommand{\cA}{\mathcal{A}}
\newcommand{\cB}{\mathcal{B}}

\newcommand{\cQ}{\mathcal{Q}}
\newcommand{\F}{\mathbf{F}}

\begin{document}

\begin{abstract}
We describe the variety of
left actions of the mod 2 Steenrod algebra $\cA$ on its
subalgebra $\cA(2)$ that extend the action of $\cA(2)$
on itself and preserve the short exact sequence
\[
% \label{Zext}
\tag{$\mathcal{Q}$}
0 \to \Img{Q_2} \longto \cA(2) \longto \cA(2)\modmod E[Q_2] \to 0
\]
together with the isomorphism between $\Img{Q_2}$
and the quotient $\cA(2)\modmod E[Q_2]$.
We call such $\cA$ actions {\em symmetric}.
These arise as the cohomology of $v_2$ self maps
$\Sigma^7 Z \longto Z$, as in \cite{BE} and \cite{BBBCX}.

We find that there are $256$ $\F_2$ points in this variety,
arising from
$16$ such actions of $Sq^8$ and, independently,
another $16$ actions of $Sq^{16}$.
This is in contrast to the $1600$ $\cA$-module structures
on $\cA(2)$ found by~\cite{Roth}, which do not necessarily
relate in this fashion to $\Img{Q_2}$.

We also describe the variety  of actions
found by Roth, arising from $100$ possible $Sq^8$
actions and, independently, $16$ possible $Sq^{16}$ actions,
and the embedding of the variety of symmetric actions into
the variety of all actions.  We then describe two related varieties
and the maps between them.

Next, we examine the effect of Spanier-Whitehead duality on
these $\cA$-actions.

Finally, we note that the actions which have been used in the literature
correspond to the simplest choices, in which all the coordinates
equal zero.

\end{abstract}
\maketitle

\setcounter{tocdepth}{1}
\tableofcontents

\renewcommand*{\arraystretch}{1.4}

\section{Introduction and Results}

Let $\cA$ be the mod 2 Steenrod algebra and let $\cA(2)$
be the sub Hopf algebra generated by $Sq^1$, $Sq^2$ and $Sq^4$.
Spectra whose cohomology is $\cA(2)$, and related spectra, have
proven to be useful tools in algebraic topology.
Such cohomology modules have
actions of the whole Steenrod algebra.  It is therefore interesting
to inquire about the possibilities of such actions.  

The spectra of most interest in the literature,
(e.g., see~\cite{BE} and \cite{BBBCX}), arise as the cofibers of maps
\[
\Sigma^7 Z \overset{v}{\longto} Z
\]
detected by $Q_2 \in \cA(2)$, whose cohomology long exact sequence
is the short exact sequence 
\[
\label{Qext}
\tag{${\mathcal{Q}}$}
0 \to \Img{Q_2} \longto \cA(2) \longto \cA(2)\modmod E[Q_2] \to 0
\]
of $\cA$-modules, with some, as of yet unspecified, $\cA$ action
extending the evident $\cA(2)$ action.
Recall that 
\[
\cA(2)\modmod E[Q_2] = \cA(2)/\cA(2)Q_2
\]
and that $\Img{Q_2} = \cA(2)Q_2 = Q_2\cA(2)$.
Here, $Q_2 = [Sq^4, Q_1]$ with
$Q_1 = [Sq^2, Sq^1]$, and $Q_2$ is central in $\cA(2)$.
In this situation, there is
an isomorphism of $\cA$-modules $\Img{Q_2} \cong \Sigma^7
\cA(2)\modmod E[Q_2]$.
Our primary focus here is the variety of such actions.  

\begin{definition}
\label{defsym}
A {\em symmetric} $\cA$ action on $\cA(2)$ is one which
extends the product on $\cA(2)$, and further, is required 
\begin{enumerate}
\item
to make
the submodule $\Img{Q_2}$ an $\cA$-submodule, and 
\item
\label{defsym2}
to make the $\cA(2)$ isomorphism
$\Sigma^7 \cA(2)\modmod E[Q_2] \longto \Img{Q_2}$
given by $[x] = x + \Img{Q_2} \mapsto xQ_2$
an $\cA$-module isomorphism.
\end{enumerate}
\end{definition}

\begin{definition}
Let $\cB(2)  = \cA(2)\modmod E[Q_2]$.
\end{definition}

If we implicitly incorporate the isomorphism in~\ref{defsym}.(2),
then we are parameterizing the $\cA$-module extensions
\begin{equation}
\label{AZext}
0 \to \Sigma^7 B \longto A \longto B \to 0
\end{equation}
whose restriction to $\cA(2)$-Mod is the
short exact sequence
\begin{equation}
\label{Zext}
\tag{$\mathcal{E}$}
0 \to \Sigma^7\cB(2) \longto \cA(2) \longto \cB(2) \to 0.
\end{equation}
These are the actions which we call {\em symmetric}.

In Section~\ref{sec:symmetric} we 
describe the variety of symmetric $\cA$-actions on $\cA(2)$
and show that it has $256$ $\F_2$ points.   By an observation
of Marcel B\"okstedt, it fibers over an affine space $\F_2^4$
parameterizing the $Sq^8$ actions, with fiber $\F_2^4$
parameterizing the $Sq^{16}$-actions.

We then consider the general case in Section~\ref{sec:general},
recovering the result of Marilyn Roth in~\cite{Roth} that there
are $1600$ $\cA$-actions on $\cA(2)$ if we do not impose the
condition of symmetry.  Again, we describe this as the $\F_2$ 
points of an algebraic variety.

In each case, since we are only interested in the $\F_2$ points,
we augment the relations imposed by the Adem relations by the
relation $x^2 = x$ for each of the parameters $x$ defining the action.
In the section on duality, this has interesting effects
related to the failure of the Nullstellensatz over fields
which are not algebraically closed.

We round out these considerations by describing the variety
of $\cA$-actions on $\cB(2)$ in Section~\ref{sec:Baction},
showing that there are $32$ such actions, of which $28$ 
lift to $\cA$-actions on $\cA(2)$.

Since the symmetric actions are actions, there is an inclusion map of varieties.
This can be factored through
an intermediate variety which consists of those $\cA$-actions
which preserve  the submodule $\Img{Q_2}$.   
While symmetric actions induce the same $\cA$-action
on $\Img{Q_2}$ and on $\cB(2)$, these intermediate
actions induce possibly distinct $\cA$-module
structures on $\Img{Q_2}$ and $\cB(2)$.   Recall that, as $\cA(2)$-modules,
these are isomorphic.

Let us write $V_{\text{gen}}$,
$V_{\text{sym}}$,
$V_{\cB}$, and
$V_{\cQ}$ for the varieties of
\begin{itemize}
\item general $\cA$-module structures on $\cA(2)$,
\item symmetric $\cA$-module structures on $\cA(2)$,
\item $\cA$-module structures on $\cB(2)$, and
\item $\cA$-module structures on $\cA(2)$ that preserve the short exact 
sequence~\eqref{Qext},
\end{itemize}
respectively.   Then we have a diagram
\[
\tag{$\mathcal{D}$}
\label{diagram:D}
\xymatrix{
&
V_{\cB}
&
\\
V_{\text{sym}} 
\ar[r]
&
V_{\cQ}
\ar[r]
\ar^{s}[u]
\ar_{q}[d]
&
V_{\text{gen}}
\\
&
V_{\cB}
\\
}
\]
where $sA$ and $qA$ are the $\cA$-modules $\Img{Q_2}$ and
$A/\Img{Q_2}$ induced by an $\cA$-module $A$ which lies in $V_\cQ$.

In Section~\ref{sec:relations} we describe each of these maps.
In particular, in Theorem~\ref{thm:VQ} we note that $V_\cQ$ is simply the
intersection of $V_{\text{gen}}$ with a hyperplane $a_{60} = 0$.

Next, we consider how these $\cA$-module structures behave
under duality.  We recall some standard definitions and facts.

\begin{definition}
If $M$ is an $\cA$-module, let $DM = \Hom_\cA(M,\F_2)$ with
$\cA$ action $a \cdot f = f \circ L_{\chi(a)}$, where $\chi$ is
the antipode and $L$ is left translation, $L_{y}(x)
= yx$.  
We define $DM$ similarly for $\cA(2)$-modules, or modules over any
other Hopf algebra (with antipode).
If $V$ is an $\F_2$-vector space (such as $M_k$, the degree $k$ part of
an $\cA$-module $M$) let $V^* = \Hom_{\F_2}(V,\F_2)$.
\end{definition}

\begin{proposition}
\label{prop:gendual}
Let $H$ be a connected finite dimensional Hopf  algebra of formal
dimension $N$ with socle $H_N = \langle s \rangle$, and let
$s^* : H_N \overset{\cong}{\longto} \Sigma^N \F_2$ be the linear dual of $s$.
There is an isomorphism of left $H$-modules $\theta : H \to \Sigma^N DH$
given by $\theta(x) = s^* \circ L_{\chi(x)}$.
\end{proposition}

\begin{corollary}
\label{cor:ABdual}
There is an isomorphism of $\cA(2)$-modules
$\theta: \cA(2) \longto \Sigma^{23}D\cA(2)$.
There is an isomorphism of $\cB(2)$-, and hence, of $\cA(2)$-modules
$\theta: \cB(2) \longto \Sigma^{16}D\cB(2)$.
\end{corollary}

We can use this duality to transport an $\cA$-module structure on
$\cA(2)$ to the dual $\cA$ action on $D\cA(2)$, and then along $\theta$
to get a dual $\cA$-module structure on $\cA(2)$.   The same can be 
applied to $\cB$.  These induce maps on the varieties $V_{\text{gen}}$,
$V_{\text{sym}}$, $V_{\cQ}$, and $V_{\cB}$ 
which we determine in Section~\ref{sec:duality}.

\begin{remark}
The isomorphisms in Corollary~\ref{cor:ABdual}
have a remarkably simple form with respect to the Milnor basis
as shown in Theorems~\ref{prop:Adual} and \ref{prop:Bdual}.
This simple form suggests a general result, Conjecture~\ref{conj:dual}.
\end{remark}

Finally, in Section~\ref{sec:literature}
we identify the $\cA$-module structures that have appeared
in the literature and determine their duals.

The author wishes to thank
Prasit Bhattacharya, Marcel B\"okstedt,
and John Rognes for useful conversations while working
out these results.

\section{The symmetric case}
\label{sec:symmetric}

We can describe the variety of symmetric or general $\cA$-actions on $\cA(2)$
as follows.
First, note that $\cA(2)$ is
concentrated in degrees 0 to 23, so that all $Sq^{2^i}$ with
$2^i > 23$ must act trivially.  
Since the $Sq^{2^i}$ generate $\cA$, and since the  $Sq^1$, $Sq^2$
and  $Sq^4$ actions are already determined by the action of $\cA(2)$ on itself,
it suffices to describe $Sq^8$ and $Sq^{16}$.  
Each of these is determined by a sequence of linear transformations
$A_i \longto A_{i+8}$ or $A_{i+16}$.   We introduce indeterminates
for the coordinates of these linear transformations.   
This requires $119$ variables $a_1$, \ldots, $d_{24}$ in the symmetric case, and $150$
variables $a_1$, \ldots, $b_{26}$ in the general case.

Certain Adem relations allow us to determine $Sq^9$ through $Sq^{23}$ 
in terms of these indeterminates.   We then compute all Adem relations,
resulting in a set of relations which determine the subvariety 
of $\F_2^{119}$ (resp., $\F_2^{150}$) consisting of the symmetric
(resp, all) $\cA$-module actions on $\cA(2)$.  The relations are initially of
degree $\leq 2$, since the Adem relations are, but we can dramatically
reduce the number of variables required at the cost of introducing 
higher degree relations.

We will use the Milnor basis notation throughout:   
$Sq{(r_1,\ldots,r_k)}$ is
the dual of $\xi_1^{r_1}\cdots\xi_k^{r_k}$.

\clearpage

\begin{theorem}\leavevmode
\label{thm:symmetric}
\begin{enumerate}
\item
The variety of symmetric $\cA$-module actions on $\cA(2)$ is defined by
\[
\F_2[a_1, a_2, a_{13}, a_{23}, b_1]/I \otimes 
  \F_2[c_1,d_1,d_2,d_3]
\]
where 
\[
    I = (a_{1} a_{2} + a_{1} a_{13} + a_{2} a_{13} a_{23} + 
    a_{2} a_{13} b_{1} + a_{2} a_{13} + a_{2} a_{23} +
    a_{2} b_{1} + a_{2} + b_{1}).
\]
The first factor defines the $Sq^8$ action and has $16$ $\F_2$ points.
The second factor defines the $Sq^{16}$ action and also has $16$ $\F_2$ points.
This gives $256$ ways to define a symmetric $\cA$-module structure on $\cA(2)$.
\item
The $Sq^8$ actions are described in Appendices~\ref{sq8}, \ref{allsymSq8} and
\ref{symrelations}.  The coordinates $a_1$, \ldots, $b_1$ determine, and are
determined by, the following:
\begin{itemize}
\item
$ Sq^8 \cdot 1  = a_1 Sq{(5,1)} + a_2 Sq{(2,2)} + b_1 Sq{(1,0,1)},$
\item
$a_{13}$ is the coefficient of $Sq{(6,2)}$ in $Sq^8 \cdot Sq^4$, and
\item
$a_{23}$ is the coefficient of $Sq{(5,3)}$ in $Sq^8 \cdot Sq{(0,2)}$.
\end{itemize}
\item
The $Sq^{16}$ actions are described in Appendices~\ref{sq16} and
\ref{symrelations}.  
The coordinates $c_1$, $d_1$, $d_2$ and $d_3$ determine, and are
determined by
\[
Sq^{16} \cdot 1 = c_1 Sq(7,3) + d_1 Sq(6,1,1) 
                + d_2 Sq(3,2,1) + d_3 Sq(0,3,1).
\]
\end{enumerate}
\end{theorem}

% There are several other equivalent ways to describe the last two coordinates
% determining $Sq^8$:
% \begin{align*}
% Sq^8 \cdot Sq{(5,1)} & = a_{13} Sq{(7,3)} + Sq{(6,1,1)} + Sq{(3,2,1)}, \\
% Sq^8 \cdot Sq{(5,0,1)} & = a_{13} Sq{(7,2,1)}, \\
% Sq^8 \cdot Sq{(6,0,1)} & = (1 + a_{13}) Sq{(5,3,1)}, \\
% Sq^8 \cdot Sq{(4,1,1)} & = a_{13} Sq{(6,3,1)}, \\
% Sq^8 \cdot Sq{(5,1,1)} & = a_{13} Sq{(7,3,1)}, \\[2ex]
% Sq^8 \cdot Sq{(0,2,1)} & = a_{23} Sq{(5,3,1)}, \\
% Sq^8 \cdot Sq{(2,2,1)} & = a_{23} Sq{(7,3,1)}. \\
% \end{align*}

% The description in Theorem~\ref{thm:symmetric} relates easily to the
% coefficients of the $\cA$ action, as in parts (2) and (3) of the Theorem.
The variety of $Sq^8$ actions is in fact an affine space.
I owe this observation to Marcel B\"okstedt.
Let $V_2(R) = \Hom_{\text{Ring}}(R,\F_2)$ denote the variety of $\F_2$ points.

\begin{corollary}
\label{affine}
The space of $Sq^8$ actions  in Theorem~\ref{thm:symmetric}
is $\F_2^4$, as shown by the inverse isomorphisms
\[
F : V_2(\F_2[a_1,a_2,a_{13},z]) \longto V_2(\F_2[a_1, a_2, a_{13},a_{23}, b_1]/I)
\]
and
\[
G : V_2(\F_2[a_1, a_2, a_{13},a_{23}, b_1]/I) \longto V_2(\F_2[a_1,a_2,a_{13},z])
\]
given by 
\[
G(a_1,a_2,a_{13},a_{23},b_1) = (a_1, a_2, a_{13}, a_{23}+b_1)
\]
and
\[
F(a_1, a_2, a_{13},z) = (a_1,a_2,a_{13},z+f, f),
\]
where $f = f(a_1,a_2,a_{13},z) = z(a_2 + a_2a_{13}) + a_2 + a_1a_2
+ a_1 a_{13} + a_2 a_{13}$.
\end{corollary}

\begin{proof}[Proof of Corollary~\ref{affine}]
Clearly $GF = Id$.  The composite
\[
FG(a_1,a_2,a_{13},a_{23},b_1) =
(a_1, a_2,a_{13}, a_{23}+b_1 + f,f)
\]
is the identity since $f \equiv b_1 \mod{I}$.   Finally,
$F$ does map to the variety where $I = 0$ since the relations
$a_{23} = z+f$ and $b_1=f$ make the relation
\[
    0 = a_{1} a_{2} + a_{1} a_{13} + a_{2} a_{13} a_{23} + 
    a_{2} a_{13} b_{1} + a_{2} a_{13} + a_{2} a_{23} +
    a_{2} b_{1} + a_{2} + b_{1}
\]
true.
\end{proof}

\begin{proof}[Proof of Theorem~\ref{thm:symmetric}]

Let $A = \cA(2)$ and let $A_n$ be its degree $n$ component.
We use the Milnor basis 
\[
\left\{ Sq(r_1,r_2,r_3) \mid 0 \leq r_1 \leq 7,\,\, 0 \leq r_2 \leq 3,\,\,
0 \leq r_3 \leq 1\right\}
\]
as our basis for $A$.  
Let $B_n$ be the sub vector space of $A_n$ spanned by the 
$\left\{ Sq(R) \mid r_3 = 0\right\}$.  Then
\begin{equation}
\label{split}
\tag{$\mathcal{V}$}
A_n = B_n \oplus B_{n-7} Q_2
\end{equation}
as $\F_2$-vector spaces.
Note that $Q_2$ is central in $\cA(2)$, and that, in fact
\[
Q_2 Sq(r_1,r_2,r_3) =
Sq(r_1,r_2,r_3) Q_2  =
\left\{
\begin{array}{lr}
Sq(r_1,r_2,1) & r_3=0{\phantom{.}}\\
0 & r_3=1. \\
\end{array}
\right.
\]

The vector space splitting~\eqref{split} is well related to the short exact 
sequence~\eqref{Zext}, in that the basis
for $B_n$
given by
$\left\{ Sq(R) \mid r_3 = 0\right\}$ 
passes to a basis for the degree $n$ component of 
$\cB(2) = \cA(2)\modmod E[Q_2]$, while
the basis
$\left\{ Sq(R) \mid r_3 = 1\right\}$ for
$B_{n-7}Q_2$ goes to $0$ there and forms a basis for $\Img{Q_2}$
in degree $n$.

The action of $Sq^i$ on $A$, for $0 \leq i \leq 7$, is given by
the $\cA(2)$-action. These are stored as matrices with entries in
$\F_2$, with the matrix $Sq(i,n)$ giving the left action of $Sq^i$ on $A_n$.
(In reading the MAGMA code in the appendices, bear in mind that MAGMA
works with right actions, so that $Sq^a Sq^b : A_n \longto A_{n+a+b}$
will be written {\tt Sq(b,n) * Sq(a,n+b)}.)

An action of $Sq^8$ is a series of
linear transformations
\[
A_n = B_n \oplus B_{n-7} Q_2 \longto
A_{n+8} = B_{n+8} \oplus B_{n+1} Q_2.
\]
By our assumptions \ref{defsym}.(1) and \ref{defsym}.(2) 
in the definition of a symmetric $\cA$-action,
\[
Sq^8(Sq(R)Q_2)) = (Sq^8Sq(R))Q_2,
\]
so that these linear transformations  have the block form
\[
\left(
\begin{array}{ll}
M_n & N_n \\
0 & M_{n-7} \\
\end{array}
\right).
\]
Note, in particular, that $(Sq^8Sq(R))$ here denotes the (exotic)
action of $Sq^8$ on $\cB(2)$, {\em not} the usual product in $\cA$.
The linear transformations $M_n : B_n \longto B_{n+8}$ for
$0 \leq n \leq 15$ require 28 entries, which we initially
represent by indeterminates $a_1, \ldots, a_{28}$.   Similarly,
the $N_n : B_n \longto B_{n+1}$ require another 66 entries,
which we initially set equal to indeterminates $b_1, \ldots, b_{66}$.
See Appendix~\ref{sq8}.

Entirely analogously, a $Sq^{16}$ action
is a series of linear transformations
\[
A_n = B_n \oplus B_{n-7} Q_2 \longto
A_{n+16} = B_{n+16} \oplus B_{n+9} Q_2.
\]
Again,
these have the block form
\[
\left(
\begin{array}{ll}
K_n & L_n \\
0 & K_{n-7} \\
\end{array}
\right)
\]
There is only one linear transformation $K_n : B_n \longto B_{n+16}$, for
$n=0$, which is initially equal to the indeterminate $c_1$.
Similarly,
the $L_n : B_n \longto B_{n+9}$ require another 24 entries,
which are initially equal to indeterminates $d_1, \ldots, d_{24}$.
See Appendix~\ref{sq16}.

We therefore work over the polynomial ring
\[
R = \F_2[a_1,\ldots,a_{28},b_1,\ldots,b_{66},c_1,
d_1, \ldots, d_{24}].
\]

\subsection{First step}

We initially define only the $Sq^i$, $0\leq i \leq 7$ and
$Sq^8$.   We then use the Adem relations for
$Sq^1 Sq^{2n}$, $Sq^2 Sq^{4n}$, and $Sq^4 Sq^{8n}$ to
compute the matrices with entries in $R$ which describe
the resulting $Sq^i$, for $0 \leq i \leq 15$.

We then compute all the Adem relations for $Sq^a Sq^b$,
$a < 2b$, with $1 \leq a,b \leq 15$ for which the right
hand side is available (i.e., does not contain a $Sq^i$
with $i > 15$).   From this a set of 496
relations, of which 452 are linear,
is obtained.   Rather than process these
directly, we extract those which are of degree 1 and
compute a Gr\"obner basis for the ideal {\tt Rel1}
which they define. (This is just row reduction of linear equations.)
This ideal has $81$ generators, and allows us
to rewrite 81 of the first 94 variables in terms of
the other 13.

We then replace the entries in the matrices for $Sq^i$,
$i \leq 15$, by their normal forms with respect to the
ideal {\tt Rel1}, thereby reducing the number of variables 
involved from 94 to 13, the twelve $a_i$ and $b_i$ listed below
together with $b_{51}$, which will be eliminated in the next step.

\subsection{Second step}

We next define $Sq^{16}$ in terms of the 25 variables
$c_1, d_1, \ldots, d_{24}$ as above (see Appendix~\ref{sq16}),
and again use the Adem relations for
$Sq^1 Sq^{2n}$, $Sq^2 Sq^{4n}$, and $Sq^4 Sq^{8n}$ to
compute the matrices with entries in $R$ which describe
the resulting $Sq^i$, for $0 \leq i \leq 23$.

We then compute all Adem relations which can act nontrivially
on $\cA(2)$.  The resulting 95 relations are used to generate an
ideal {\tt NewRel}.   Note that it involves $c_1$, the $d_i$
and 13 of the
$a_i$ and $b_i$.  Again, we extract the relations of degree 1
(there are 50 of them)
and compute a Gr\"obner basis for the ideal {\tt NewRel1} which
they generate. It has 19 generators, which allows us to rewrite
$b_{51}$ together with 
18 of the $c_1$, $d_1$, \ldots, $d_{24}$ in terms of the remaining seven.

We have then reduced the number of variables to $12 + 7 = 19$,
namely
\[
a_1,a_2,a_{13},a_{14},a_{23},b_1,b_9,b_{10},b_{25},b_{26},b_{27},b_{52},
   c_1,d_1,d_2,d_3,d_{13},d_{14},d_{21}.
\]

We then compute the normal forms of the 95 relations generating
the full ideal {\tt NewRel} of relations with respect to the ideal
generated by the linear relations.   This leaves us with
19 relations among the 19 variables listed above.  These relations
are considered in the next section.

\subsection{Third step}
\label{sub:thirdstep}

We can now simplify the problem by working in the ring 
\[
S = \F_2[
a_1,a_2,a_{13},a_{14},a_{23},b_1,b_9,b_{10},b_{25},b_{26},b_{27},b_{52},
   c_1,d_1,d_2,d_3,d_{13},d_{14},d_{21}]
\]
generated by the 19 variables above with the following relations
(those which were found at the end of the last section).

\begin{longtable}{>{$}l<{$} >{$}l<{$}}
  &    d_{14} + d_{13} + d_{3} + d_{2} + d_{1} + c_{1} + b_{1} a_{13} + b_{1} + a_{13} a_{1} + a_{2} + a_{1} + 1, \\
  &    b_{27} a_{13} + b_{27} + b_{26} + b_{25} a_{13} + b_{25} + b_{10} + b_{9} a_{13} + b_{9} + b_{1} a_{13} + a_{23} a_{13} + a_{13}^2 + a_{13} a_{2} + a_{13} a_{1}, \\
  &    b_{27} a_{2} + b_{10} + b_{1} + 1, \\
  &    b_{52} + b_{27} + b_{26} + b_{9} + b_{1} a_{2} + b_{1} + 1, \\
  &    d_{13} + d_{2} + c_{1} + b_{27} a_{2} + b_{27} + b_{25} a_{2} + b_{25} + b_{9} a_{2} + b_{9} + b_{1} a_{13} + b_{1} a_{2} + a_{23} a_{2} + a_{23} + a_{13} a_{2} + a_{13} a_{1} + \\
  & \phantom{xxxxx} a_{13} + a_{2}^2 + a_{2} a_{1} + a_{2} + a_{1} + 1, \\
  &    b_{26} a_{2} + b_{25} a_{2} + b_{1} a_{2} + b_{1} + a_{23} a_{2} + a_{13} a_{1} + a_{2}, \\
  &    b_{27} a_{2} + b_{9} + b_{1} a_{2} + b_{1} + a_{14} + a_{2} + a_{1} + 1, \\
  &    d_{14} + d_{3} + d_{1} + b_{27} a_{2} + b_{27} + b_{25} a_{2} + b_{25} + b_{9} a_{2} + b_{9} + b_{1} a_{2} + b_{1} + a_{23} a_{2} + a_{23} + a_{13} a_{2} + a_{13} + a_{2}^2 + a_{2} a_{1}, \\
  &    d_{13} + d_{2} + c_{1} + b_{27} a_{2} + b_{27} + b_{25} a_{2} + b_{25} + b_{9} a_{2} + b_{9} + b_{1} a_{13} + b_{1} a_{2} + a_{23} + a_{14} + a_{13} a_{2} + a_{2}^2 + a_{2} a_{1} + a_{2} + 1, \\
  &    b_{26} + b_{25} + b_{10} + b_{9} a_{13} + a_{23} + a_{14} + a_{13} + 1, \\
  &    b_{27} a_{2} + b_{9} + b_{1} a_{2} + b_{1} + a_{23} a_{2} + a_{13} a_{1} + a_{13} + a_{2} + 1, \\
  &    d_{21} + d_{14} + d_{13} + d_{3} + d_{2} + b_{10} + b_{9} + b_{1} a_{13} + a_{14}, \\
  &    b_{27} a_{13} + b_{27} + b_{25} a_{13} + b_{9} + b_{1} a_{13} + a_{23} a_{13} + a_{23} + a_{14} + a_{13}^2 + a_{13} a_{2} + a_{13} a_{1} + a_{13} + 1, \\
  &    b_{52} + b_{27} + b_{25} + b_{9} a_{13} + b_{1} + a_{23} + a_{13} + a_{2} + a_{1}, \\
  &    b_{10} + b_{9} + b_{1} a_{2} + a_{14} + a_{2} + a_{1}, \\
  &    b_{52} + b_{26} a_{13} + b_{26} + b_{25} + b_{9} + b_{1} + a_{23} + a_{13} + a_{1}, \\
  &    b_{26} a_{2} + b_{25} a_{2} + b_{10} + b_{9} + b_{1} + a_{13}, \\
  &    a_{23} a_{2} + a_{14} + a_{13} a_{1} + a_{13} + a_{1}, \\
  &    b_{27} a_{2} + b_{26} a_{2} + b_{25} a_{2} + b_{9} + a_{13} + 1. \\
\end{longtable}

Before using these relations, we can make one more useful simplification.    
Since the variables
will be taking only the values $0$ and $1$, we can reduce modulo
the ideal generated by the $x^2+x$.   This simplifies the relations
further by replacing 5 squares $x^2$ by first powers $x$, in three cases
cancelling the sum $x^2 + x$.

We will continue to reduce modulo this ideal in this next step, since the 
process of eliminating variables which we undertake will introduce
higher degree terms.

The simplified relations are the following:

\begin{align*}
r_{1}  & =    a_{1} a_{13} + a_{1} + a_{2} + a_{13} b_{1} + b_{1} + c_{1} + d_{1} + d_{2} + d_{3} + d_{13} + d_{14} + 1, \\
r_{2} & =    a_{1} a_{13} + a_{2} a_{13} + a_{13} a_{23} + a_{13} b_{1} + a_{13} b_{9} + a_{13} b_{25} + a_{13} b_{27} + a_{13} + b_{9} + b_{10} + b_{25} + b_{26} + b_{27}, \\
r_{3} & =    a_{2} b_{27} + b_{1} + b_{10} + 1, \\
r_{4} & =    a_{2} b_{1} + b_{1} + b_{9} + b_{26} + b_{27} + b_{52} + 1, \\
r_{5} & =    a_{1} a_{2} + a_{1} a_{13} + a_{1} + a_{2} a_{13} + a_{2} a_{23} + a_{2} b_{1} + a_{2} b_{9} + a_{2} b_{25} + a_{2} b_{27} + a_{13} b_{1} + a_{13} + a_{23} + \\
       & \phantom{xx}  b_{9} + b_{25} + b_{27} + c_{1} + d_{2} + d_{13} + 1, \\
r_{6} & =    a_{1} a_{13} + a_{2} a_{23} + a_{2} b_{1} + a_{2} b_{25} + a_{2} b_{26} + a_{2} + b_{1}, \\
r_{7} & =    a_{1} + a_{2} b_{1} + a_{2} b_{27} + a_{2} + a_{14} + b_{1} + b_{9} + 1, \\
r_{8} & =    a_{1} a_{2} + a_{2} a_{13} + a_{2} a_{23} + a_{2} b_{1} + a_{2} b_{9} + a_{2} b_{25} + a_{2} b_{27} + a_{2} + a_{13} + a_{23} + b_{1} + b_{9} + b_{25} + b_{27} + \\
       &  \phantom{xx}    d_{1} + d_{3} + d_{14}, \\
r_{9} & =    a_{1} a_{2} + a_{2} a_{13} + a_{2} b_{1} + a_{2} b_{9} + a_{2} b_{25} + a_{2} b_{27} + a_{13} b_{1} + a_{14} + a_{23} + b_{9} + b_{25} + b_{27} + c_{1} + \\
       &   \phantom{xx}    d_{2} + d_{13} + 1, \\
r_{10} & =    a_{13} b_{9} + a_{13} + a_{14} + a_{23} + b_{10} + b_{25} + b_{26} + 1, \\
r_{11} & =    a_{1} a_{13} + a_{2} a_{23} + a_{2} b_{1} + a_{2} b_{27} + a_{2} + a_{13} + b_{1} + b_{9} + 1, \\
r_{12} & =    a_{13} b_{1} + a_{14} + b_{9} + b_{10} + d_{2} + d_{3} + d_{13} + d_{14} + d_{21}, \\
r_{13} & =    a_{1} a_{13} + a_{2} a_{13} + a_{13} a_{23} + a_{13} b_{1} + a_{13} b_{25} + a_{13} b_{27} + a_{14} + a_{23} + b_{9} + b_{27} + 1, \\
r_{14} & =    a_{1} + a_{2} + a_{13} b_{9} + a_{13} + a_{23} + b_{1} + b_{25} + b_{27} + b_{52}, \\
r_{15} & =    a_{1} + a_{2} b_{1} + a_{2} + a_{14} + b_{9} + b_{10}, \\
r_{16} & =    a_{1} + a_{13} b_{26} + a_{13} + a_{23} + b_{1} + b_{9} + b_{25} + b_{26} + b_{52}, \\
r_{17} & =    a_{2} b_{25} + a_{2} b_{26} + a_{13} + b_{1} + b_{9} + b_{10}, \\
r_{18} & =    a_{1} a_{13} + a_{1} + a_{2} a_{23} + a_{13} + a_{14}, \\
r_{19} & =    a_{2} b_{25} + a_{2} b_{26} + a_{2} b_{27} + a_{13} + b_{9} + 1 \\
\end{align*}

The variable ordering has been reversed in the passage from $R$ to $S$,
changing the superficial appearance of these relations.

These relations can be used to further reduce the number of variables as follows.
We will consider the variables in reverse order, $d_{21}$, $d_{14}$, \ldots,
$a_1$.   If a variable under consideration
occurs in one of the $r_i$ as a term, but not as a factor
in a product, then the effect of setting $r_i$ to $0$ can be accomplished by
replacing the variable in question by its sum with $r_i$ in all the relations.
This will replace the relation $r_i$ by $0$, and will eliminate the variable in question from
all the relations.  Repeating this process for as long as possible will reduce both
the number of variables and the number of relations.

Note that, by proceeding through the variables in reverse order,
each expression replacing a variable
will contain only  previously considered
variables, later in the ordering of variables,
which are not going to be eliminated, and
earlier variables which have yet to be considered.
This ensures that the resulting smaller number of variables
and relations will still have exactly the same solution set.

To determine the possible reductions, we execute the MAGMA command
\begin{verbatim}
[<j, [<i,safe(S.j,rels[i]),rels[i]>
       : i in [1..#rels] | S.j in Terms(rels[i])]>
       : j in [1..#vars]];
\end{verbatim}
This displays, for each variable {\tt S.j}, those relations {\tt rels[i]}
in which {\tt S.j} appears just once as a term (this is checked by the
function {\tt safe}), together with the relation {\tt rels[i]}.   This allows
us to choose the simplest relation among those which could be used
to eliminate the variable {\tt S.j}.
See Appendix~\ref{codesym3}.

Precisely, we make the following reductions.
Relation $r_{12}$ allows us to replace $d_{21}$ by
\[
d_{21} \longto d_{21} + r_{12} =
   a_{13} b_{1} + a_{14} + b_{9} + b_{10} + d_{2} + d_{3} + d_{13} + d_{14}.
\]
This has little effect, since $d_{21}$ does not appear in any other relation.

Both relations $r_1$ and $r_8$ could be used to eliminate $d_{14}$.   We choose to use the shorter
relation $r_1$:
\[
d_{14} \longto d_{14} + r_{1} =
   a_{1} a_{13} + a_{1} + a_{2} + a_{13} b_{1} + b_{1} + c_{1} + d_{1} + d_{2} + d_{3} + d_{13}  + 1.
\]
This has the effect of replacing $r_1$ by $0$, and making $r_5 = r_8$.

The relation $r_9$ is now the shortest relation allowing us to eliminate
$d_{13}$.  This gives
\begin{align*}
d_{13} \longto \,\, & d_{13} + r_9 =  \\
   & a_{1} a_{2} + a_{2} a_{13} + a_{2} b_{1} + a_{2} b_{9} + a_{2} b_{25} + a_{2} b_{27} + a_{13} b_{1} + a_{14} + a_{23} + b_{9} + b_{25} + b_{27} + c_{1} + 
        d_{2} + 1. \\
\end{align*}

Note that this has also affected the values that $d_{21}$ and $d_{14}$ reduce to:
\begin{align*}
d_{21} \longto \,\, & 
a_{1} a_{13} + a_{1} + a_{2} + a_{14} + b_{1} + b_{9} + b_{10} + c_{1} + d_{1} + 1 \\
d_{14} \longto \,\, &
a_{1} a_{2} + a_{1} a_{13} + a_{1} + a_{2} a_{13} + a_{2} b_{1} + a_{2} b_{9} + a_{2} b_{25} + a_{2} b_{27} \\
& {\phantom{xxx}} + a_{2} + a_{14} + a_{23} + b_{1} + b_{9} + b_{25} + b_{27} + d_{1} + d_{3}
\end{align*}
Note also that the relations continue to change as these substitutions are made.   The reader
should run the MAGMA code, or carry out these substitutions by hand in order to follow the
calculation.  (The MAGMA code can be found in Appendix~\ref{codesym3} with its
output in Appendix~\ref{thirdstep}.)

The next variable which can be replaced is $b_{52}$.  The relation $r_4$ is
the shortest among those which could be used ($r_4$, $r_{14}$ and $r_{16})$.
We get
\[
b_{52} \longto \,\,  b_{52} + r_4 =  
    a_{2} b_{1} + b_{1} + b_{9} + b_{26} + b_{27} + 1
\]

Next, $b_{27}$ can be eliminated using $r_{16}$:
\[
b_{27} \longto \,\,  b_{27} + r_{16} =   a_{1} + a_{2} b_{1} + a_{13} b_{26} + a_{13} + a_{23} + b_{25} + 1
\]

Next, $b_{26}$ can be eliminated using $r_{10}$:
\[
b_{26} \longto \,\,  b_{26} + r_{10} =
a_{13} b_{9} + a_{13} + a_{14} + a_{23} + b_{10} + b_{25} +  1
\]

Next, $b_{25}$ can be eliminated using $r_{2}$, which is now equal to $r_{13}$:
\[
b_{25} \longto \,\,  b_{25} + r_{2} =
a_{1} + a_{2} a_{13} b_{1} + a_{2} a_{13} + a_{2} b_{1} + a_{13} b_{1} + a_{13} + a_{14} + b_{9}
\]

Next, $b_{10}$ can be eliminated using $r_{14}$, which is now equal to $r_{15}$:
\[
b_{10} \longto \,\,  b_{10} + r_{14} =
 a_{1} + a_{2} b_{1} + a_{2} + a_{14} + b_{9} 
\]

Next, $b_{9}$ can be eliminated using $r_{19}$:
\[
b_{9} \longto \,\,  b_{9} + r_{19} =
a_{1} a_{2} + a_{2} a_{13} a_{14} + a_{2} a_{13} a_{23} + a_{2} a_{13} + a_{2} a_{14} + a_{2} b_{1} + a_{2} + a_{13} + 1 
\]

Next, $a_{14}$ can be eliminated using $r_{5}$:
\[
a_{14} \longto \,\,  a_{14} + r_{5} =
 a_{1} a_{13} + a_{1} + a_{2} a_{23} + a_{13} 
\]

This leaves 9 independent variables, $a_1$, $a_2$, $a_{13}$, $a_{23}$,
$b_1$, $c_1$, $d_1$, $d_2$, and $d_3$.
The remaining variables are given in terms of these by
\begin{align*}
     a_{14} &= a_{1} a_{13} + a_{1} + a_{2} a_{23} + a_{13} ,\\
     b_{9} &= a_{1} a_{2} a_{13} + a_{2} a_{13} + a_{2} a_{23} + a_{2} b_{1} + a_{2} + a_{13} + 1 ,\\
     b_{10} &= a_{1} a_{2} a_{13} + a_{1} a_{13} + a_{2} a_{13} + 1 ,\\
     b_{25} &= a_{1} a_{2} a_{13} + a_{1} a_{13} + a_{2} a_{13} b_{1} + a_{2} + a_{13} b_{1} + a_{13} + 1 ,\\
     b_{26} &= a_{1} a_{2} a_{13} + a_{1} a_{13} + a_{1} + a_{2} a_{13} a_{23} + a_{2} a_{13} + a_{2} a_{23} +
           a_{2} + a_{13} b_{1} + a_{13} + a_{23} + 1 ,\\
     b_{27} &= a_{1} a_{13} + a_{1} + a_{2} a_{13} b_{1} + a_{2} b_{1} + a_{2} + a_{13} a_{23} + a_{23} ,\\
     b_{52} &= a_{2} a_{13} a_{23} + a_{2} a_{13} b_{1} + a_{2} b_{1} + a_{2} + a_{13} a_{23} + a_{13} b_{1} + b_{1} + 1 ,\\
     d_{13} &= a_{1} a_{13} + a_{2} a_{13} a_{23} + a_{2} a_{13} b_{1} + a_{2} a_{13} + a_{2} b_{1} +
            a_{13} a_{23} + a_{13} + c_{1} + d_{2} + 1 ,\\
     d_{14} &= a_{1} + a_{2} a_{13} a_{23} + a_{2} a_{13} b_{1} + a_{2} a_{13} + a_{2} b_{1} + a_{2} +
            a_{13} a_{23} + a_{13} b_{1} + a_{13} + b_{1} + d_{1} + d_{3} ,\\
     d_{21} &= a_{1} a_{13} + a_{2} b_{1} + b_{1} + c_{1} + d_{1} + 1.
\end{align*}
The full reduction of all $119$ variables to these 9 variables 
is in Appendix~\ref{symrelations}.

This process has eliminated ten relations and caused three more
to become equal to others, as noted above. The remaining six
have now become equal, leaving us with a single relation 
\[
    a_{1} a_{2} + a_{1} a_{13} + a_{2} a_{13} a_{23} + a_{2} a_{13}
    b_{1} + a_{2} a_{13} + a_{2} a_{23} + a_{2} b_{1} + a_{2} +
    b_{1} = 0
\]

The second and third statements in the theorem follow from the role of the remaining
nine variables in defining $Sq^8$ and $Sq^{16}$.

This completes the proof of Theorem~\ref{thm:symmetric}.
\end{proof}

\begin{remark}
The actions of $Sq^8$ and $Sq^{16}$
which are not given in the theorem can be determined by using the
initial generic form of the action in Appendix~\ref{sq8}, e.g.,
\[
Sq^8\cdot Sq{(5,1)} = a_{27} Sq{(7,3)} + b_{34} Sq{(6,1,1)}
                     + b_{35} Sq{(3,2,1)} + b_{36} Sq{(0,3,1)}
\]
and the relations in Appendix~\ref{symrelations} to write them in terms of the
final $9$ variables, e.g.,
\[
Sq^8\cdot Sq{(5,1)} = a_{13} Sq{(7,3)} +  Sq{(6,1,1)} +
                       Sq{(3,2,1)}.  \phantom{ + b_{36} Sq{(0,3,1)}}
\]
If you are running the MAGMA code, you can retrieve these quickly, once the final
version of the squaring operations has been computed, by the 
command {\tt Sq(8,8)} to get the full matrix for
$Sq^8 : A_8 \longto A_{16}$:
\begin{verbatim}

> A_bas(8);
[
    [ 5, 1, 0 ],
    [ 2, 2, 0 ],
    [ 1, 0, 1 ]
]
> A_bas(16);
[
    [ 7, 3, 0 ],
    [ 6, 1, 1 ],
    [ 3, 2, 1 ],
    [ 0, 3, 1 ]
]
> Sq(8,8);
[      a13         1         1         0]
[      a23 b26 + b25       b27   b27 + 1]
[        0         1        a2         0]
\end{verbatim}

The first row of the matrix gives the value of $Sq^8 \cdot Sq{(5,1)}$ since $Sq{(5,1)}$ is
the first basis element in degree~8, as exhibited above.
\end{remark}

\section{The general case}
\label{sec:general}

Turning to the general case, we follow the same strategy.  There are $124$
indeterminates required to define the $Sq^8$ action and another $26$
needed for $Sq^{16}$.  We use certain Adem relations to define all $Sq^i$
in terms of these,
then compute the ideal defining the subvariety of $\F_2^{150}$ containing
all possible $\cA$ actions.

\begin{theorem}\leavevmode
\label{thm:general}
\begin{enumerate}
\item
The variety of $\cA$-module actions on $\cA(2)$ is defined by
\[
\F_2[a_1, a_2, a_{3}, a_{21}, a_{47}, a_{48}, 
a_{60}, a_{61}, a_{62}]/I
  \otimes \F_2[b_1,b_2,b_3,b_4]
\]
where $I$ is the ideal with generators
    \begin{multline*}
    a_{1} a_{21} + a_{1} a_{60} + a_{1} a_{62} + a_{2} a_{3} a_{21}
    a_{60} a_{62} + a_{2} a_{3} a_{21} a_{60} + a_{2} a_{3} a_{21} +
	a_{2} a_{3} a_{60} + a_{2} a_{3} a_{62} + a_{2} a_{3} +\\
	a_{2} a_{21} a_{47} a_{60} + a_{2} a_{21} a_{47} a_{62} +
	a_{2} a_{21} a_{47} + a_{2} a_{21} a_{48} a_{60} + a_{2}
	a_{21} a_{48} a_{62} + a_{2} a_{21} a_{60} + a_{2} a_{48} + \\
	a_{2} + a_{3} a_{21} a_{60} a_{62} + a_{3} a_{21} + a_{3}
	a_{60} a_{62} + a_{3} a_{60} + a_{21} a_{47} a_{60} + a_{21}
	a_{47} a_{62} + a_{21} a_{48} a_{60} + \\
	a_{21} a_{48} a_{62} + 
	a_{21} a_{62} + a_{21} + a_{47} a_{60} + a_{47} a_{62} +
	a_{48} a_{60} + a_{48} a_{62} + a_{48} + a_{62} + 1,
\end{multline*}
\begin{multline*}
    a_{1} a_{2} a_{21} a_{60} + a_{1} a_{2} a_{60} a_{62} + a_{1}
    a_{21} + a_{1} a_{60} a_{62} + a_{1} a_{60} + a_{1} a_{62} +
	a_{2} a_{3} a_{21} a_{60} a_{62} + a_{2} a_{3} a_{60} a_{62} + \\
	a_{2} a_{21} a_{47} a_{60} + 
	a_{2} a_{21} a_{48} a_{60}
	+ a_{2} a_{21} a_{60} a_{62} + a_{2} a_{47} a_{60} a_{62}
	+ a_{2} a_{47} a_{60} + a_{2} a_{47} + a_{2} a_{48} a_{60} a_{62} + \\
	 a_{2} a_{61} + a_{3} a_{21} a_{60} a_{62} + 
	a_{21} a_{47} a_{60} + a_{21} a_{48} a_{60} + a_{21} a_{60} a_{62}
	+ a_{21} a_{61} + \\
	a_{47} a_{60} a_{62} + a_{47} a_{60} +  
	a_{47} a_{62} + a_{48} a_{60} a_{62} + a_{48} a_{60} + 
	a_{60} a_{61} + a_{60} a_{62},
\end{multline*}
and
\begin{multline*}
    a_{1} a_{2} a_{21} a_{60} + a_{1} a_{2} a_{21} a_{62} + a_{1}
    a_{2} a_{60} + a_{1} a_{2} a_{62} + a_{1} a_{2} + a_{1} a_{21} + 
	a_{2} a_{3} a_{21} a_{60} + \\
	a_{2} a_{3} a_{21} + 
	a_{2} a_{3} a_{60} a_{62} +
	a_{2} a_{21} a_{48} + a_{2} a_{21} a_{62}
	+ a_{2} a_{21} + a_{2} a_{48} a_{60} + \\
	a_{2} a_{48} a_{62} + 
	a_{2} a_{60} a_{62} + a_{2} a_{62} + a_{2} + a_{3} a_{60}
	+ a_{3} a_{62} + a_{3}.
\end{multline*}
The first factor defines the $Sq^8$ action and has $100$ $\F_2$ points.
The second factor defines the $Sq^{16}$ action and has $16$ $\F_2$ points.
This gives $1600$ ways to define an $\cA$-module structure on $\cA(2)$.
\item
The $Sq^8$ actions are described in Appendices~\ref{gensq8} and
\ref{genrelations}.  The coordinates $a_1$, $a_2$, $a_3$, $a_{21}$,
$a_{47}$, $a_{48}$, $a_{60}$, $a_{61}$, and $a_{62}$
determine, and are determined by, the following:
\begin{itemize}
\item
$ Sq^8 \cdot 1  = a_1 Sq{(5,1)} + a_2 Sq{(2,2)} + a_3 Sq{(1,0,1)},$
\item
$a_{21}$ is the coefficient of $Sq{(6,2)}$ in $Sq^8 \cdot Sq^4$, 
\item
$a_{47}$ is the coefficient of $Sq{(5,3)}$ in $Sq^8 \cdot Sq{(0,2)}$,
\item
$a_{48}$ is the coefficient of $Sq{(7,0,1)}$ in $Sq^8 \cdot Sq{(0,2)}$,
\item
$ Sq^8 \cdot Q_2  = a_{60} Sq{(6,3)} + a_{61} Sq{(5,1,1)} + a_{62} Sq{(2,2,1)},$
\end{itemize}
\item
The $Sq^{16}$ actions are described in Appendices~\ref{gensq16} and
\ref{genrelations}.
The coordinates $b_1$, $b_2$, $b_3$ and $b_4$ determine, and are
determined by
\[
Sq^{16} \cdot 1 = b_1 Sq(7,3) + b_2 Sq(6,1,1) 
                + b_3 Sq(3,2,1) + b_4 Sq(0,3,1).
\]
\end{enumerate}
\end{theorem}

This reproduces, in quite different form, the results of Marilyn  Roth
in~\cite{Roth}.

% There are several other equivalent ways to describe some of these coordinates:
% \begin{align*}
% Sq^8 \cdot Sq{(5,1)} & = a_{21} Sq{(7,3)} + Sq{(6,1,1)} + Sq{(3,2,1)}, \\
% Sq^8 \cdot Sq{(1,0,1)} & = a_{60} Sq{(7,3)} + Sq{(6,1,1)} + a_{62} Sq{(3,2,1)}, \\
% Sq^8 \cdot Sq{(2,0,1)} & = a_{61} Sq{(7,1,1)} + Sq{(4,2,1)} + a_{62} Sq{(1,3,1)}, \\
% Sq^8 \cdot Sq{(0,1,1)} & = a_{62} Sq{(2,3,1)}, \\
% Sq^8 \cdot Sq{(1,1,1)} & = a_{62} Sq{(3,3,1)}. \\
% \end{align*}

\begin{proof}[Proof of Theorem~\ref{thm:general}]
We adopt the same basic strategy in the general case, with the
addition of a couple of steps to handle additional complexity.
The action of $Sq^8$
requires 124 variables $a_1$, \ldots, $a_{124}$, while the action of
$Sq^{16}$ requires $26$ variables $b_1$, \ldots, $b_{26}$.  
The exact role each of these variables plays can be found in
Appendices~\ref{gensq8} and \ref{gensq16}

We use the
same Adem relations to compute $Sq^9$, \ldots, $Sq^{15}$ from $Sq^8$
as in the symmetric case, and then compute all the Adem relations 
involving only $Sq^i$ for $i \leq 15$.  This produces $564$ distinct
relations, of which $519$ are linear.  A Gr\"obner basis (row reduction)
of the linear relations gives $105$ relations which allow us to eliminate all but
$124-105 = 19$ of the $a_i$.   

Before proceeding to the second step, we repeat the first step.  That is,
we use these $105$ relations to redefine
$Sq^8$, \ldots, $Sq^{15}$ in terms of the remaining $19$ variables, and again
compute all the relations between them implied by the Adem relations 
involving only the $Sq^i$ for $i \leq 15$.    This produces $22$ distinct
relations, of which $3$ are linear.   Those allow us to eliminate three more
variables, leaving the $16$ variables
\[ a_{1}, a_{2}, a_{3}, a_{21}, a_{22}, a_{23}, a_{24},
   a_{47},a_{48},a_{49},a_{50},a_{60},a_{61},a_{62},a_{90}, a_{102}.
\]

We again rewrite the
$Sq^i$ for $8 \leq i \leq 15$ in terms of these remaining $16$,
and compute the Adem relations once more.   No new linear relations are found.

We then use the Adem relations to define $Sq^{17}$, \ldots, $Sq^{23}$,
and compute all the relations determined by all Adem relations acting
on $\cA(2)$.  This produces $92$ relations, of which $45$ are linear.
Their Gr\"obner basis gives $18$ relations, allowing us to eliminate 
$18$ of the $26$ $b_i$.  Again, we rewrite the $Sq^i$ using these
$18$ linear relations to eliminate variables, then recompute the Adem
relations.

This leaves us with no new linear relations, and $22$ relations of higher
degree involving the $16$ $a_i$ listed above and the 8 variables
$
    b_{1}, b_{2}, b_{3}, b_{4}, b_{14}, b_{15}, b_{22}, b_{26}.
$

We then define $S$ to be the polynomial ring on these $24$ $a_i$ and
$b_i$, and consider the $22$ nonlinear relations they must satisfy.

As in the symmetric case, we reduce these relations modulo the ideal
generated by the $x^2+x$.  We then take the `third step', as in the symmetric
case, working our way through the variables in reverse order, from $b_{26}$ to
$a_1$, seeking relations which allow us to eliminate that variable and
rewrite the remaining relations.  

At the start, the relations are

\begin{longtable}{>{$}r<{$} >{$}c<{$}  >{$}l<{$}}
r_{1} &=&     a_{1} a_{21} + a_{1} + a_{2} a_{47} + a_{3} a_{60} + a_{21} + a_{22}, \\
r_{2} &=&     a_{3} a_{62} + a_{3} + a_{23} + a_{49} + a_{50} + a_{90} + 1, \\
r_{3} &=&     a_{2} a_{48} + a_{2} a_{49} + a_{3} + a_{21} + a_{23} + a_{24}, \\
r_{4} &=&     a_{1} a_{2} + a_{2} a_{3} + a_{2} a_{21} + a_{2} a_{23} + a_{2} a_{47} + a_{2} a_{48} + a_{2} a_{50} + a_{2} + a_{3} + a_{21} +  \\
  & &  \phantom{xx}          a_{23} + a_{47} + a_{48} + a_{50} + b_{2} + b_{4} + b_{15}, \\
r_{5} &=&     a_{1} a_{21} + a_{2} a_{21} + a_{3} a_{21} + a_{21} a_{23} + a_{21} a_{47} + a_{21} a_{48} + a_{21} a_{50} + a_{21} + a_{23} + \\
  & &  \phantom{xx}           a_{24} + a_{48} + a_{49} + a_{50}, \\
r_{6} &=&     a_{1} + a_{2} a_{23} + a_{2} + a_{3} + a_{21} a_{23} + a_{21} + a_{23} a_{60} + a_{23} a_{62} + a_{47} + a_{48} + a_{50}  + a_{90}, \\
r_{7} &=&     a_{2} a_{50} + a_{3} a_{62} + a_{3} + a_{23} + a_{61} + a_{62} + a_{102} + 1, \\
r_{8} &=&     a_{1} a_{21} + a_{2} a_{47} + a_{2} a_{48} + a_{2} a_{49} + a_{2} + a_{3} a_{60} + a_{3} a_{62} + a_{3}, \\
r_{9} &=&     a_{1} a_{21} + a_{1} a_{60} + a_{1} a_{62} + a_{1} + a_{2} a_{21} + a_{2} a_{23} + a_{2} a_{47} + a_{2} a_{48} + a_{2} a_{50} + \\
  & &  \phantom{xx}   a_{3} a_{21} + a_{3} a_{60} + a_{3} a_{62} + a_{21} + a_{23} + a_{47} + a_{48} + a_{50} + b_{1} + b_{3} + b_{14} + 1, \\
r_{10} &=&     a_{2} a_{48} + a_{2} a_{49} + a_{2} a_{50} + a_{21} + a_{23} + 1, \\
r_{11} &=&     a_{1} a_{60} + a_{1} a_{62} + a_{2} a_{60} + a_{2} a_{61} + a_{2} + a_{3} a_{60} + a_{21} a_{60} + a_{21} a_{61} + a_{21} +  \\
   & &  \phantom{xx}           a_{23} a_{60} + a_{47} a_{60} + a_{47} a_{62} + a_{48} a_{60} + a_{50} a_{60} + a_{60} a_{61} + a_{60} a_{62} +  \\
   & &  \phantom{xx}          a_{60} + a_{61} + a_{62} + a_{102}, \\
r_{12} &=&     a_{1} a_{21} + a_{2} a_{47} + a_{2} a_{50} + a_{2} + a_{3} a_{60} + a_{3} a_{62} + a_{3} + a_{21} + a_{23} + 1, \\
r_{13} &=&     a_{1} a_{21} + a_{2} a_{21} + a_{2} a_{23} + a_{3} a_{21} + a_{21} a_{47} + a_{21} a_{48} + a_{21} a_{50} + a_{22} +  \\
   & &  \phantom{xx}         a_{23} a_{60} + a_{23} a_{62} + a_{23} + a_{47} + a_{50} + 1, \\
r_{14} &=&     a_{1} + a_{2} a_{49} + a_{2} + a_{3} + a_{21} a_{49} + a_{21} + a_{23} + a_{47} + a_{48} + a_{49} a_{60} + a_{49} a_{62}  + \\
   & &  \phantom{xx}          a_{49} + a_{62} + a_{90}, \\
r_{15} &=&     a_{2} a_{50} + a_{3} + a_{24} + 1, \\
r_{16} &=&     a_{1} + a_{2} + a_{3} a_{62} + a_{22} + a_{23} + a_{24}, \\
r_{17} &=&     a_{2} a_{23} + a_{21} a_{23} + a_{21} + a_{22} + a_{23} a_{60} + a_{23} a_{62} + a_{24} + a_{47} + a_{48} + a_{49} + 1, \\
r_{18} &=&     a_{1} + a_{2} a_{50} + a_{2} + a_{3} a_{62} + a_{3} + a_{22} + a_{23} + 1, \\
r_{19} &=&     a_{2} a_{21} + a_{2} a_{23} + a_{2} a_{48} + a_{2} a_{50} + a_{2} a_{60} + a_{2} a_{61} + a_{3} a_{21} + a_{3} a_{60} +  \\
   & &  \phantom{xx}         a_{3} a_{62} + a_{22} + a_{23} + a_{47} + a_{48} + a_{50} + b_{3} + b_{14} + b_{26} + 1, \\
r_{20} &=&     a_{1} a_{21} + a_{1} + a_{2} a_{21} + a_{2} a_{23} + a_{2} + a_{3} a_{21} + a_{21} a_{47} + a_{21} a_{48} + a_{21} a_{50} +  \\
   & &  \phantom{xx}         a_{23} a_{60} + a_{23} a_{62} + a_{23} + a_{47} + a_{50} + a_{61} + a_{62} + a_{102} + 1, \\
r_{21} &=&     a_{2} a_{3} + a_{3} a_{21} + a_{3} a_{60} + a_{3} a_{62} + a_{22} + a_{23} + a_{24} + b_{3} + b_{4} + b_{14} + b_{15} + b_{22}, \\
r_{22} &=&     a_{1} a_{2} + a_{1} a_{21} + a_{1} a_{60} + a_{1} a_{62} + a_{1} + a_{2} a_{3} + a_{2} + a_{3} a_{21} + a_{3} a_{60} +  \\
   & &  \phantom{xx}         a_{3} a_{62} + a_{3} + b_{1} + b_{2} + b_{3} + b_{4} + b_{14} + b_{15} + 1 \\
\end{longtable}

Using these, we are able to
\begin{enumerate}
\item
eliminate $b_{26}$ using $r_{19}$,
\item
eliminate $b_{22}$ using $r_{21}$,
\item
eliminate $b_{15}$ using $r_{4}$,
\item
eliminate $b_{14}$ using $r_{9}$,
\item
eliminate $a_{102}$ using $r_{7}$,
\item
eliminate $a_{90}$ using $r_{2}$,
\item
eliminate $a_{50}$ using $r_{14}$,
\item
eliminate $a_{49}$ using $r_{17}$,
\item
eliminate $a_{24}$ using $r_{16}$,
\item
eliminate $a_{23}$ using $r_{15}$, and
\item
eliminate $a_{22}$ using $r_{1}$.
\end{enumerate}

This leaves the $13$ variables and $3$ relations given in the theorem.
Appendix~\ref{genrelations} expresses each of the $150$ variables in
terms of the remaining $13$ variables.

Finally, we manually count the number of $\F_2$ points in the 
variety defining the $Sq^8$ action, giving the result in 
Appendix~\ref{allgenSq8}.  The MAGMA code for these three
steps in the general case can be found in Appendices~\ref{codegen}
and \ref{codegen3}.

\end{proof}

\section{Actions on $\cB(2)$}
\label{sec:Baction}

We next analyze the $\cA$ actions on $\cB(2)$.

\begin{theorem}\leavevmode
\label{thm:Baction}
\begin{enumerate}
\item
The variety of $\cA$-module actions on $\cB(2)$ is defined by
\[
\F_2[a_1, a_2, a_{13}, a_{23}] \otimes \F_2[c_1].
\]
The first factor defines the $Sq^8$ action and has $16$ $\F_2$ points.
The second factor defines the $Sq^{16}$ action and has $2$ $\F_2$ points,
giving $32$ ways to define an $\cA$-module structure on $\cB(2)$.
\item
The coordinates $a_1$, $a_2$, $a_{13}$ and $a_{23}$ determine, and 
are determined by
\begin{itemize}
\item
$ Sq^8 \cdot 1  = a_1 Sq{(5,1)} + a_2 Sq{(2,2)}$,
\item
$a_{13}$ is the coefficient of $Sq{(6,2)}$ in $Sq^8 \cdot Sq^4$,  and
\item
$ Sq^8 \cdot Sq{(2,2)}  = a_{23} Sq{(7,3)}$.
\end{itemize}
\item
The coordinate $c_1$ determines, and is determined by
\[
Sq^{16} \cdot 1 = c_1 Sq{(7,3)}.
\]
\end{enumerate}
\end{theorem}

% There are several other equivalent ways to describe these coordinates:
% \begin{align*}
% Sq^8 \cdot Sq^{5} & = a_{13} Sq{(7,2)}, \\
% Sq^8 \cdot Sq^{6} & = (1 + a_{13}) Sq{(5,3)}, \\
% Sq^8 \cdot Sq{(4,1)} & = a_{13} Sq{(6,3)}, \\
% Sq^8 \cdot Sq{(5,1)} & = a_{13} Sq{(7,3)}, \\[2ex]
% Sq^8 \cdot Sq{(0,2)} & = a_{23} Sq{(5,3)}. \\
% \end{align*}

\begin{corollary}
\label{cor:B2}
Four of the $32$ $\cA$ actions on $\cB(2)$ do not lift to 
symmetric $\cA$ actions
on $\cA(2)$.   These are the actions in which
$(a_1, a_2, a_{13}, a_{23}) = (0,1,0,0)$ or $(1,1,0,1)$.

The $Sq^8$ actions with 
$(a_1, a_2, a_{13}, a_{23}) = (0,1,0,1)$ or $(1,1,0,0)$ have
two symmetric lifts each, with $b_1=0$ or $1$.   The $Sq^8$ actions corresponding to the
remaining $12$ values of $(a_1, a_2, a_{13}, a_{23})$ lift uniquely
to a symmetric  $Sq^8$ action on $\cA(2)$.
\end{corollary}

\begin{proof}[Proof of Theorem~\ref{thm:Baction}]
We carry out the same analysis as in Section~\ref{sec:symmetric}, using only
the subspace $B_n$ in the decomposition~\eqref{split} of $A_n$.
This entails working over the ring
\[
R = \F_2[a_1,\ldots,a_{28},c_1]
\]
and carrying out the same three steps.  The second and
third steps dramatically simplify.   In the second step,
the only possible nonzero
$Sq^{16}$ is $Sq^{16} \cdot 1 = c_1 Sq{(7,3)}$ and it does
not enter into any Adem relations.  
In the third step, there is only one relation, and it allows us to 
eliminate one more variable.

Precisely, we find that the linear relations satisfied by the $Sq^8$ action
allow us to reduce to $a_1$, $a_2$, $a_{13}$, $a_{14}$, $a_{23}$, and $c_1$
with one relation
\[
a_1 + a_{13} + a_{14} + a_1 a_{13} + a_2 a_{23} = 0.
\]
This allows us to write $a_{14}$ in terms of the other four.
All sixteen values of the remaining four variables give $Sq^8$
actions.   
\end{proof}

\begin{proof}[Proof of Corollary~\ref{cor:B2}]
Examining the $Sq^8$ actions on $\cA(2)$ given in Appendix~\ref{allsymSq8},
we see that two of the sixteen possible $Sq^8$ actions on $\cB(2)$
do not lift, namely
$(a_1,a_2,a_{13},a_{23}) = (0,1,0,0)$ and $(1,1,0,1)$.

Further, two values of $(a_1,a_2,a_{13},a_{23})$, namely
$(0,1,0,1)$ and $(1,1,0,0)$, occur twice, once with
$b_1 =  0$ and once with $b_1 = 1$, while the remaining 12 entries occur only once.
\end{proof}

\section{Relations between spaces of $\cA$-module structures}
\label{sec:relations}

We now determine the maps in Diagram~\ref{diagram:D}
between these varieties.
The map $V_{\text{sym}} \longto V_{\text{gen}}$
is easily determined by the roles of the indeterminates in defining the
$Sq^8$ and $Sq^{16}$ actions.   

\begin{theorem}
\label{thm:symtogen}
The inclusion of the variety of symmetric actions into the variety of all actions
is given on coordinates by the homomorphism
\[
\F_2[a_1, a_2, a_{3}, a_{21}, a_{47}, a_{48}, 
a_{60}, a_{61}, a_{62}, b_1,b_2,b_3,b_4]
\longto
\F_2[a_1, a_2, a_{13}, a_{23}, b_1, c_1,d_1,d_2,d_3]
\]
sending
\begin{align*}
a_1 & \mapsto a_1 \\
a_2 & \mapsto a_2 \\
a_{3} & \mapsto b_1 \\
a_{21} & \mapsto a_{13} \\
a_{47} & \mapsto a_{23} \\
a_{48} & \mapsto b_{25} =
a_{1} a_{2} a_{13} + a_{1} a_{13} + 
a_{2} a_{13} b_{1} + a_{2} + a_{13} b_{1} + a_{13} + 1 \\
a_{60} & \mapsto 0 \\
a_{61} & \mapsto a_1 \\
a_{62} & \mapsto a_2 \\
b_1 & \mapsto  c_1 \\
b_2 & \mapsto d_1 \\
b_3 & \mapsto d_2 \\
b_4 & \mapsto d_3. \\
\end{align*}
\end{theorem}

\begin{proof}[Proof of Theorem~\ref{thm:symtogen}]
Comparing parts (2) and (3) of Theorems~\ref{thm:symmetric} and
\ref{thm:general} allows us to identify the images of
all but $a_{48}$, $a_{60}$, $a_{61}$ and $a_{62}$.
Since $a_{48}$ is the coefficient of $Sq{(7,0,1)}$ in
$Sq^8 \cdot Sq{(0,2)}$, we see from Appendix~\ref{sq8}
that it must go to $b_{25}$, and from Appendix~\ref{symrelations}
we see that this equals
\[
a_{1} a_{2} a_{13} + a_{1} a_{13} + 
a_{2} a_{13} b_{1} + a_{2} + a_{13} b_{1} + a_{13} + 1.
\]
Finally, in the symmetric case, 
$Sq^8 \cdot Q_2 = a_1 Sq{(5,1,1)} + a_2 Sq{(2,2,1)}$,
which implies that $a_{60}$ goes to $0$,
$a_{61}$ goes to $a_1$ and
$a_{62}$ goes to $a_2$.
\end{proof}

\begin{remark}
Most of this is evident, e.g. from 
\[
 Sq^8 \cdot 1  = a_1 Sq{(5,1)} + a_2 Sq{(2,2)} + a_3 Sq{(1,0,1)}
\]
in the general case, and
\[
 Sq^8 \cdot 1  = a_1 Sq{(5,1)} + a_2 Sq{(2,2)} + b_1 Sq{(1,0,1)}
\]
in the symmetric case,
or
\[
Sq^{16} \cdot 1 = b_1 Sq(7,3) + b_2 Sq(6,1,1) 
                + b_3 Sq(3,2,1) + b_4 Sq(0,3,1)
\]
in the general case, and
\[
Sq^{16} \cdot 1 = c_1 Sq(7,3) + d_1 Sq(6,1,1) 
                + d_2 Sq(3,2,1) + d_3 Sq(0,3,1)
\]
in the symmetric case.   The actions on $Q_2 = Sq{(0,0,1)}$ are
of special interest.   We have
\[
Sq^8 \cdot Q_2  = a_{60} Sq{(6,3)} + a_{61} Sq{(5,1,1)} + a_{62} Sq{(2,2,1)}
\]
in the general case, and
\[
Sq^8 \cdot Q_2  = \phantom{a_{60} Sq{(6,3)} +} a_{1} Sq{(5,1,1)} + a_{2} Sq{(2,2,1)}
\]
in the symmetric case.   That the coefficient of $Sq{(5,3)}$ must be $0$ follows from the
assumption that $\Img{Q_2}$ is an $\cA$-submodule in the symmetric case.  The
coefficients $a_1$ and $a_2$ of $Sq{(5,1,1)} = Sq{(5,1)} Q_2$ and
$Sq{(2,2,1)} = Sq{(2,2)} Q_2$ follow from the isomorphism~\ref{defsym}.(\ref{defsym2})
in the definition of a symmetric action.
\end{remark}

Next, Theorem~\ref{thm:VQ} shows that $V_{\cQ}$ is simply the intersection of
$V_{\text{gen}}$ with the hyperplane $a_{60} = 0$.  From this, both
horizontal maps in Diagram~\ref{diagram:D}
are clear from Theorem~\ref{thm:symtogen}.

\begin{proposition}
If $A \in V_\cQ$ then there is a short exact sequence of
$\cA$-modules
\[
0 \to \Sigma^7 sA \longto A  \longto qA \to 0
\]
\end{proposition}

\begin{proof}
This is simply a restatement of the definition of $V_\cQ$ and of
$s$ and $q$.
\end{proof}

\begin{theorem}
\label{thm:VQ}
An $\cA$ action on $\cA(2)$ preserves $\Img{Q_2}$, and hence induces
actions making the short exact sequence~\eqref{Qext} one of
$\cA$-modules, iff $a_{60} = 0$.
\end{theorem}

\begin{proof}[Proof of Theorem~\ref{thm:VQ}]
Since $\Img{Q_2}$ is zero below degree $7$,
and $\cB(2)$ is zero above degree $16$, the only $Sq^8$ actions which could send
a class in $\Img{Q_2}$ outside $\Img{Q_2}$ are from degrees $7$ and $8$.   In 
these degrees, the map from $\Img{Q_2}$ to the subspace spanned by $B_n$
is given by $a_{60}$, and $a_{71}$ respectively, by Appendix~\ref{gensq8}.
From Appendix~\ref{genrelations} we see that $a_{71} = a_{60}$, so the
only condition is that $a_{60} = 0$.
\end{proof}

The vertical maps in Diagram~\ref{diagram:D} are $s$, restriction to the
submodule $\Img{Q_2}$, and $q$, passage to the quotient $\cA(2)\modmod E[Q_2]$.

\begin{theorem}\leavevmode
\label{thm:gentoB}
\begin{enumerate}
\item
The projection $q$ from  the variety $V_\cQ$ to the variety of $\cA$-module
structures on $\cB(2)$
is given on coordinates by the homomorphism
\[
\F_2[a_1, a_2, a_{13}, a_{23}, c_1]
\longto
\F_2[a_1, a_2, a_{3}, a_{21}, a_{47}, a_{48}, 
a_{60}, a_{61}, a_{62}, b_1,b_2,b_3,b_4]/(a_{60})
\]
sending
\begin{align*}
a_1 & \mapsto a_1 \\
a_2 & \mapsto a_2 \\
a_{13} & \mapsto a_{21} \\
a_{23} & \mapsto a_{47} \\
c_1 & \mapsto  b_1 \\
\end{align*}
\item
The projection $s$ from  the variety $V_\cQ$ to the variety of $\cA$-module
structures on $\Img{Q_2}$
is given on coordinates by the homomorphism
\[
\F_2[a_1, a_2, a_{13}, a_{23}, c_1]
\longto
\F_2[a_1, a_2, a_{3}, a_{21}, a_{47}, a_{48}, 
a_{60}, a_{61}, a_{62}, b_1,b_2,b_3,b_4]/(a_{60})
\]
sending
\begin{align*}
a_1 & \mapsto a_{61} \\
a_2 & \mapsto a_{62} \\
a_{13} & \mapsto a_2 + a_{21} + a_{62} \\
a_{23} & \mapsto a_1 + a_{47} + a_{61} \\
c_1 & \mapsto  b_1 + a_1 a_{62} + a_2 a_{61}  \\
\end{align*}
\end{enumerate}
\end{theorem}

\begin{remark}
The maps $s$ and $q$ are equalized by $V_{\text{sym}} \longto
V_\cQ$ (as they must be).
\end{remark}

\begin{proof}[Proof of Theorem~\ref{thm:gentoB}]
The projection $q$ to the actions on the quotient $\cB(2)$ can be read off
parts (2) and (3) of Theorems~\ref{thm:general} and \ref{thm:Baction} 
by discarding $\Img{Q_2}$.  

The projection $s$ given by restricting to $\Img{Q_2}$ can be
calculated by multiplying parts (2) and (3) of Theorem~\ref{thm:Baction}
by $Q_2$.   This gives that the images of $a_1$ and $a_2$ are determined by
$Sq^8 \cdot Q_2$ in Theorem~\ref{thm:general}.(2).   Thus $a_1$ and
$a_2$ map to $a_{61}$ and $a_{62}$ respectively.
Similarly, the image of $a_{13}$ is the coefficient of $Sq{(6,2,1)} =
Sq{(6,2)}Q_2$ in $Sq^8\cdot Sq{(4,0,1)} = Sq^8 \cdot Sq^4 Q_2$.   By
Appendix~\ref{gensq16} this is the coordinate $a_{101}$, which equals 
$a_{62} + a_{21} + a_2$, by Appendix~\ref{genrelations}.   (Recall that
$a_{60} = 0 $ in $V_\cQ$.)
Similarly, the image of $a_{23}$ is $a_{124} = a_{61} + a_{47} + a_1$
and the image of $c_1$ is $b_{26} = b_1 + a_1 a_{62} + a_2 a_{61}$.
\end{proof}

\section{Duality}
\label{sec:duality}

We start by giving a proof of the general fact that a connected 
finite dimensional
Hopf algebra is self-dual which exposes the precise form
of the isomorphism.

\begin{proof}[Proof of Proposition~\ref{prop:gendual}]
Since $H$ is a free $H$-module on one generator, any homomorphism
$\theta : H \longto \Sigma^N DH$ is entirely determined by $\theta(1)$.
To be non-zero, $\theta(1)$ must be $s^*$ rather than $0$.
Then $\theta(a) = a\cdot\theta(1) = a\cdot s^*
= s^*\circ L_{\chi(a)}$.   Finally, $\theta$ must be an isomorphism because
$\theta(s) = 1^*$ is non-zero, and the socle is contained in every non-zero
ideal.
(E.g., see Margolis~\cite{Margolis}*{Chap. 12, Thm. 9}.)
\end{proof}

\begin{remark}
We could have used the right action of $H$ on itself to make $DH$ into
an $H$-module.   The isomorphism $\theta$ would then be replaced by an
isomorphism which sends $x$ to $s^* \circ R_x$, where $R_x$ is right
translation by $x$.   These two module structures are related by $D\chi$.  
\end{remark}

We now make a remarkable observation.

\begin{proposition}
\label{prop:Adual}
The isomorphism $\theta : \cA(2) \to  \Sigma^{23}D\cA(2)$
is given by 
\[
\theta(Sq{(r_1,r_2,r_3)}) = Sq{(7-r_1,3-r_2,1-r_3)^*}.
\]
\end{proposition}

\begin{proof}
By  Proposition~\ref{prop:gendual}, this is equivalent to the statement that
$s^*\circ L_{\chi(Sq(R'))}(Sq(R'')) = 1$ iff $R'+R'' = (7,3,1)$, where we are adding
the sequences termwise.  
In other words, 
\[
\chi(Sq(R'))Sq(R'') = \delta_{R'+R''=(7,3,1)} Sq(7,3,1),
\]
where we use 
the (generalized) Kronecker $\delta_P$, which
is $1$ if the statement $P$ is true, and $0$ if false.
This is dual to the statement that
\[
\psi(\xi_1^7 \xi_2^3 \xi_3) =
 \sum_{R' + R'' = (7,3,1)} \chi(\xi(R')) \otimes \xi(R'') =
 \sum_{R' + R'' = (7,3,1)} \xi(R') \otimes \chi(\xi(R'')),
\]
where the second equality holds by applying $\chi$ to both sides
and noting that
$\chi(\xi_1^7 \xi_2^3 \xi_3) = \xi_1^7 \xi_2^3 \xi_3$.
This can be checked quickly by any computer algebra system,
since it is an identity between polynomials in
$\F_2[\xi_1,\xi_2,\xi_3]/(\xi_1^8,\xi_2^4,\xi_3^2)$.
Here, $\xi(R) = \xi(r_1,\ldots,r_k) = \xi_1^{r_1}\cdots \xi_k^{r_k}.$
%
% Since $\cA(2)$ is free on one generator $1 = Sq^0$, $\theta$ is entirely
% determined by $\theta(1)$.   If $\theta$ is to be an isomorphism, we
% must have $\theta(1) = {Sq{(7,3,1)}}^*$, the linear dual of $Sq{(7,3,1)}$.
% (We choose as our basis for the linear functionals, the dual to the Milnor
% basis for $\cA(2)$.)  It is clear that $\theta$ is an isomorphism of  graded
% $\F_2$-vector spaces.   To see that it is $\cA(2)$-linear we need only
% show that
% \[
% \xymatrix{
% \cA(2)_k
% \ar^{Sq^i}[d]
% \ar^{\theta}[r]
% &
% \cA(2)_{23-k}^*
% \ar^{\chi(Sq^i)^*}[d]
% \\
% \cA(2)_{k+i}
% \ar^{\theta}[r]
% &
% \cA(2)_{23-k-i}^*
% \\
% }
% \]
% commutes for $i=1$, $2$, and $4$.   This is equivalent to saying that
% $\theta(Sq^iSq(R)) = \theta(Sq(R)) \chi(Sq^i)$ for each $R=(r_1,r_2,r_3)$
% and for $i=1$, $2$, and $4$:
% \[
% \xymatrix{
% \cA(2)_k
% \ar_{\chi(Sq^i)}[d]
% \ar^{\theta(Sq^iSq(R))}[rrd]
% &&
% \\
% \cA(2)_{k+i}
% \ar_{\theta(Sq(R))}[rr]
% &&
% \F_2.
% \\
% }
% \]
% It is an elementary but slightly tedious calculation
% to verify that this holds  when
% we evaluate both sides on each $Sq(S) \in \cA(2)$.
\end{proof}

A similar result holds for the quotient $\cB(2)$.

\begin{proposition}
\label{prop:Bdual}
There is a unique isomorphism of $\cA(2)$-modules $\cB(2) \to  \Sigma^{16}D\cB(2)$.
On Milnor basis elements it is given by $\theta(Sq{(r_1,r_2)}) =
Sq{(7-r_1,3-r_2)}^*$.
\end{proposition}

\begin{proof}
The proof is similar but easier.   The relevant polynomial identity is
\[
\psi(\xi_1^7 \xi_2^3) =
 \sum_{R' + R'' = (7,3)} \chi(\xi(R')) \otimes \xi(R'') =
 \sum_{R' + R'' = (7,3)} \xi(R') \otimes \chi(\xi(R'')).
\]
in the polynomial ring
$\F_2[\xi_1,\xi_2]/(\xi_1^8,\xi_2^4)$.
Here, $\psi$ and $\chi$ are defined by setting $\xi_3=0$ in
the formulas for them in $\cA(2)$.

Since the $\cA(2)$ action on $\cB(2)$
is the pullback of the $\cB(2)$ action on itself 
along the quotient map,
it is equivalent to say that this is an isomorphism of $\cB(2)$-modules.
\end{proof}

At the end of this section we speculate on generalizations
of these last two results.

Having dealt with these generalities, we
now use this duality to transport an $\cA$-module structure on
$\cA(2)$ to the dual $\cA$ action on $D\cA(2)$, and then along $\theta$
to get a dual $\cA$-module structure on $\cA(2)$, and similarly for
$\cB(2)$.  We compute the maps this process induces 
on the varieties $V_{\text{gen}}$,
$V_{\text{sym}}$, $V_{\cQ}$, and $V_{\cB}$.
First, let us note how duality interacts with $s$ and $q$.

\begin{theorem}
$Ds = \Sigma^7 qD$ and 
$Dq = \Sigma^7 sD$.
\end{theorem}

\begin{proof}
For any $\cA$-module $A$ in $V_{\cQ}$ we have a short exact sequence of
$\cA$-modules
\[
0 \to \Sigma^7 sA \longto A \longto qA \to 0.
\]
Applying $\Sigma^{23} D$, we obtain a short exact sequence
\[
0 \to \Sigma^{23} DqA \longto \Sigma^{23}DA \longto \Sigma^{16}DsA \to 0.
\]
By definition of $s$ and $q$, this means that
\begin{itemize}
\item
$\Sigma^{16} D sA  = q\Sigma^{23} DA = \Sigma^{23} qDA$, and
\item
$\Sigma^{23} D qA  = \Sigma^7 s\Sigma^{23} DA = \Sigma^{23} \Sigma^7sDA$,
\end{itemize}
from which the result follows.
\end{proof}

In order to get explicit formulas for the duality map,
the first things we must compute are the
\[
\chi(Sq^a) = \sum_{j=1}^a Sq^j \chi(Sq^{a-j}).
\]
The (very brief) MAGMA code for this in in Appendix~\ref{codedual}.
We then use parts (2) and (3) of Theorems~\ref{thm:symmetric},
\ref{thm:general}, and \ref{thm:Baction}
to compute the parameters for the dual of the generic $\cA$ actions in
each case.

\begin{theorem}
Duality $D : V_{\text{sym}} \longto V_{\text{sym}}$ is given by
\label{thm:dualsymmetric}
\begin{align*}
a_1 & \mapsto a_{23} + 1 \\
a_2 & \mapsto a_{13} + 1 \\
a_{13} & \mapsto a_2 + 1 \\
a_{23} & \mapsto a_1 + 1 \\
b_1 & \mapsto 
b_1 + a_{13} + a_2 a_{13} + a_2 a_{23} + a_{13} a_{23} + b_1 a_{13} \\
c_1 & \mapsto c_1 + a_{13} + a_{23} + a_1 a_{13} + a_2 a_{23} \\
d_1 & \mapsto d_1 + 1 +  a_1 + a_{23} + c_1 + a_1 a_{13} 
 + a_2 a_{13} + a_2 a_{23} + a_{13} a_{23} + a_{13} b_1 \\
d_2 & \mapsto d_2 + 1 + a_1 + a_{23} + b_1 + c_1 
+ a_2 a_{13} + a_1 a_2 a_{13} + a_{23} a_{13} + a_2 a_{13} a_{23} 
+ a_2 b_1 + a_{13} b_1 \\
d_3 & \mapsto  
d_3 + d_2 + d_1 + c_1 +  a_{13} + a_1 a_2 + a_2 b_1 + a_{13} b_1 
\end{align*}
Sixteen  of the $256$ symmetric $\cA$-module structures are self dual:
four of the $16$ possible $Sq^8$ actions each have four $Sq^{16}$ actions which
make the resulting $\cA$-module self-dual.   These can be found in
Appendix~\ref{app:selfdualsym}.

\end{theorem}

\begin{remark}
On the surface, this may appear incorrect,
since, by first principles, $D^2 = I$, yet on coordinates we have
\begin{itemize}
\item
$D(D(d_2))  = d_2 + a_2 + b_1 +  a_1 a_2 + a_1 a_{13}
   + a_2 a_{13} + a_2 a_{23} + a_2 a_{13} a_{23} +  a_2 b_1$, and
\item
$D(D(d_3)) = d_3 + a_1 a_{13} + a_1 a_2 a_{13} +  a_{13} b_1$.
\end{itemize}
However, the `error terms' in these formulas vanish 
when the relation defining $V_{\text{sym}}$ holds and the
coordinates lie in $\F_2$,
so that, in fact, $D^2 = I$ on the variety $V_{\text{sym}}$.
As a `sanity check', we have compared the dualization obtained from the formulas
above with that obtained by the {\tt dualizeDef} command in {\tt ext}.   They agree.
\end{remark}

\begin{theorem}
\label{thm:dualgeneral}
Duality $D : V_{\text{gen}} \longto V_{\text{gen}}$ is given by
\begin{align*}
a_1 & \mapsto 
a_{61} + a_{60} + a_{47} + a_{1} + 1\\
a_2 & \mapsto 
a_{62} + a_{60} + a_{21} + a_{2} + 1\\
a_3 & \mapsto 
 a_{62} a_{60} a_{3} + a_{62} a_{60} a_{2} + a_{62} a_{48} + 
 a_{62} a_{47} a_{2} + a_{62} a_{47} + a_{62} a_{21} a_{3} a_{2} +
 a_{62} a_{21} a_{3} + a_{62} a_{21} a_{2} a_{1} +\\
 &\phantom{YY}  a_{62} a_{21} + a_{62} a_{2} + 
   a_{62} a_{1} + a_{60} a_{48} + a_{60} a_{47} a_{2} +
 a_{60} a_{47} + a_{60} a_{21} a_{3} a_{2} + a_{60} a_{21} a_{2} a_{1} +
 a_{60} a_{21} a_{2} +\\
 &\phantom{YY} a_{60} a_{21} + a_{60} a_{2} + a_{60} a_{1} +
  a_{60} + a_{48} a_{21} + a_{48} a_{2} + a_{47} a_{21} a_{2} +
 a_{47} a_{21} + a_{21} a_{3} a_{2} + a_{21} a_{2} a_{1} +\\
 &\phantom{YY} a_{21} a_{1} + a_{21} + a_{3}\\
a_{21} & \mapsto 
a_{62} + 1\\
a_{47} & \mapsto 
a_{61} + 1\\
a_{48} & \mapsto 
a_{48}\\
a_{60} & \mapsto 
a_{60}\\
a_{61} & \mapsto 
a_{47} + 1\\
a_{62} & \mapsto 
a_{21} + 1\\
b_1 & \mapsto 
  b_{1} + a_{62} a_{60} a_{48} + a_{62} a_{60} a_{47} a_{2} + a_{62} a_{60} a_{47} +
  a_{62} a_{60} a_{21} a_{3} a_{2} + a_{62} a_{60} a_{21} a_{3} +
  a_{62} a_{60} a_{21} a_{2} a_{1} +\\
  & \phantom{YY} a_{62} a_{60} a_{21} + a_{62} a_{60} a_{3} +
  a_{62} a_{60} a_{1} + a_{62} a_{60} + a_{62} a_{47} + a_{62} + a_{61} a_{60} +
  a_{61} a_{21} + a_{61} +\\
  &\phantom{YY} a_{60} a_{48} a_{21} + a_{60} a_{48} a_{2} +
  a_{60} a_{48} + a_{60} a_{47} a_{21} a_{2} + a_{60} a_{47} a_{21} +
  a_{60} a_{47} a_{2} + a_{60} a_{47} + a_{60} a_{21} a_{2} +\\
  &\phantom{YY} a_{60} a_{21} a_{1} + a_{47} + a_{21} + a_{2} + a_{1}\\
b_2 & \mapsto 
  b_{2} + b_{1} + a_{62} a_{60} a_{3} a_{2} + a_{62} a_{60} a_{3} +
  a_{62} a_{60} a_{2} + a_{62} a_{48} + a_{62} a_{47} a_{2} + a_{62} a_{47} +
  a_{62} a_{21} a_{2} a_{1} +\\
  & \phantom{YY} a_{62} a_{21} a_{2} + a_{62} a_{21} a_{1} +
  a_{62} a_{21} + a_{62} a_{3} a_{2} + a_{62} a_{2} a_{1} + a_{62} + a_{61} +
  a_{60} a_{48} + a_{60} a_{47} a_{2} + a_{60} a_{47} +\\
  & \phantom{YY}a_{60} a_{21} a_{3} a_{2} +
  a_{60} a_{21} a_{3} + a_{60} a_{21} a_{2} a_{1} + a_{60} a_{21} a_{1} +
  a_{60} a_{21} + a_{60} a_{3} a_{2} + a_{60} a_{2} a_{1} + a_{48} a_{21} +
  a_{48} a_{2} +\\
  &\phantom{YY} a_{47} a_{21} a_{2} + a_{47} a_{21} + a_{47} + a_{21} a_{3} a_{2} +
  a_{21} a_{3} + a_{21} a_{2} a_{1} + a_{21} a_{2} + a_{21} a_{1} + a_{3} a_{2} + 1\\
b_3 & \mapsto 
  b_{3} + b_{1} + a_{62} a_{60} a_{21} a_{3} + a_{62} a_{60} a_{21} a_{2} +
  a_{62} a_{60} a_{3} a_{2} + a_{62} a_{60} a_{2} + a_{62} a_{48} a_{21} +
  a_{62} a_{47} a_{21} a_{2} +\\
  & \phantom{YY} a_{62} a_{47} a_{21} + a_{62} a_{47} a_{2} +
  a_{62} a_{21} a_{3} + a_{62} a_{21} a_{2} a_{1} + a_{62} a_{21} a_{2} +
  a_{62} a_{21} a_{1} + a_{62} a_{3} a_{2} + a_{62} a_{3} +\\
  &\phantom{YY} a_{62} a_{2} a_{1} +
  a_{62} a_{1} + a_{62} + a_{61} + a_{60} a_{48} a_{21} + a_{60} a_{47} a_{21} a_{2} +
  a_{60} a_{47} a_{21} + a_{60} a_{47} a_{2} +\\
  &\phantom{YY} a_{60} a_{21} a_{3} a_{2} +
  a_{60} a_{21} a_{2} a_{1} + a_{60} a_{21} a_{2} + a_{60} a_{21} a_{1} +
  a_{60} a_{21} + a_{60} a_{3} a_{2} + a_{60} a_{2} a_{1} + a_{60} a_{1} +\\
  &\phantom{YY} 
  a_{48} a_{21} a_{2} + a_{48} a_{21} + a_{48} a_{2} + a_{47} a_{21} + a_{47} +
  a_{21} a_{3} a_{2} + a_{21} a_{2} a_{1} + a_{21} a_{2} + a_{21} a_{1} +\\
  &\phantom{YY} 
  a_{3} a_{2} + a_{3} + a_{2} a_{1} + 1\\
b_4 & \mapsto 
  b_{4} + b_{3} + b_{2} + b_{1} + a_{62} a_{60} a_{3} a_{2} + a_{62} a_{21} a_{2} a_{1} +
  a_{62} a_{21} + a_{62} a_{3} a_{2} + a_{62} a_{2} a_{1} + a_{62} a_{1} +
  a_{60} a_{21} a_{3} a_{2} +\\
  &\phantom{YY} a_{60} a_{21} a_{2} a_{1} + a_{60} a_{21} + a_{60} a_{3} +
  a_{60} a_{2} a_{1} + a_{60} a_{1} + a_{21} a_{3} a_{2} + a_{21} a_{3} + a_{21} a_{2} +
  a_{21} a_{1} + a_{21} +\\
  &\phantom{YY} a_{3} a_{2} + a_{3} + a_{2} a_{1}
\end{align*}
Forty of the $1600$ $\cA$-module structures are self dual:
ten of the $100$ possible $Sq^8$ actions each have four $Sq^{16}$ actions which
make the resulting $\cA$-module self-dual.   These can be found in
Appendix~\ref{app:selfdualgen}.
\end{theorem}

\begin{remark}
As in the symmetric case, the `discrepancy' between $D^2$ and the
identity homomorphism vanishes on the $\F_2$ points in $V_{\text{gen}}$.
\end{remark}

\begin{theorem}
\label{thm:dualBaction}
Duality $D : V_{\cB} \longto V_{\cB}$ is given by
\begin{align*}
a_1 & \mapsto 
a_{23} + 1\\
a_2 & \mapsto 
a_{13} + 1\\
a_{13} & \mapsto 
a_2 + 1\\
a_{23} & \mapsto 
a_1 + 1\\
c_1 & \mapsto 
c_1 + a_{13} + a_1 a_{13} + a_{23} + a_2 a_{23} 
\end{align*}
The eight $\cA$-module structures which satisfy 
$a_{13} = a_2 + 1$ and $a_{23} = a_1 + 1$ are self dual.
\end{theorem}

Finally, we end this section with some remarks on the surprisingly simple
Propositions~\ref{prop:Adual} and
\ref{prop:Bdual}.

\begin{remark}
My first proof of these two propositions was a hand calculation
of the fact 
that defining $\theta(Sq(r_1,r_2,r_3)) = Sq(7-r_1,3-r_2,1-r_3)^*$
resulted in an $\cA(2)$-module {\em homomorphism}, since it is clearly an
$\F_2$-vector space {\em isomorphism}.   To check that
it is an $\cA(2)$-module homomorphism, it suffices to check
that it preserves the action by $Sq^1$, $Sq^2$ and $Sq^4$.  To check that,
we may check that both sides produce the same result when evaluated
on any $Sq(R'')$.  This amounts to checking that, under the assumption
that $R' + R'' = (7,3,1)$ or $(7,3)$, a dozen or so combinatorial
identities hold.
\end{remark}

The form of these isomorphisms suggests the following conjecture.
It seems that it should be an elementary computation from the 
definition of the antipode.

\begin{conjecture}
\label{conj:dual}
Let $H$ be a connected finite dimensional Hopf algebra with antipode
of formal dimension $N$.
Let $\{x_\alpha\}$ and $\{y_\alpha\}$ be dual bases for $H$.
That is, if $\deg(x_\alpha)+\deg(y_\beta) = N$ then
$x_\alpha y_\beta = \delta_{\alpha,\beta} s$ where the
socle of $H$ is $\langle s \rangle$.
Then
\[
\psi(s) = \sum_\alpha x_\alpha \otimes \chi(y_\alpha)?
\]
Note that $\psi(s) = \psi(x_\alpha) \psi(y_\alpha)$ for each
$\alpha$, and that $\psi(x_\alpha)\psi(y_\beta) = 0$ if
$\alpha \neq \beta$.

Dually,
suppose that $\{x_\alpha\}$ and $\{y_\alpha\}$ are dual bases with
respect to the coproduct:  $\psi(s) = \sum_\alpha x_\alpha \otimes y_\alpha$.
Then
\[
x_\alpha \chi(y_\beta) = \delta_{\alpha,\beta} s
\]
when $\deg(x_\alpha)+\deg(y_\beta) = N$?
\end{conjecture}

\begin{remark}
If not a general fact about the antipode in a Hopf algebra, perhaps
Propositions~\ref{prop:Adual} and \ref{prop:Bdual} are examples of a general
polynomial identity.
Let $(n_1, \ldots, n_k)$ be a sequence of non-negative integers, 
let $P = \F_2[\xi_1, \ldots,\xi_k]$ and let
$I = (\xi_1^{2^{n_1}}, \ldots, \xi_k^{2^{n_k}})$.  Interpret $\xi_0$ as $1 \in P$.
Define $\psi : P \longto P \otimes P$ by
\[
\psi(\xi_j) = \sum_{i=0}^{j} \xi_{j-i}^{2^i} \otimes \xi_i
\]
and define $\chi : P \longto P$ by 
\[
\chi(\xi_j) = \sum_{i=0}^{j-1} \xi_{j-i}^{2^i} \otimes \chi(\xi_i),
\]
so that 
\[
\sum_{i=0}^{j} \xi_{j-i}^{2^i} \otimes \chi(\xi_i) = 0.
\]

What are the conditions on the sequence
$(n_1, \ldots, n_k)$ which ensure that 
\begin{align*}
\psi(\xi_1^{2^{n_1}-1}\cdots \xi_k^{2^{n_k}-1}) & =
 \sum_{R' + R'' = (2^{n_1}-1,\ldots,2^{n_k}-1)} \chi(\xi(R')) \otimes \xi(R'') \\
 & = \sum_{R' + R'' = (2^{n_1}-1,\ldots,2^{n_k}-1)} \xi(R') \otimes \chi(\xi(R''))
\end{align*}
modulo $I$?   It is not  necessary that it be a valid profile
for a sub Hopf algebra of the Steenrod algebra since it is true for
both the valid profile 
$(3,2,1)$ and the invalid profile $(3,2)$.  It is false
for the invalid profile $(3,0,1)$.
\end{remark}

\section{Actions in the literature}
\label{sec:literature}

\begin{proposition}
\label{prop:BBBCX}
The five author paper~\cite{BBBCX} uses the symmetric $\cA$-module structure on
$\cA(2)$ in which all of $a_1$, $ a_2$, $ a_{13}$, $ a_{23}$, $ b_1$, $ 
c_1$, $ d_1$, $ d_2$, $ d_3$ are zero.  
The dual $\cA$-module structure has
$a_1 = a_2 = a_{13} = a_{23} =  d_1 = d_2 = 1$ and
$b_{1} = c_1 = d_3 = 0$.
\end{proposition}

\begin{proposition}
\label{prop:BE}
The paper~\cite{BE} uses the $\cA$ action on $\cB(2)$ in which
all of $a_1$, $ a_2$, $ a_{13}$, $ a_{23}$ and $c_1$  are zero.  
This $\cA$-module sits in the short exact sequence~\eqref{Qext} 
with the $\cA$ action on $\cA(2)$ of~\cite{BBBCX}.
\end{proposition}

In order to determine the parameters associated to the $\cA$-module
structures used in \cite{BBBCX} and \cite{BE}, we need to relate the
Milnor basis used here to the generators used in their {\tt moddef} 
(module definition) files.
One part of this process is to write the Milnor basis elements
in admissible form, e.g., by using the {\tt sage} code for the Steenrod algebra,
as in
\begin{verbatim}
sage: A = SteenrodAlgebra(p=2)
sage: Sq(0,2).change_basis('Adem')
Sq^4 Sq^2 + Sq^5 Sq^1 + Sq^6.
\end{verbatim}
Let us write $g_i$ for generator number $i$ in a  {\tt moddef} file.
These generators are ordered first by degree, then by reverse lexicographic order:
\[
g_0 = Sq(0),\, g_1 = Sq(1),\, g_2 = Sq(2),\, g_3 = Sq(3),\, g_4 = Sq(0,1),\,
g_5 = Sq(4),\, g_6 = Sq(1,1),\,
\ldots .
\]
Ordering by degree is required by {\tt ext}, but the choice of ordering
within each degree is chosen by the person writing the {\tt moddef} file.
From Theorem~\ref{thm:symmetric}, we see that for Proposition~\ref{prop:BBBCX}
we only need to verify this ordering in degrees 12 and 14. In degree 12,
the {\tt moddef} file in \cite{BBBCX} gives
\begin{align*}
Sq(6,2)\cdot g_0 & = (Sq^{10}Sq^2 + Sq^{11}Sq^1)  \cdot g_0 &= g_{32}, \\
Sq(3,3)\cdot g_0 & = (Sq^{9}Sq^3 )  \cdot g_0 &= g_{33}, \\
Sq(5,0,1)\cdot g_0 & = (Sq^{9}Sq^2Sq^1 + Sq^9 Sq^3 + Sq^{11}Sq^1)  \cdot g_0 &= g_{34}, \\
Sq(2,1,1)\cdot g_0 & = (Sq^{8}Sq^3Sq^1 + Sq^9 Sq^2Sq^1 + Sq^9 Sq^3 ) \cdot g_0 &= g_{35},
\end{align*}
while in degree 14, it gives
\begin{align*}
Sq(5,3)\cdot g_0 & = (Sq^{11}Sq^3 + Sq^{13}Sq^1)  \cdot g_0 &= g_{41}, \\
Sq(7,0,1)\cdot g_0 & = (Sq^{11}Sq^2Sq^1 + Sq^{11}Sq^3 + Sq^{13}Sq^1)  \cdot g_0 &= g_{42}, \\
Sq(4,1,1)\cdot g_0 & = (Sq^{10}Sq^3Sq^1 + Sq^{11}Sq^2Sq^1 + Sq^{11}Sq^3)  \cdot g_0 &= g_{43}, \\
Sq(1,2,1)\cdot g_0 & = (Sq^{9}Sq^4Sq^1 + Sq^{11}Sq^2Sq^1 + Sq^{13}Sq^1)  \cdot g_0 &= g_{44}.
\end{align*}
We are now in a position to carry out the proofs.

\begin{proof}[Proof of Proposition~\ref{prop:BBBCX}]
First, $Sq^8 \cdot 1 = Sq^8 \cdot g_0 = 0$ implies that $a_1=a_2=b_1 = 0$.
Then, $Sq^{16} \cdot 1 = Sq^{16} \cdot g_0 = 0$ implies that
$c_1 = d_1 = d_2 = d_3 = 0$.  Next,
$Sq^8\cdot Sq^4 = a_{13} Sq(6,2) + \cdots = a_{13} g_{32} + \cdots$, while
$Sq^8 \cdot g_5 = g_{34}  + g_{35}$ gives $a_{13} = 0$.
Finally,   $Sq^8 \cdot Sq(0,2) = a_{23} Sq(5,3) + \cdots
= a_{23} g_{41} + \cdots$ while $Sq^8 \cdot g_{11} = g_{42} + g_{43}$
gives $a_{23} = 0$.  Here we have used the additional information from
the {\tt moddef} file that $Sq^4 \cdot g_0 = g_5$ and
$Sq(0,2) \cdot g_0 = (Sq^6 + Sq^5 Sq^1 + Sq^4 Sq^2)\cdot g_0 = g_{11}$.

The dual $\cA$-module structure is then given by Theorem~\ref{thm:dualsymmetric}.
\end{proof}

The proof of Proposition~\ref{prop:BE} is much easier, this time using
Theorem~\ref{thm:Baction} to determine the coefficients.

\begin{proof}[Proof of Proposition~\ref{prop:BE}]
First,  $Sq^8 \cdot 1 = Sq^8 \cdot g_0 = 0$ gives $a_1 = a_2 = 0$.
Then, $Sq^{16} \cdot 1 = Sq^{16} \cdot g_0 = 0$ gives $c_1 = 0$.
Finally, $Sq^8 \cdot Sq^4  = Sq^8 \cdot Sq^4 g_0 = Sq^8 \cdot g_5 = 0$
gives $a_{13}  = 0$,
and  $Sq^8\cdot Sq(2,2) = Sq^8 \cdot g_{16} = 0$
gives $a_{23} = 0$.   Here we use that
$Sq(2,2) \cdot g_0 = (Sq^6Sq^2 + Sq^7Sq^1) \cdot g_0 = g_{16}$.

The parameters $a_i$, etc., are chosen compatibly in the calculations for
$\cA(2)$ and $\cB(2)$, so the final claim follows from the fact that they 
are all zero in the $\cA$-module structures used in the two papers.
\end{proof}

\appendix

\vspace{10ex}

\section{$Sq^8$ in the Symmetric Case}
\label{sq8}

The initial linear transformation $Sq^8$ in the symmetric case
is given below, to show how the
block structure appears in the MAGMA code.  
The omitted matrices at the end all map to the 0 vector space.

MAGMA represents vector
space homomorphisms by the matrix which gives their action on the right.
Hence each row represents the value of the transformation on a
corresponding basis element.  

For example, the action of $Sq^8$ on degree 0 is
given by the first matrix below:  $Sq^8 \cdot 1 =
a_1 Sq(5,1) + a_2 Sq(2,2) + b_1 Sq(1,0,1)$.   

For a second example, 
the action of $Sq^8$ on degree $7$ is given by the eighth matrix below.
The ordered bases we use for $A_7$  and $A_{15}$ are
\begin{itemize}
\item
for $A_7$, $(Sq(7), Sq(4,1), Sq(1,2), Sq(0,0,1))$, and
\item
for $A_{15}$,  $(Sq(6,3), Sq(5,1,1), Sq(2,2,1))$.  
\end{itemize}
Hence the matrix below says that
\begin{align*}
Sq^8 \cdot Sq(7) &= a_{24} Sq(6,3) + b_{28} Sq(5,1,1) + b_{29} Sq(2,2,1), \\
Sq^8 \cdot Sq(4,1) &= a_{25} Sq(6,3) + b_{30} Sq(5,1,1) + b_{31} Sq(2,2,1), \\
Sq^8 \cdot Sq(1,2) &= a_{26} Sq(6,3) + b_{32} Sq(5,1,1) + b_{33} Sq(2,2,1),  \\
Sq^8 \cdot Sq(0,0,1) &= \phantom{a_{24} Sq(6,3) +} a_{1} Sq(5,1,1) + a_{2} Sq(2,2,1). \\
\end{align*}

\begin{verbatim}
Initial Sq^8:
[*
[a1 a2 b1],

[a3 a4 a5 b2],

[a6 a7 a8 b3 b4],

[ a9 a10  b5  b6]
[a11 a12  b7  b8],

[a13 a14  b9 b10]
[a15 a16 b11 b12],

[a17 a18 b13 b14 b15]
[a19 a20 b16 b17 b18],

[a21 b19 b20 b21]
[a22 b22 b23 b24]
[a23 b25 b26 b27],

[a24 b28 b29]
[a25 b30 b31]
[a26 b32 b33]
[  0  a1  a2],

[a27 b34 b35 b36]
[a28 b37 b38 b39]
[  0  a3  a4  a5],

[b40 b41 b42]
[b43 b44 b45]
[b46 b47 b48]
[ a6  a7  a8],

[b49 b50]
[b51 b52]
[b53 b54]
[ a9 a10]
[a11 a12],

[b55 b56]
[b57 b58]
[a13 a14]
[a15 a16],

[b59 b60]
[b61 b62]
[a17 a18]
[a19 a20],

[b63]
[b64]
[a21]
[a22]
[a23],

[b65]
[a24]
[a25]
[a26],

[b66]
[a27]
[a28],

Matrices with 4 rows and 0 columns, 
...
*]
\end{verbatim}

\section{$Sq^{16}$ in the Symmetric Case}
\label{sq16}

\begin{verbatim}

Initial Sq^16:
[*
[c1 d1 d2 d3],

[d4 d5 d6],

[d7 d8],

[ d9 d10]
[d11 d12],

[d13 d14]
[d15 d16],

[d17]
[d18],

[d19]
[d20]
[d21],

[d22]
[d23]
[d24]
[ c1],

Matrix with 3 rows and 0 columns, 
...
*]
\end{verbatim}

\section{$Sq^8$ in the General Case}
\label{gensq8}

\begin{verbatim}
Initial Sq^8:
[*
[a1 a2 a3],

[a4 a5 a6 a7],

[ a8  a9 a10 a11 a12],

[a13 a14 a15 a16]
[a17 a18 a19 a20],

[a21 a22 a23 a24]
[a25 a26 a27 a28],

[a29 a30 a31 a32 a33]
[a34 a35 a36 a37 a38],

[a39 a40 a41 a42]
[a43 a44 a45 a46]
[a47 a48 a49 a50],

[a51 a52 a53]
[a54 a55 a56]
[a57 a58 a59]
[a60 a61 a62],

[a63 a64 a65 a66]
[a67 a68 a69 a70]
[a71 a72 a73 a74],

[a75 a76 a77]
[a78 a79 a80]
[a81 a82 a83]
[a84 a85 a86],

[a87 a88]
[a89 a90]
[a91 a92]
[a93 a94]
[a95 a96],

[ a97  a98]
[ a99 a100]
[a101 a102]
[a103 a104],

[a105 a106]
[a107 a108]
[a109 a110]
[a111 a112],

[a113]
[a114]
[a115]
[a116]
[a117],

[a118]
[a119]
[a120]
[a121],

[a122]
[a123]
[a124],

Matrix with 4 rows and 0 columns,
...
*]
\end{verbatim}

\section{$Sq^{16}$ in the General Case}
\label{gensq16}

\begin{verbatim}
Initial Sq^16:
[*
[b1 b2 b3 b4],

[b5 b6 b7],

[b8 b9],

[b10 b11]
[b12 b13],

[b14 b15]
[b16 b17],

[b18]
[b19],

[b20]
[b21]
[b22],

[b23]
[b24]
[b25]
[b26],

Matrix with 3 rows and 0 columns,
...
*]
\end{verbatim}

\section{The sixteen $Sq^8$ in the symmetric case}
\label{allsymSq8}

The following table gives the values of 
$[a_1, a_2, a_{13}, a_{23}, b_1]$ which solve the single
relation we found in the symmetric case.  
These were produced by the MAGMA code in Appendix~\ref{codesym3}.

\begin{verbatim}
       [a1,a2,a13,a23,b1]
   1 : [ 0, 0,  0,  0, 0]
   2 : [ 0, 0,  0,  1, 0]
   3 : [ 0, 0,  1,  0, 0]
   4 : [ 0, 0,  1,  1, 0]
   5 : [ 0, 1,  0,  1, 0]
   6 : [ 0, 1,  0,  1, 1]
   7 : [ 0, 1,  1,  0, 0]
   8 : [ 0, 1,  1,  1, 0]
   9 : [ 1, 0,  0,  0, 0]
  10 : [ 1, 0,  0,  1, 0]
  11 : [ 1, 0,  1,  0, 1]
  12 : [ 1, 0,  1,  1, 1]
  13 : [ 1, 1,  0,  0, 0]
  14 : [ 1, 1,  0,  0, 1]
  15 : [ 1, 1,  1,  0, 0]
  16 : [ 1, 1,  1,  1, 0]
\end{verbatim}

\section{The $100$ $Sq^8$ actions in the general case}
\label{allgenSq8}

\begin{verbatim}
       [ a1, a2, a3,a21,a47,a48,a60,a61,a62 ]
   1 : [  0,  0,  0,  0,  0,  0,  0,  0,  1 ]
   2 : [  0,  0,  0,  0,  0,  0,  0,  1,  1 ]
   3 : [  0,  0,  0,  0,  0,  0,  1,  1,  1 ]
   4 : [  0,  0,  0,  0,  0,  1,  0,  0,  0 ]
   5 : [  0,  0,  0,  0,  0,  1,  0,  0,  1 ]
   6 : [  0,  0,  0,  0,  0,  1,  0,  1,  0 ]
   7 : [  0,  0,  0,  0,  0,  1,  0,  1,  1 ]
   8 : [  0,  0,  0,  0,  1,  0,  1,  0,  1 ]
   9 : [  0,  0,  0,  0,  1,  0,  1,  1,  0 ]
  10 : [  0,  0,  0,  0,  1,  1,  0,  0,  0 ]
  11 : [  0,  0,  0,  0,  1,  1,  0,  1,  0 ]
  12 : [  0,  0,  0,  0,  1,  1,  1,  0,  0 ]
  13 : [  0,  0,  0,  1,  0,  0,  0,  0,  0 ]
  14 : [  0,  0,  0,  1,  0,  0,  0,  0,  1 ]
  15 : [  0,  0,  0,  1,  0,  0,  1,  0,  0 ]
  16 : [  0,  0,  0,  1,  0,  0,  1,  0,  1 ]
  17 : [  0,  0,  0,  1,  0,  0,  1,  1,  0 ]
  18 : [  0,  0,  0,  1,  0,  0,  1,  1,  1 ]
  19 : [  0,  0,  0,  1,  1,  0,  0,  0,  0 ]
  20 : [  0,  0,  0,  1,  1,  0,  0,  1,  1 ]
  21 : [  0,  0,  0,  1,  1,  0,  1,  0,  0 ]
  22 : [  0,  0,  0,  1,  1,  0,  1,  0,  1 ]
  23 : [  0,  0,  0,  1,  1,  0,  1,  1,  0 ]
  24 : [  0,  0,  0,  1,  1,  0,  1,  1,  1 ]
  25 : [  0,  0,  1,  0,  0,  0,  0,  0,  1 ]
  26 : [  0,  0,  1,  0,  0,  0,  0,  1,  1 ]
  27 : [  0,  0,  1,  0,  0,  0,  1,  0,  0 ]
  28 : [  0,  0,  1,  0,  0,  1,  0,  0,  1 ]
  29 : [  0,  0,  1,  0,  0,  1,  0,  1,  1 ]
  30 : [  0,  0,  1,  0,  0,  1,  1,  1,  0 ]
  31 : [  0,  0,  1,  1,  0,  0,  1,  0,  0 ]
  32 : [  0,  0,  1,  1,  0,  0,  1,  1,  0 ]
  33 : [  0,  0,  1,  1,  0,  1,  0,  0,  1 ]
  34 : [  0,  0,  1,  1,  1,  0,  1,  0,  0 ]
  35 : [  0,  0,  1,  1,  1,  0,  1,  1,  0 ]
  36 : [  0,  0,  1,  1,  1,  1,  0,  1,  1 ]
  37 : [  0,  1,  0,  0,  1,  0,  0,  0,  1 ]
  38 : [  0,  1,  0,  0,  1,  1,  1,  0,  0 ]
  39 : [  0,  1,  0,  0,  1,  1,  1,  1,  0 ]
  40 : [  0,  1,  0,  1,  0,  0,  1,  0,  0 ]
  41 : [  0,  1,  0,  1,  0,  1,  0,  0,  1 ]
  42 : [  0,  1,  0,  1,  0,  1,  0,  1,  1 ]
  43 : [  0,  1,  0,  1,  0,  1,  1,  0,  1 ]
  44 : [  0,  1,  0,  1,  1,  0,  1,  1,  0 ]
  45 : [  0,  1,  0,  1,  1,  1,  0,  0,  1 ]
  46 : [  0,  1,  0,  1,  1,  1,  0,  1,  1 ]
  47 : [  0,  1,  1,  0,  1,  0,  0,  0,  1 ]
  48 : [  0,  1,  1,  1,  0,  0,  0,  0,  0 ]
  49 : [  0,  1,  1,  1,  0,  0,  0,  1,  0 ]
  50 : [  0,  1,  1,  1,  0,  0,  1,  0,  0 ]
  51 : [  0,  1,  1,  1,  0,  1,  1,  1,  1 ]
  52 : [  0,  1,  1,  1,  1,  0,  1,  1,  0 ]
  53 : [  1,  0,  0,  0,  0,  0,  1,  0,  1 ]
  54 : [  1,  0,  0,  0,  0,  0,  1,  1,  0 ]
  55 : [  1,  0,  0,  0,  0,  1,  0,  0,  0 ]
  56 : [  1,  0,  0,  0,  0,  1,  0,  1,  0 ]
  57 : [  1,  0,  0,  0,  0,  1,  1,  0,  0 ]
  58 : [  1,  0,  0,  0,  1,  0,  0,  0,  1 ]
  59 : [  1,  0,  0,  0,  1,  0,  0,  1,  1 ]
  60 : [  1,  0,  0,  0,  1,  0,  1,  1,  1 ]
  61 : [  1,  0,  0,  0,  1,  1,  0,  0,  0 ]
  62 : [  1,  0,  0,  0,  1,  1,  0,  0,  1 ]
  63 : [  1,  0,  0,  0,  1,  1,  0,  1,  0 ]
  64 : [  1,  0,  0,  0,  1,  1,  0,  1,  1 ]
  65 : [  1,  0,  1,  0,  1,  0,  0,  0,  1 ]
  66 : [  1,  0,  1,  0,  1,  0,  0,  1,  1 ]
  67 : [  1,  0,  1,  0,  1,  0,  1,  0,  0 ]
  68 : [  1,  0,  1,  0,  1,  1,  0,  0,  1 ]
  69 : [  1,  0,  1,  0,  1,  1,  0,  1,  1 ]
  70 : [  1,  0,  1,  0,  1,  1,  1,  1,  0 ]
  71 : [  1,  0,  1,  1,  0,  0,  0,  1,  0 ]
  72 : [  1,  0,  1,  1,  0,  1,  1,  0,  1 ]
  73 : [  1,  0,  1,  1,  0,  1,  1,  1,  1 ]
  74 : [  1,  0,  1,  1,  1,  0,  0,  1,  0 ]
  75 : [  1,  0,  1,  1,  1,  1,  1,  0,  1 ]
  76 : [  1,  0,  1,  1,  1,  1,  1,  1,  1 ]
  77 : [  1,  1,  0,  0,  0,  0,  0,  0,  0 ]
  78 : [  1,  1,  0,  0,  0,  0,  0,  1,  1 ]
  79 : [  1,  1,  0,  0,  0,  1,  0,  0,  0 ]
  80 : [  1,  1,  0,  0,  0,  1,  1,  0,  0 ]
  81 : [  1,  1,  0,  0,  0,  1,  1,  1,  0 ]
  82 : [  1,  1,  0,  0,  1,  0,  0,  1,  0 ]
  83 : [  1,  1,  0,  0,  1,  1,  0,  1,  0 ]
  84 : [  1,  1,  0,  1,  0,  0,  1,  1,  0 ]
  85 : [  1,  1,  0,  1,  0,  1,  0,  0,  1 ]
  86 : [  1,  1,  0,  1,  0,  1,  0,  1,  1 ]
  87 : [  1,  1,  0,  1,  1,  0,  1,  0,  0 ]
  88 : [  1,  1,  0,  1,  1,  1,  0,  0,  1 ]
  89 : [  1,  1,  0,  1,  1,  1,  0,  1,  1 ]
  90 : [  1,  1,  0,  1,  1,  1,  1,  0,  1 ]
  91 : [  1,  1,  1,  0,  0,  0,  0,  1,  1 ]
  92 : [  1,  1,  1,  0,  0,  0,  1,  0,  1 ]
  93 : [  1,  1,  1,  0,  0,  0,  1,  1,  1 ]
  94 : [  1,  1,  1,  0,  1,  0,  1,  0,  1 ]
  95 : [  1,  1,  1,  0,  1,  0,  1,  1,  1 ]
  96 : [  1,  1,  1,  1,  0,  0,  1,  1,  0 ]
  97 : [  1,  1,  1,  1,  1,  0,  0,  0,  0 ]
  98 : [  1,  1,  1,  1,  1,  0,  0,  1,  0 ]
  99 : [  1,  1,  1,  1,  1,  0,  1,  0,  0 ]
 100 : [  1,  1,  1,  1,  1,  1,  1,  1,  1 ]
\end{verbatim}

\section{Resulting relations, symmetric case}
\label{symrelations}

Writing out the matrices defining $Sq^8$ and $Sq^{16}$ after
reducing to the minimal set of variables defining them is impractical
since some of the entries are rather lengthy sums, making the matrices
unreadable.   
Instead, we refer the reader to the initial versions with
all indeterminates in Appendices~\ref{sq8} and \ref{sq16},
and provide the definitions of those indeterminates in terms
of the $9$ variables that we have reduced to.

\begin{verbatim}

Writing each of the 119 variables in terms of the final 9 variables:
the entries are <index, name, value>.

[
    <1, d24, d1 + c1 + b1*a13 + b1 + a23*a13 + a13*a2*a1 + a13*a1 + a13 + a1>,
    <2, d23, d2 + c1 + b1*a13*a2 + b1*a2 + a23*a13*a2 + a23*a13 + a23*a2 + 
        a13*a2 + a13 + a1>,
    <3, d22, d3 + d2 + d1 + c1 + b1*a13 + b1 + a23*a2 + a13*a2*a1 + a13*a2 + 
        a13*a1 + a2>,
    <4, d21, d1 + c1 + b1*a2 + b1 + a13*a1 + 1>,
    <5, d20, 1>,
    <6, d19, d3 + d2 + d1 + c1 + b1*a13 + b1*a2 + b1 + a13*a1>,
    <7, d18, b1>,
    <8, d17, d3 + d1 + b1*a13*a2 + b1*a13 + b1*a2 + b1 + a23*a13*a2 + a23*a13 + 
        a13*a2 + a13 + a2 + a1>,
    <9, d16, 0>,
    <10, d15, a2 + 1>,
    <11, d14, d3 + d1 + b1*a13*a2 + b1*a13 + b1*a2 + b1 + a23*a13*a2 + a23*a13 +
        a13*a2 + a13 + a2 + a1>,
    <12, d13, d2 + c1 + b1*a13*a2 + b1*a2 + a23*a13*a2 + a23*a13 + a13*a2 + 
        a13*a1 + a13 + 1>,
    <13, d12, d2 + a2 + a1>,
    <14, d11, a2 + 1>,
    <15, d10, d3 + d2 + a2 + a1>,
    <16, d9, 0>,
    <17, d8, d3 + d2 + b1 + a1>,
    <18, d7, b1 + a2 + a1 + 1>,
    <19, d6, d3>,
    <20, d5, 0>,
    <21, d4, d1 + b1 + a1>,
    <22, d3, d3>,
    <23, d2, d2>,
    <24, d1, d1>,
    <25, c1, c1>,
    <26, b66, b1*a13 + b1 + a23*a13 + a23*a2 + a13*a2 + a13 + 1>,
    <27, b65, 1>,
    <28, b64, b1*a13 + b1 + a23*a13 + a23*a2 + a13*a2 + a13>,
    <29, b63, b1*a13 + b1 + a23*a13 + a23*a2 + a13*a2 + a13 + 1>,
    <30, b62, 0>,
    <31, b61, 1>,
    <32, b60, b1*a13 + b1 + a23*a13 + a23*a2 + a13*a2 + a13 + 1>,
    <33, b59, b1*a13 + b1 + a23*a13 + a23*a2 + a13*a2 + a13>,
    <34, b58, b1*a13*a2 + b1*a2 + a23*a13 + a23 + a13*a1 + a2 + a1 + 1>,
    <35, b57, 1>,
    <36, b56, b1*a13*a2 + b1*a13 + b1*a2 + b1 + a23*a13*a2 + a23*a13 + a2 + 1>,
    <37, b55, 0>,
    <38, b54, 1>,
    <39, b53, 1>,
    <40, b52, b1*a13*a2 + b1*a13 + b1*a2 + b1 + a23*a13*a2 + a23*a13 + a2 + 1>,
    <41, b51, b1*a13 + b1 + a23*a13 + a23*a2 + a13*a2 + a13>,
    <42, b50, 0>,
    <43, b49, 1>,
    <44, b48, b1*a13*a2 + b1*a2 + a23*a13 + a23 + a13*a1 + a2 + a1>,
    <45, b47, 1>,
    <46, b46, b1*a13*a2 + b1*a13 + a13*a2*a1 + a13*a1 + a13 + a2 + 1>,
    <47, b45, b1*a13*a2 + b1*a2 + a23*a13 + a23 + a13*a1 + a2 + a1 + 1>,
    <48, b44, 0>,
    <49, b43, b1*a13*a2 + a23*a13*a2 + a23*a2 + a23 + a13*a2 + a1>,
    <50, b42, 0>,
    <51, b41, 1>,
    <52, b40, b1*a2 + a23*a2 + a13*a2*a1 + a13*a2 + a13 + a2>,
    <53, b39, b1*a13*a2 + b1*a2 + a23*a13 + a23 + a13*a1 + a2 + a1 + 1>,
    <54, b38, b1*a13*a2 + b1*a2 + a23*a13 + a23 + a13*a1 + a2 + a1>,
    <55, b37, b1*a13*a2 + a23*a13*a2 + a23*a2 + a23 + a13*a2 + a1>,
    <56, b36, 0>,
    <57, b35, 1>,
    <58, b34, 1>,
    <59, b33, 0>,
    <60, b32, b1*a13 + a23*a13*a2 + a23*a2 + a23 + a13*a2*a1 + a13*a2 + a13*a1 +
        a13 + a2 + a1 + 1>,
    <61, b31, 1>,
    <62, b30, b1*a2 + a23*a2 + a13*a2*a1 + a13*a2 + a13 + a2 + 1>,
    <63, b29, 0>,
    <64, b28, b1*a2 + a23*a2 + a13*a2*a1 + a13*a2 + a13 + a2 + 1>,
    <65, b27, b1*a13*a2 + b1*a2 + a23*a13 + a23 + a13*a1 + a2 + a1>,
    <66, b26, b1*a13 + a23*a13*a2 + a23*a2 + a23 + a13*a2*a1 + a13*a2 + a13*a1 +
        a13 + a2 + a1 + 1>,
    <67, b25, b1*a13*a2 + b1*a13 + a13*a2*a1 + a13*a1 + a13 + a2 + 1>,
    <68, b24, 1>,
    <69, b23, 0>,
    <70, b22, 1>,
    <71, b21, 0>,
    <72, b20, b1*a2 + a23*a2 + a13*a2*a1 + a13*a2 + a13 + a2 + 1>,
    <73, b19, b1*a2 + a23*a2 + a13*a2*a1 + a13*a2 + a13 + a2 + 1>,
    <74, b18, 1>,
    <75, b17, b1 + 1>,
    <76, b16, 1>,
    <77, b15, 0>,
    <78, b14, a13*a2*a1 + a13*a2 + a13*a1 + 1>,
    <79, b13, 0>,
    <80, b12, 1>,
    <81, b11, 1>,
    <82, b10, a13*a2*a1 + a13*a2 + a13*a1 + 1>,
    <83, b9, b1*a2 + a23*a2 + a13*a2*a1 + a13*a2 + a13 + a2 + 1>,
    <84, b8, b1>,
    <85, b7, 1>,
    <86, b6, b1>,
    <87, b5, 0>,
    <88, b4, b1>,
    <89, b3, b1>,
    <90, b2, 0>,
    <91, b1, b1>,
    <92, a28, a23>,
    <93, a27, a13>,
    <94, a26, 1>,
    <95, a25, a13>,
    <96, a24, 1>,
    <97, a23, a23>,
    <98, a22, 1>,
    <99, a21, a13 + 1>,
    <100, a20, 1>,
    <101, a19, 0>,
    <102, a18, 0>,
    <103, a17, a13>,
    <104, a16, a2>,
    <105, a15, 0>,
    <106, a14, a23*a2 + a13*a1 + a13 + a1>,
    <107, a13, a13>,
    <108, a12, a2>,
    <109, a11, 0>,
    <110, a10, 1>,
    <111, a9, 1>,
    <112, a8, a2>,
    <113, a7, 1>,
    <114, a6, a1>,
    <115, a5, 0>,
    <116, a4, a2>,
    <117, a3, 1>,
    <118, a2, a2>,
    <119, a1, a1>
]
\end{verbatim}

\section{Resulting relations, general case}
\label{genrelations}

Writing out the matrices defining $Sq^8$ and $Sq^{16}$ after
reducing to the minimal set of variables defining them is impractical
since some of the entries are rather lengthy sums, making the matrices
unreadable.   
Instead, we refer the reader to the inital versions with
all indeterminates in Appendices~\ref{gensq8} and \ref{gensq16},
and provide the definitions of those indeterminates in terms
of the $13$ variables that we have reduced to.

\begin{verbatim}

Writing each of the 150 variables in terms of the final 13 variables:
the entries are <index, name, value>.

\begin{verbatim}
[
    <150, a1, a1>,
    <149, a2, a2>,
    <148, a3, a3>,
    <147, a4, 1>,
    <146, a5, a2>,
    <145, a6, 0>,
    <144, a7, 0>,
    <143, a8, a1>,
    <142, a9, 1>,
    <141, a10, a2>,
    <140, a11, a3>,
    <139, a12, a3>,
    <138, a13, 1>,
    <137, a14, 1>,
    <136, a15, 0>,
    <135, a16, a3>,
    <134, a17, 0>,
    <133, a18, a2>,
    <132, a19, 1>,
    <131, a20, a3>,
    <130, a21, a21>,
    <129, a22, a60*a3 + a47*a2 + a21*a1 + a21 + a1>,
    <128, a23, a62*a60*a3*a2 + a62*a48*a2 + a62*a47*a2 + a62*a21*a3*a2 +
        a62*a21*a2 + a62*a3*a2 + a62*a3 + a62*a2*a1 + a62*a2 + a60*a48*a2 +
        a60*a47*a2 + a60*a21*a2 + a60*a3 + a60*a2*a1 + a48*a21*a2 + a47*a21*a2 +
        a47*a2 + a21*a2*a1 + a21*a2 + a21*a1 + a21 + a3 + a2 + 1>,
    <127, a24, a62*a60*a3*a2 + a62*a48*a2 + a62*a47*a2 + a62*a21*a3*a2 +
        a62*a21*a2 + a62*a3*a2 + a62*a2*a1 + a62*a2 + a60*a48*a2 + a60*a47*a2 +
        a60*a21*a2 + a60*a2*a1 + a48*a21*a2 + a47*a21*a2 + a21*a2*a1 + a21*a2 +
        a3 + 1>,
    <126, a25, 0>,
    <125, a26, a2>,
    <124, a27, 1>,
    <123, a28, 1>,
    <122, a29, a21>,
    <121, a30, 0>,
    <120, a31, 0>,
    <119, a32, a62*a60*a3*a2 + a62*a48*a2 + a62*a47*a2 + a62*a21*a3*a2 +
        a62*a21*a2 + a62*a3*a2 + a62*a2*a1 + a62*a2 + a60*a48*a2 + a60*a47*a2 +
        a60*a21*a2 + a60*a2*a1 + a48*a21*a2 + a47*a21*a2 + a21*a2*a1 + a21*a2 +
        a3 + 1>,
    <118, a33, 0>,
    <117, a34, 0>,
    <116, a35, 1>,
    <115, a36, 1>,
    <114, a37, a3 + 1>,
    <113, a38, 1>,
    <112, a39, a21 + 1>,
    <111, a40, a62*a60*a3*a2 + a62*a48*a2 + a62*a47*a2 + a62*a21*a3*a2 +
        a62*a21*a2 + a62*a3*a2 + a62*a3 + a62*a2*a1 + a62*a2 + a60*a48*a2 +
        a60*a47*a2 + a60*a21*a2 + a60*a3 + a60*a2*a1 + a48*a21*a2 + a47*a21*a2 +
        a47*a2 + a21*a2*a1 + a21*a2 + a21*a1 + a21 + a3 + a2 + 1>,
    <110, a41, a62*a60*a3*a2 + a62*a48*a2 + a62*a47*a2 + a62*a21*a3*a2 +
        a62*a21*a2 + a62*a3*a2 + a62*a3 + a62*a2*a1 + a62*a2 + a60*a48*a2 +
        a60*a47*a2 + a60*a21*a2 + a60*a3 + a60*a2*a1 + a48*a21*a2 + a47*a21*a2 +
        a47*a2 + a21*a2*a1 + a21*a2 + a21*a1 + a21 + a3 + a2 + 1>,
    <109, a42, 0>,
    <108, a43, 1>,
    <107, a44, 1>,
    <106, a45, 0>,
    <105, a46, 1>,
    <104, a47, a47>,
    <103, a48, a48>,
    <102, a49, a62*a60*a3*a2 + a62*a60*a2 + a62*a48*a2 + a62*a21*a3*a2 +
        a62*a21*a3 + a62*a21*a1 + a62*a21 + a62*a2*a1 + a62 + a60*a48*a2 +
        a60*a21*a3 + a60*a21*a2 + a60*a21*a1 + a60*a21 + a60*a3*a2 + a60*a3 +
        a60*a2*a1 + a60*a2 + a60 + a48*a21*a2 + a48 + a47 + a21*a3 + a21*a2 +
        a3*a2 + a3 + a1>,
    <101, a50, a62*a60*a3 + a62*a48 + a62*a47 + a62*a21*a3 + a62*a21 + a62*a3*a2
        + a62*a2 + a62*a1 + a60*a48 + a60*a47 + a60*a21 + a60*a2 + a60*a1 + a60
        + a48*a21 + a48*a2 + a48 + a47*a21 + a47*a2 + a47 + a21*a1 + a21 + a2*a1
        + a2 + a1 + 1>,
    <100, a51, 1>,
    <99, a52, a62*a60*a3*a2 + a62*a48*a2 + a62*a47*a2 + a62*a21*a3*a2 +
        a62*a21*a2 + a62*a3*a2 + a62*a3 + a62*a2*a1 + a62*a2 + a60*a48*a2 +
        a60*a47*a2 + a60*a21*a2 + a60*a3 + a60*a2*a1 + a48*a21*a2 + a47*a21*a2 +
        a47*a2 + a21*a2*a1 + a21*a2 + a21*a1 + a21 + a3 + a2 + 1>,
    <98, a53, 0>,
    <97, a54, a21>,
    <96, a55, a62*a60*a3*a2 + a62*a48*a2 + a62*a47*a2 + a62*a21*a3*a2 +
        a62*a21*a2 + a62*a3*a2 + a62*a3 + a62*a2*a1 + a62*a2 + a60*a48*a2 +
        a60*a47*a2 + a60*a21*a2 + a60*a3 + a60*a2*a1 + a48*a21*a2 + a47*a21*a2 +
        a47*a2 + a21*a2*a1 + a21*a2 + a21*a1 + a21 + a3 + a2 + 1>,
    <95, a56, 1>,
    <94, a57, 1>,
    <93, a58, a62*a60*a3*a2 + a62*a60*a2 + a62*a48*a2 + a62*a21*a3*a2 +
        a62*a21*a3 + a62*a21*a1 + a62*a21 + a62*a2*a1 + a62 + a60*a48*a2 +
        a60*a21*a3 + a60*a21*a2 + a60*a21*a1 + a60*a21 + a60*a3*a2 + a60*a3 +
        a60*a2*a1 + a60*a2 + a60 + a48*a21*a2 + a48 + a47 + a21*a3 + a21*a2 +
        a3*a2 + a3 + a1>,
    <92, a59, 0>,
    <91, a60, a60>,
    <90, a61, a61>,
    <89, a62, a62>,
    <88, a63, a21>,
    <87, a64, 1>,
    <86, a65, 1>,
    <85, a66, 0>,
    <84, a67, a47>,
    <83, a68, a62*a60*a3*a2 + a62*a60*a2 + a62*a48*a2 + a62*a21*a3*a2 +
        a62*a21*a3 + a62*a21*a1 + a62*a21 + a62*a2*a1 + a62 + a60*a48*a2 +
        a60*a21*a3 + a60*a21*a2 + a60*a21*a1 + a60*a21 + a60*a3*a2 + a60*a3 +
        a60*a2*a1 + a60*a2 + a60 + a48*a21*a2 + a47 + a21*a3 + a21*a2 + a3*a2 +
        a3 + a1>,
    <82, a69, a62*a60*a3 + a62*a48 + a62*a47 + a62*a21*a3 + a62*a21 + a62*a3*a2
        + a62*a2 + a62*a1 + a60*a48 + a60*a47 + a60*a21 + a60*a2 + a60*a1 + a60
        + a48*a21 + a48*a2 + a48 + a47*a21 + a47*a2 + a47 + a21*a1 + a21 + a2*a1
        + a2 + a1 + 1>,
    <81, a70, a62*a60*a3 + a62*a48 + a62*a47 + a62*a21*a3 + a62*a21 + a62*a3*a2
        + a62*a2 + a62*a1 + a60*a48 + a60*a47 + a60*a21 + a60*a2 + a60*a1 + a60
        + a48*a21 + a48*a2 + a48 + a47*a21 + a47*a2 + a47 + a21*a1 + a21 + a2*a1
        + a2 + a1>,
    <80, a71, a60>,
    <79, a72, 1>,
    <78, a73, a62>,
    <77, a74, 0>,
    <76, a75, a62*a60*a3*a2 + a62*a48*a2 + a62*a47*a2 + a62*a21*a3*a2 +
        a62*a21*a2 + a62*a3*a2 + a62*a3 + a62*a2*a1 + a62*a2 + a60*a48*a2 +
        a60*a47*a2 + a60*a21*a2 + a60*a3 + a60*a2*a1 + a48*a21*a2 + a47*a21*a2 +
        a47*a2 + a21*a2*a1 + a21*a2 + a21*a1 + a21 + a3 + a2>,
    <75, a76, 1>,
    <74, a77, 0>,
    <73, a78, a62*a60*a3*a2 + a62*a60*a2 + a62*a48*a2 + a62*a21*a3*a2 +
        a62*a21*a3 + a62*a21*a1 + a62*a21 + a62*a2*a1 + a62 + a60*a48*a2 +
        a60*a21*a3 + a60*a21*a2 + a60*a21*a1 + a60*a21 + a60*a3*a2 + a60*a3 +
        a60*a2*a1 + a60*a2 + a60 + a48*a21*a2 + a47 + a21*a3 + a21*a2 + a3*a2 +
        a3 + a1>,
    <72, a79, 0>,
    <71, a80, a62*a60*a3 + a62*a48 + a62*a47 + a62*a21*a3 + a62*a21 + a62*a3*a2
        + a62*a2 + a62*a1 + a60*a48 + a60*a47 + a60*a21 + a60*a2 + a60*a1 + a60
        + a48*a21 + a48*a2 + a48 + a47*a21 + a47*a2 + a47 + a21*a1 + a21 + a2*a1
        + a2 + a1>,
    <70, a81, a48>,
    <69, a82, 1>,
    <68, a83, a62*a60*a3 + a62*a48 + a62*a47 + a62*a21*a3 + a62*a21 + a62*a3*a2
        + a62*a2 + a62*a1 + a60*a48 + a60*a47 + a60*a21 + a60*a2 + a60*a1 + a60
        + a48*a21 + a48*a2 + a48 + a47*a21 + a47*a2 + a47 + a21*a1 + a21 + a2*a1
        + a2 + a1 + 1>,
    <67, a84, a61>,
    <66, a85, 1>,
    <65, a86, a62>,
    <64, a87, 1>,
    <63, a88, 0>,
    <62, a89, a62*a60*a3*a2 + a62*a60*a3 + a62*a48*a2 + a62*a48 + a62*a47*a2 +
        a62*a47 + a62*a21*a3*a2 + a62*a21*a3 + a62*a21*a2 + a62*a21 + a62*a3 +
        a62*a2*a1 + a62*a1 + a60*a48*a2 + a60*a48 + a60*a47*a2 + a60*a47 +
        a60*a21*a2 + a60*a21 + a60*a3 + a60*a2*a1 + a60*a2 + a60*a1 + a60 +
        a48*a21*a2 + a48*a21 + a48*a2 + a47*a21*a2 + a47*a21 + a21*a2*a1 +
        a21*a2 + a21 + a2*a1 + a2>,
    <61, a90, a62*a60*a3 + a62*a60*a2 + a62*a48 + a62*a47*a2 + a62*a47 +
        a62*a21*a2 + a62*a21*a1 + a62*a1 + a62 + a60*a48 + a60*a47*a2 + a60*a47
        + a60*a21*a3 + a60*a21*a1 + a60*a3*a2 + a60*a1 + a48*a21 + a48*a2 +
        a47*a21*a2 + a47*a21 + a21*a3 + a21*a2*a1 + a3*a2 + a3 + a2*a1 + 1>,
    <60, a91, 1>,
    <59, a92, 1>,
    <58, a93, 1>,
    <57, a94, 1>,
    <56, a95, 0>,
    <55, a96, a62>,
    <54, a97, 0>,
    <53, a98, a62*a60*a3 + a62*a60*a2 + a62*a48 + a62*a47*a2 + a62*a47 +
        a62*a21*a2 + a62*a21*a1 + a62*a1 + a62 + a60*a48 + a60*a47*a2 + a60*a47
        + a60*a21*a3 + a60*a21*a1 + a60*a3*a2 + a60*a1 + a48*a21 + a48*a2 +
        a47*a21*a2 + a47*a21 + a21*a3 + a21*a2*a1 + a3*a2 + a3 + a2*a1 + 1>,
    <52, a99, 1>,
    <51, a100, a62*a60*a3 + a62*a48 + a62*a47 + a62*a21*a3 + a62*a21 + a62*a3*a2
        + a62*a2 + a62*a1 + a60*a48 + a60*a47 + a60*a21 + a60*a2 + a60*a1 + a60
        + a48*a21 + a48*a2 + a48 + a47*a21 + a47*a2 + a47 + a21*a1 + a21 + a2*a1
        + a2 + a1>,
    <50, a101, a62 + a60 + a21 + a2>,
    <49, a102, a62 + a61 + a60*a3 + a47*a2 + a21*a1 + a21 + a2>,
    <48, a103, 0>,
    <47, a104, a62>,
    <46, a105, a62*a60*a3*a2 + a62*a60*a3 + a62*a48*a2 + a62*a48 + a62*a47*a2 +
        a62*a47 + a62*a21*a3*a2 + a62*a21*a3 + a62*a21*a2 + a62*a21 + a62*a3 +
        a62*a2*a1 + a62*a1 + a60*a48*a2 + a60*a48 + a60*a47*a2 + a60*a47 +
        a60*a21*a2 + a60*a21 + a60*a3 + a60*a2*a1 + a60*a2 + a60*a1 + a60 +
        a48*a21*a2 + a48*a21 + a48*a2 + a47*a21*a2 + a47*a21 + a21*a2*a1 +
        a21*a2 + a21 + a2*a1 + a2>,
    <45, a106, a62*a60*a3*a2 + a62*a60*a3 + a62*a48*a2 + a62*a48 + a62*a47*a2 +
        a62*a47 + a62*a21*a3*a2 + a62*a21*a3 + a62*a21*a2 + a62*a21 + a62*a3 +
        a62*a2*a1 + a62*a1 + a60*a48*a2 + a60*a48 + a60*a47*a2 + a60*a47 +
        a60*a21*a2 + a60*a21 + a60*a3 + a60*a2*a1 + a60*a2 + a60*a1 + a60 +
        a48*a21*a2 + a48*a21 + a48*a2 + a47*a21*a2 + a47*a21 + a21*a2*a1 +
        a21*a2 + a21 + a2*a1 + a2 + 1>,
    <44, a107, 1>,
    <43, a108, 0>,
    <42, a109, a62 + a60 + a21 + a2>,
    <41, a110, 0>,
    <40, a111, 0>,
    <39, a112, 1>,
    <38, a113, a62*a60*a3*a2 + a62*a60*a3 + a62*a48*a2 + a62*a48 + a62*a47*a2 +
        a62*a47 + a62*a21*a3*a2 + a62*a21*a3 + a62*a21*a2 + a62*a21 + a62*a3 +
        a62*a2*a1 + a62*a1 + a60*a48*a2 + a60*a48 + a60*a47*a2 + a60*a47 +
        a60*a21*a2 + a60*a21 + a60*a3 + a60*a2*a1 + a60*a2 + a60*a1 + a60 +
        a48*a21*a2 + a48*a21 + a48*a2 + a47*a21*a2 + a47*a21 + a21*a2*a1 +
        a21*a2 + a21 + a2*a1 + a2 + 1>,
    <37, a114, a62*a60*a3*a2 + a62*a60*a3 + a62*a48*a2 + a62*a48 + a62*a47*a2 +
        a62*a47 + a62*a21*a3*a2 + a62*a21*a3 + a62*a21*a2 + a62*a21 + a62*a3 +
        a62*a2*a1 + a62*a1 + a60*a48*a2 + a60*a48 + a60*a47*a2 + a60*a47 +
        a60*a21*a2 + a60*a21 + a60*a3 + a60*a2*a1 + a60*a2 + a60*a1 + a60 +
        a48*a21*a2 + a48*a21 + a48*a2 + a47*a21*a2 + a47*a21 + a21*a2*a1 +
        a21*a2 + a21 + a2*a1 + a2>,
    <36, a115, a62 + a60 + a21 + a2 + 1>,
    <35, a116, 1>,
    <34, a117, a61 + a60 + a47 + a1>,
    <33, a118, 1>,
    <32, a119, 1>,
    <31, a120, a62 + a60 + a21 + a2>,
    <30, a121, 1>,
    <29, a122, a62*a60*a3*a2 + a62*a60*a3 + a62*a48*a2 + a62*a48 + a62*a47*a2 +
        a62*a47 + a62*a21*a3*a2 + a62*a21*a3 + a62*a21*a2 + a62*a21 + a62*a3 +
        a62*a2*a1 + a62*a1 + a60*a48*a2 + a60*a48 + a60*a47*a2 + a60*a47 +
        a60*a21*a2 + a60*a21 + a60*a3 + a60*a2*a1 + a60*a2 + a60*a1 + a60 +
        a48*a21*a2 + a48*a21 + a48*a2 + a47*a21*a2 + a47*a21 + a21*a2*a1 +
        a21*a2 + a21 + a2*a1 + a2 + 1>,
    <28, a123, a62 + a60 + a21 + a2>,
    <27, a124, a61 + a60 + a47 + a1>,
    <26, b1, b1>,
    <25, b2, b2>,
    <24, b3, b3>,
    <23, b4, b4>,
    <22, b5, b2 + a3 + a1>,
    <21, b6, 0>,
    <20, b7, b4>,
    <19, b8, a3 + a2 + a1 + 1>,
    <18, b9, b4 + b3 + a3 + a1>,
    <17, b10, 0>,
    <16, b11, b4 + b3 + a2 + a1>,
    <15, b12, a2 + 1>,
    <14, b13, b3 + a2 + a1>,
    <13, b14, b3 + b1 + a62*a60*a3*a2 + a62*a60*a3 + a62*a48*a2 + a62*a48 +
        a62*a47*a2 + a62*a47 + a62*a21*a3*a2 + a62*a21*a3 + a62*a21*a2 + a62*a21
        + a62*a3*a2 + a62*a2*a1 + a60*a48*a2 + a60*a48 + a60*a47*a2 + a60*a47 +
        a60*a21*a2 + a60*a21 + a60*a3*a2 + a60*a2*a1 + a60*a2 + a60 + a48*a21*a2
        + a48*a21 + a47*a21*a2 + a47*a21 + a21*a3 + a21*a2 + a21*a1 + a21 +
        a3*a2 + a3 + a2*a1 + 1>,
    <12, b15, b4 + b2 + a62*a60*a3*a2 + a62*a60*a3 + a62*a48*a2 + a62*a48 +
        a62*a47*a2 + a62*a47 + a62*a21*a3*a2 + a62*a21*a3 + a62*a21*a2 + a62*a21
        + a62*a3*a2 + a62*a3 + a62*a2*a1 + a62*a1 + a60*a48*a2 + a60*a48 +
        a60*a47*a2 + a60*a47 + a60*a21*a2 + a60*a21 + a60*a3*a2 + a60*a3 +
        a60*a2*a1 + a60*a2 + a60*a1 + a60 + a48*a21*a2 + a48*a21 + a47*a21*a2 +
        a47*a21 + a21*a2 + a21 + a2 + a1>,
    <11, b16, a2 + 1>,
    <10, b17, 0>,
    <9, b18, b4 + b2 + a62*a60*a3*a2 + a62*a60*a3 + a62*a48*a2 + a62*a48 +
        a62*a47*a2 + a62*a47 + a62*a21*a3*a2 + a62*a21*a3 + a62*a21*a2 + a62*a21
        + a62*a3*a2 + a62*a3 + a62*a2*a1 + a62*a1 + a60*a48*a2 + a60*a48 +
        a60*a47*a2 + a60*a47 + a60*a21*a2 + a60*a21 + a60*a3*a2 + a60*a3 +
        a60*a2*a1 + a60*a2 + a60*a1 + a60 + a48*a21*a2 + a48*a21 + a47*a21*a2 +
        a47*a21 + a21*a2 + a21 + a2 + a1>,
    <8, b19, a3>,
    <7, b20, b4 + b3 + b2 + b1 + a62*a1 + a60*a3 + a60*a1 + a21*a3 + a21*a1 +
        a3*a2 + a3 + a2*a1>,
    <6, b21, 1>,
    <5, b22, b2 + b1 + a62*a3 + a62*a1 + a60*a1 + a21*a1 + a3 + a2*a1 + 1>,
    <4, b23, b4 + b3 + b2 + b1 + a62*a60*a3*a2 + a62*a48*a2 + a62*a47*a2 +
        a62*a21*a3*a2 + a62*a21*a2 + a62*a3*a2 + a62*a3 + a62*a2*a1 + a62*a2 +
        a62*a1 + a60*a48*a2 + a60*a47*a2 + a60*a21*a2 + a60*a2*a1 + a60*a1 +
        a48*a21*a2 + a47*a21*a2 + a47*a2 + a21*a3 + a21*a2*a1 + a21*a2 + a3*a2 +
        a2*a1 + a2>,
    <3, b24, b3 + b1 + a62*a60*a3*a2 + a62*a60*a3 + a62*a48*a2 + a62*a48 +
        a62*a47*a2 + a62*a47 + a62*a21*a3*a2 + a62*a21*a3 + a62*a21*a2 + a62*a21
        + a62*a3*a2 + a62*a2*a1 + a60*a48*a2 + a60*a48 + a60*a47*a2 + a60*a47 +
        a60*a21*a2 + a60*a21 + a60*a3*a2 + a60*a3 + a60*a2*a1 + a60*a2 + a60 +
        a48*a21*a2 + a48*a21 + a47*a21*a2 + a47*a21 + a47*a2 + a21*a3 + a21*a2 +
        a21 + a3*a2 + a3 + a2*a1 + a1>,
    <2, b25, b2 + b1 + a62*a60*a3 + a62*a48 + a62*a47 + a62*a21*a3 + a62*a21 +
        a62*a3*a2 + a62*a3 + a62*a2 + a60*a48 + a60*a47 + a60*a21 + a60*a2 + a60
        + a48*a21 + a48*a2 + a47*a21 + a47*a2 + a21 + a3 + a2 + a1>,
    <1, b26, b1 + a62*a1 + a61*a2 + a60*a3 + a60*a2 + a60*a1>
]
\end{verbatim}

\section{Self dual actions, symmetric case}
\label{app:selfdualsym}

\begin{longtable}{>{$}c<{$} >{$}c<{$} >{$}c<{$} >{$}c<{$} >{$}c<{$} |
>{$}c<{$} >{$}c<{$} >{$}c<{$} >{$}c<{$}}
\caption{Self dual symmetric $\cA$-module structures on $\cA(2)$}\\
\hline
a_1 & a_2 & a_{13} & a_{23} & b_{1} &
 c_1 & d_1 & d_2 & d_3\\
\hline
\endfirsthead
\caption{Self dual symmetric $\cA$-module structures on $\cA(2)$ (cont.)}\\
\hline
a_1 & a_2 & a_{13} & a_{23} & b_{1} &
 c_1 & d_1 & d_2 & d_3\\
\hline
\endhead
\hline
\endfoot

0 & 0 & 1 & 1 & 0 &  1 & 0 & 0 & 0\\
0 & 0 & 1 & 1 & 0 &  1 & 0 & 0 & 1\\
0 & 0 & 1 & 1 & 0 &  1 & 1 & 1 & 0\\
0 & 0 & 1 & 1 & 0 &  1 & 1 & 1 & 1\\
\hline
1 & 0 & 1 & 0 & 1 &  0 & 0 & 0 & 0\\
1 & 0 & 1 & 0 & 1 &  0 & 0 & 0 & 1\\
1 & 0 & 1 & 0 & 1 &  0 & 1 & 1 & 0\\
1 & 0 & 1 & 0 & 1 &  0 & 1 & 1 & 1\\
\hline
1 & 1 & 0 & 0 & 0 &  0 & 0 & 1 & 0\\
1 & 1 & 0 & 0 & 0 &  0 & 0 & 1 & 1\\
1 & 1 & 0 & 0 & 0 &  0 & 1 & 0 & 0\\
1 & 1 & 0 & 0 & 0 &  0 & 1 & 0 & 1\\
\hline
1 & 1 & 0 & 0 & 1 &  0 & 0 & 0 & 0\\
1 & 1 & 0 & 0 & 1 &  0 & 0 & 0 & 1\\
1 & 1 & 0 & 0 & 1 &  0 & 1 & 1 & 0\\
1 & 1 & 0 & 0 & 1 &  0 & 1 & 1 & 1\\

\end{longtable}

\section{Self dual actions, general case}
\label{app:selfdualgen}

\begin{longtable}{>{$}c<{$} >{$}c<{$} >{$}c<{$} >{$}c<{$} >{$}c<{$} 
>{$}c<{$} >{$}c<{$} >{$}c<{$} >{$}c<{$}  |
>{$}c<{$} >{$}c<{$} >{$}c<{$} >{$}c<{$}}
\caption{Self dual $\cA$-module structures on $\cA(2)$}\\
\hline
a_1 & a_2 & a_3 & a_{21} & a_{47} & a_{48} & a_{60} &
a_{61} & a_{62} & b_1 & b_2 & b_3 & b_4\\
\hline
\endfirsthead
\caption{Self dual $\cA$-module structures on $\cA(2)$ (cont.)}\\
\hline
a_1 & a_2 & a_3 & a_{21} & a_{47} & a_{48} & a_{60} &
a_{61} & a_{62} & b_1 & b_2 & b_3 & b_4\\
\hline
\endhead
\hline
\endfoot

0 & 0 & 0 & 0 & 0 & 0 & 0 & 1 & 1 &  1 & 0 & 1 & 0\\
0 & 0 & 0 & 0 & 0 & 0 & 0 & 1 & 1 &  1 & 0 & 1 & 1\\
0 & 0 & 0 & 0 & 0 & 0 & 0 & 1 & 1 &  1 & 1 & 0 & 0\\
0 & 0 & 0 & 0 & 0 & 0 & 0 & 1 & 1 &  1 & 1 & 0 & 1\\
\hline
0 & 0 & 0 & 1 & 1 & 0 & 0 & 0 & 0 &  1 & 0 & 0 & 0\\
0 & 0 & 0 & 1 & 1 & 0 & 0 & 0 & 0 &  1 & 0 & 0 & 1\\
0 & 0 & 0 & 1 & 1 & 0 & 0 & 0 & 0 &  1 & 1 & 1 & 0\\
0 & 0 & 0 & 1 & 1 & 0 & 0 & 0 & 0 &  1 & 1 & 1 & 1\\
\hline
0 & 0 & 1 & 0 & 0 & 0 & 0 & 1 & 1 &  1 & 0 & 0 & 0\\
0 & 0 & 1 & 0 & 0 & 0 & 0 & 1 & 1 &  1 & 0 & 0 & 1\\
0 & 0 & 1 & 0 & 0 & 0 & 0 & 1 & 1 &  1 & 1 & 1 & 0\\
0 & 0 & 1 & 0 & 0 & 0 & 0 & 1 & 1 &  1 & 1 & 1 & 1\\
\hline
0 & 1 & 1 & 1 & 0 & 0 & 0 & 1 & 0 &  0 & 0 & 0 & 0\\
0 & 1 & 1 & 1 & 0 & 0 & 0 & 1 & 0 &  0 & 0 & 0 & 1\\
0 & 1 & 1 & 1 & 0 & 0 & 0 & 1 & 0 &  0 & 1 & 1 & 0\\
0 & 1 & 1 & 1 & 0 & 0 & 0 & 1 & 0 &  0 & 1 & 1 & 1\\
\hline
1 & 0 & 0 & 0 & 1 & 0 & 0 & 0 & 1 &  0 & 0 & 1 & 0\\
1 & 0 & 0 & 0 & 1 & 0 & 0 & 0 & 1 &  0 & 0 & 1 & 1\\
1 & 0 & 0 & 0 & 1 & 0 & 0 & 0 & 1 &  0 & 1 & 0 & 0\\
1 & 0 & 0 & 0 & 1 & 0 & 0 & 0 & 1 &  0 & 1 & 0 & 1\\
\hline
1 & 0 & 1 & 0 & 1 & 0 & 0 & 0 & 1 &  0 & 0 & 0 & 0\\
1 & 0 & 1 & 0 & 1 & 0 & 0 & 0 & 1 &  0 & 0 & 0 & 1\\
1 & 0 & 1 & 0 & 1 & 0 & 0 & 0 & 1 &  0 & 1 & 1 & 0\\
1 & 0 & 1 & 0 & 1 & 0 & 0 & 0 & 1 &  0 & 1 & 1 & 1\\
\hline
1 & 0 & 1 & 1 & 0 & 0 & 0 & 1 & 0 &  0 & 0 & 0 & 0\\
1 & 0 & 1 & 1 & 0 & 0 & 0 & 1 & 0 &  0 & 0 & 0 & 1\\
1 & 0 & 1 & 1 & 0 & 0 & 0 & 1 & 0 &  0 & 1 & 1 & 0\\
1 & 0 & 1 & 1 & 0 & 0 & 0 & 1 & 0 &  0 & 1 & 1 & 1\\
\hline
1 & 1 & 0 & 0 & 0 & 0 & 0 & 1 & 1 &  0 & 0 & 1 & 0\\
1 & 1 & 0 & 0 & 0 & 0 & 0 & 1 & 1 &  0 & 0 & 1 & 1\\
1 & 1 & 0 & 0 & 0 & 0 & 0 & 1 & 1 &  0 & 1 & 0 & 0\\
1 & 1 & 0 & 0 & 0 & 0 & 0 & 1 & 1 &  0 & 1 & 0 & 1\\
\hline
1 & 1 & 1 & 0 & 0 & 0 & 0 & 1 & 1 &  0 & 0 & 0 & 0\\
1 & 1 & 1 & 0 & 0 & 0 & 0 & 1 & 1 &  0 & 0 & 0 & 1\\
1 & 1 & 1 & 0 & 0 & 0 & 0 & 1 & 1 &  0 & 1 & 1 & 0\\
1 & 1 & 1 & 0 & 0 & 0 & 0 & 1 & 1 &  0 & 1 & 1 & 1\\
\hline
1 & 1 & 1 & 1 & 1 & 0 & 0 & 0 & 0 &  0 & 0 & 0 & 0\\
1 & 1 & 1 & 1 & 1 & 0 & 0 & 0 & 0 &  0 & 0 & 0 & 1\\
1 & 1 & 1 & 1 & 1 & 0 & 0 & 0 & 0 &  0 & 1 & 1 & 0\\
1 & 1 & 1 & 1 & 1 & 0 & 0 & 0 & 0 &  0 & 1 & 1 & 1\\

\end{longtable}

\section{MAGMA code, symmetric case}
\label{codesym}

The first two steps in the symmetric case are taken by this MAGMA
code.   Note that the code was run `incrementally', executing 
larger and larger initial segments of it, so that, for example,
the reference to the fact that Basis(Rel) contains 496 elements
was added after having run the code up to that point.  Similarly,
the comment that N1 = 28 and N2 = 66 was added after having run the code far 
enough to have computed these values.

\begin{verbatim}
/*
Symmetric case
*/

/* Reversing the order of the variables so that relations will 
get reduced to the earliest instance of each.
This is accomplished by reversing the AssignNames list of names,
and by reversing the assignment of R.i's to matrix entries by
using R.(N+1-i) instead.
*/

/* 
A = A(2).   Write as B + BQ_2, w. the v.s. B spanned by
the (r1,r2) and BQ_2 spanned by the (r1,r2,1).
Write A_n, B_n, Q_n = B_{n-7}Q_2 for the degree n subspaces.
*/

/* 
Then Sq^8 : A_n ---> A_n+8 can be written as a block matrix
B_n + Q_n ---> B_n+8 + Q_n+8 in the form
[ S_n  T_n     ]
[ 0    S_{n-7} ]
with S_n = B_n --> A_n --Sq^8--> A_n+8 --> B_n+8 and
 and T_n = B_n --> A_n --Sq^8--> A_n+8 --> Q_n+8 =B_n+1
*/

/*
There are N1 = 28 and N2 = 66 indeterminates required to
describe S = {S_n} and T = {T_n}, respectively.
(Calculated below)
*/


XBbas := [&cat[[ [n-3*j-7*k,j,k]
             : j in [0..3] | n-3*j-7*k ge 0 and n-3*j-7*k le 7]
             : k in [0..0]]
             : n in [0..39]];

XQbas := [&cat[[ [n-3*j-7*k,j,k]
             : j in [0..3] | n-3*j-7*k ge 0 and n-3*j-7*k le 7]
             : k in [1..1]]
             : n in [0..39]];

function B_bas(j)
  if j ge 0 and j+1 le #XBbas then
    return XBbas[j+1];
  else
    return [];
  end if;
end function;

function Q_bas(j)
  if j ge 0 and j+1 le #XQbas then
    return XQbas[j+1];
  else
    return [];
  end if;
end function;

function A_bas(j)
  return B_bas(j) cat Q_bas(j);
end function;

N1 :=  &+[#B_bas(j)*#B_bas(j+8) : j in [0..16-8]];
N2 :=  &+[#B_bas(j)*#B_bas(j+1) : j in [0..16-1]];

/* N1 := 28;   N2 := 66;
*/

/* 
Similarly Sq^16 : A_n ---> A_n+16 can be written as a block matrix
B_n + Q_n ---> B_n+16 + Q_n+16 in the form
[ S_n  T_n     ]
[ 0    S_{n-7} ]
with S_n = B_n --> A_n --Sq^16--> A_n+16 --> B_n+16 and
 and T_n = B_n --> A_n --Sq^16--> A_n+16 --> Q_n+16 =B_n+9
*/

/* Compute Sq^16 separately, later, after using degree 1 relations
implied by Adem relations among the first N1+N2 variables to simplify
the Sq^i for i < 15.
*/

M1 :=  1;
M2 :=  &+[#B_bas(j)*#B_bas(j+9) : j in [0..16-9]];

/*
M1 := 1;  M2 := 24;
*/

N := N1+N2+M1+M2;


R := PolynomialRing(GF(2),N);
AssignNames(~R,
        Reverse(["a" cat IntegerToString(i) : i in [1..N1]] cat
                ["b" cat IntegerToString(i-N1) : i in [N1+1..N1+N2]] cat
                ["c" cat IntegerToString(i-N1-N2) : i in [N1+N2+1..N1+N2+M1]] cat
                ["d" cat IntegerToString(i-N1-N2-M1) : i in [N1+N2+M1+1..N]]));

/* Define Sq^a action on Milnor basis element Sq(r1,r2,r3)
assuming that a < 8.)
*/

function MSq(a,r)
  return &cat [ [ [a+r[1]-3*i-4*j, r[2]+i-j, r[3]+j]
                 : i in [0..Min(r[1],Truncate((a-4*j)/2))] 
		   | a-2*i-4*j ge 0 and
		     IsOdd(Binomial(a+r[1]-3*i-4*j,r[1]-i)) and
		     IsOdd(Binomial(r[2]+i-j,i)) and
		     IsOdd(Binomial(r[3]+j,j)) ]
		 : j in [0..Min(r[2],Truncate(a/4))] ];
end function;

function In(x,L)
  if x in L then return 1; else return 0; end if;
end function;


/* 
Sq(i,j) = XSq[i+1,j+1] is Sq^i from degree j to i+j
*/

XSq := [ [* Matrix(R,#A_bas(j),#A_bas(i+j),
                  [ R!0 : k in [1..#A_bas(j)*#A_bas(i+j) ]])
         : j in [0..24] *]
	 : i in [0..23] ];

for i in [0..7] do
for j in [0..23] do
   XSq[i+1,j+1] := Matrix(R,#A_bas(j),#A_bas(i+j),
                   [[In(b,MSq(i,r)) : b in A_bas(i+j)] : r in A_bas(j)] );
end for;
end for;

/*
Define this AFTER computing the entries
*/

function Sq(i,j)
  return XSq[i+1,j+1];
end function;

/* Test Adem relations in this range
Sanity check:   expect no output.
*/

for n in [0..23] do
for b in [1..7] do
for a in [1..2*b-1] do
/* check that all the terms needed are defined	*/
if n+a+b le 23  and
   &and [ IsEven(Binomial(b-j-1,a-2*j)) or
             (a+b-j le 7 and j le 7)
	   : j in [0..Truncate(a/2)] ] then

  M := Sq(b,n)*Sq(a,n+b);
  for j in [0..Truncate(a/2)] do
     if IsOdd(Binomial(b-j-1,a-2*j)) then
        M +:= Sq(j,n)*Sq(a+b-j,n+j);
     end if;
  end for;
  if not IsZero(M) then
    print "Wrong: ",a,b,M;
  end if;

end if;
end for;
end for;
end for;

/* Define Sq^8
*/

XS := [* Matrix(R,#B_bas(j),#B_bas(8+j),
                  [ R!0 : k in [1..#B_bas(j)*#B_bas(8+j) ]])
         : j in [0..24] *];

last := 0;
for j in [0..24-8] do
  next := last+#B_bas(j)*#B_bas(8+j);
  XS[j+1] := Matrix(R,#B_bas(j),#B_bas(8+j),
                  [ R.(N+1-i) : i in [last+1..next] ]);
  last := next;
end for;


XT := [* Matrix(R,#B_bas(j),#B_bas(1+j),
                  [ R!0 : k in [1..#B_bas(j)*#B_bas(1+j) ]])
         : j in [0..24] *];

last := N1;
for j in [0..24-8] do
  next := last+#B_bas(j)*#B_bas(1+j);
  XT[j+1] := Matrix(R,#B_bas(j),#B_bas(1+j),
                  [ R.(N+1-i) : i in [last+1..next] ]);
  last := next;
end for;

/*
Now assemble the blocks into a single matrix
*/

for n in [0..15] do
XSq[9][n+1] := Matrix(R,#A_bas(n),#A_bas(n+8),
  [[XS[n+1][i,j] : j in [1..#B_bas(n+8)]] cat
   [XT[n+1][i,j] : j in [1..#Q_bas(n+8)]] 
    : i in [1..#B_bas(n)] ] cat
  [[ R!0         : j in [1..#B_bas(n+8)]] cat
   [XS[n-6][i,j] : j in [1..#Q_bas(n+8)]] 
    : i in [1..#Q_bas(n)] ]);
end for;

printf "\nInitial Sq^8:\n%o\n",XSq[9];


/*
Redefine Sq
*/

function Sq(i,j)
  return XSq[i+1,j+1];
end function;



/* Sq^9 = Sq^1 Sq^8
*/

for j in [0..14] do
  XSq[10,j+1] := Sq(8,j)*Sq(1,j+8);
end for;

/*
Redefine Sq
*/

function Sq(i,j)
  return XSq[i+1,j+1];
end function;


/* Sq^10 = Sq^2 Sq^8 + Sq^9 Sq^1
*/

for j in [0..13] do
  XSq[11,j+1] := Sq(8,j)*Sq(2,j+8) + Sq(1,j)*Sq(9,j+1);
end for;

/*
Redefine Sq
*/

function Sq(i,j)
  return XSq[i+1,j+1];
end function;


/* Sq^11 = Sq^1 Sq^10
*/

for j in [0..12] do
  XSq[12,j+1] := Sq(10,j)*Sq(1,j+10);
end for;

/*
Redefine Sq
*/

function Sq(i,j)
  return XSq[i+1,j+1];
end function;

/* Sq^12 = Sq^4 Sq^8 + Sq^11 Sq^1 + Sq^10 Sq^2
*/

for j in [0..11] do
  XSq[13,j+1] := Sq(8,j)*Sq(4,j+8) + Sq(1,j)*Sq(11,j+1) + Sq(2,j)*Sq(10,j+2);
end for;

/*
Redefine Sq
*/

function Sq(i,j)
  return XSq[i+1,j+1];
end function;


/* Sq^13 = Sq^1 Sq^12
*/

for j in [0..10] do
  XSq[14,j+1] := Sq(12,j)*Sq(1,j+12);
end for;

/*
Redefine Sq
*/

function Sq(i,j)
  return XSq[i+1,j+1];
end function;


/* Sq^14 = Sq^2 Sq^12 + Sq^13 Sq^1
*/

for j in [0..9] do
  XSq[15,j+1] := Sq(12,j)*Sq(2,j+12) + Sq(1,j)*Sq(13,j+1);
end for;

/*
Redefine Sq
*/

function Sq(i,j)
  return XSq[i+1,j+1];
end function;



/* Sq^15 = Sq^1 Sq^14
*/

for j in [0..8] do
  XSq[16,j+1] := Sq(14,j)*Sq(1,j+14);
end for;

/*
Redefine Sq
*/

function Sq(i,j)
  return XSq[i+1,j+1];
end function;


/* Now use Adem relations to determine relations
   implied only by Sq^i, i < 16
*/

Rel := ideal<R | 0>;

printf "\nComputing relations for Sq^8 action on A(2)\n";

for b in [1..15] do
for a in [1..Min(15,2*b-1)] do
for j in [0..23-a-b] do
if &and [ IsEven(Binomial(b-j-1,a-2*j)) or
             (a+b-j le 15 and j le 15)
	   : j in [0..Truncate(a/2)] ] then

  M := Sq(b,j)*Sq(a,j+b);
  for k in [0..Truncate(a/2)] do
     if IsOdd(Binomial(b-k-1,a-2*k)) then
        M -:= Sq(k,j)*Sq(a+b-k,j+k);
     end if;
  end for;
  if not IsZero(M) then
    /* printf "Relation from Sq^%o Sq%o:\n%o\n\n ",a,b,M; */
    Rel +:= ideal<R | &cat[[M[i,j] 
                      : i in [1..#Rows(M)]] 
		      : j in [1..#Rows(Transpose(M))]]>;
  end if;

end if;
end for;
end for;
end for;


bb := {x : x in Basis(Rel) | not IsZero(x)};
printf "There are %o relations defining Rel.\n",#bb;
bb1 := { x : x in Basis(Rel) | Degree(x) eq 1};
Rel1 := ideal<R | bb1>;
Groebner(Rel1);
printf "Of these, %o relations are of degree 1, defining Rel1\n",#bb1;
printf "The Groebner basis for Rel1 has %o elements.\n",#Basis(Rel1);

/*
Basis(Rel) contains 496 nonzero elements.   Separate out those of degree 1,
of which there are 81 independent relations,
to reduce the number of variables from 94 to 13 = 94-81.

Change the entries in XSq to their normal form w.r.t. this ideal
to eliminate excess variables, by applying the hom f, which
replaces each variable by its normal form with respect to
the ideal Rel1.

Then define Sq^16 and compute the complete ideal of all relations.
*/


f := hom<R->R | [NormalForm(R.i,Rel1) : i in [1..N]]>;


/*
Replace Sq^i entries by their normal forms to simplify
the relations produced by the remaining Adem relations
*/

for i in [0..#XSq-1] do
for j in [0..#XSq[i+1]-1] do
   XSq[i+1,j+1] := Matrix(R,#A_bas(j),#A_bas(i+j),
                   [[ f(XSq[i+1,j+1][ii,jj])
		      : jj in [1..#A_bas(i+j)]]
		      : ii in [1..#A_bas(j)]]);
end for;
end for;

/*
Redefine Sq
*/

function Sq(i,j)
  return XSq[i+1,j+1];
end function;


/* 
Now we have reduced the use of the first 94 variables down
to the minimum 13 given the linear relations they must satisfy.
Those are:  a1, a2, a13, a14, a23, 
            b1, b9, b10, b25, b26, b27, b51, b52.
Proceed to define Sq^16 and compute the remaining relations
*/


/*  Sq^16
*/

XU := [* Matrix(R,#B_bas(j),#B_bas(16+j),
                  [ R!0 : k in [1..#B_bas(j)*#B_bas(16+j) ]])
         : j in [0..24] *];

last := N1+N2;
for j in [0..24-16] do
  next := last+#B_bas(j)*#B_bas(16+j);
  XU[j+1] := Matrix(R,#B_bas(j),#B_bas(16+j),
                  [ R.(N+1-i) : i in [last+1..next] ]);
  last := next;
end for;


XV := [* Matrix(R,#B_bas(j),#B_bas(9+j),
                  [ R!0 : k in [1..#B_bas(j)*#B_bas(9+j) ]])
         : j in [0..24] *];

last := N1+N2+M1;
for j in [0..24-16] do
  next := last+#B_bas(j)*#B_bas(9+j);
  XV[j+1] := Matrix(R,#B_bas(j),#B_bas(9+j),
                  [ R.(N+1-i) : i in [last+1..next] ]);
  last := next;
end for;

for n in [0..7] do
XSq[17][n+1] := Matrix(R,#A_bas(n),#A_bas(n+16),
  [[XU[n+1][i,j] : j in [1..#B_bas(n+16)]] cat
   [XV[n+1][i,j] : j in [1..#Q_bas(n+16)]] 
    : i in [1..#B_bas(n)] ] cat
  [[ R!0         : j in [1..#B_bas(n+16)]] cat
   [XU[n-6][i,j] : j in [1..#Q_bas(n+16)]] 
    : i in [1..#Q_bas(n)] ]);
end for;


printf "Initial Sq^16:\n%o\n",XSq[17];

/*
Redefine Sq
*/

function Sq(i,j)
  return XSq[i+1,j+1];
end function;


/* Sq^17 = Sq^1 Sq^16
*/

for j in [0..6] do
  XSq[18,j+1] := Sq(16,j)*Sq(1,j+16);
end for;

/*
Redefine Sq
*/

function Sq(i,j)
  return XSq[i+1,j+1];
end function;


/* Sq^18 = Sq^2 Sq^16 + Sq^17 Sq^1
*/

for j in [0..5] do
  XSq[19,j+1] := Sq(16,j)*Sq(2,j+16) + Sq(1,j)*Sq(17,j+1);
end for;

/*
Redefine Sq
*/

function Sq(i,j)
  return XSq[i+1,j+1];
end function;


/* Sq^19 = Sq^1 Sq^18
*/

for j in [0..4] do
  XSq[20,j+1] := Sq(18,j)*Sq(1,j+18);
end for;

/*
Redefine Sq
*/

function Sq(i,j)
  return XSq[i+1,j+1];
end function;

/* Sq^20 = Sq^4 Sq^16 + Sq^19 Sq^1 + Sq^18 Sq^2
*/

for j in [0..3] do
  XSq[21,j+1] := Sq(16,j)*Sq(4,j+16) + Sq(1,j)*Sq(19,j+1) + Sq(2,j)*Sq(18,j+2);
end for;

/*
Redefine Sq
*/

function Sq(i,j)
  return XSq[i+1,j+1];
end function;


/* Sq^21 = Sq^1 Sq^20
*/

for j in [0..2] do
  XSq[22,j+1] := Sq(20,j)*Sq(1,j+20);
end for;

/*
Redefine Sq
*/

function Sq(i,j)
  return XSq[i+1,j+1];
end function;


/* Sq^22 = Sq^2 Sq^20 + Sq^21 Sq^1
*/

for j in [0..1] do
  XSq[23,j+1] := Sq(20,j)*Sq(2,j+20) + Sq(1,j)*Sq(21,j+1);
end for;

/*
Redefine Sq
*/

function Sq(i,j)
  return XSq[i+1,j+1];
end function;



/* Sq^23 = Sq^1 Sq^22
*/

for j in [0..0] do
  XSq[24,j+1] := Sq(22,j)*Sq(1,j+22);
end for;

/*
Redefine Sq
*/

function Sq(i,j)
  return XSq[i+1,j+1];
end function;


/* ---------------  Now use Adem relations to determine all relations
*/

NewRel := ideal<R | 0>;

printf "\nRelations for Sq^8 and Sq^16 action on A(2)\n\n";

for b in [1..23] do
for a in [1..2*b-1] do
for j in [0..23-a-b] do
  M := Sq(b,j)*Sq(a,j+b);
  for k in [0..Truncate(a/2)] do
     if IsOdd(Binomial(b-k-1,a-2*k)) then
        M -:= Sq(k,j)*Sq(a+b-k,j+k);
     end if;
  end for;
  if not IsZero(M) then
    /* printf "Relation from Sq^%o Sq%o:\n%o\n\n ",a,b,M; */
    NewRel +:= ideal<R | &cat[[M[i,j] 
                      : i in [1..#Rows(M)]] 
		      : j in [1..#Rows(Transpose(M))]]>;
  end if;

end for;
end for;
end for;

bb := {x : x in Basis(NewRel) | not IsZero(x)};
printf "There are %o relations defining NewRel.\n",#bb;

/*
Now eliminate variables using the new degree 1 relations
*/

Newbb1 := { x : x in bb | Degree(x) eq 1};
NewRel1 := ideal<R | Newbb1>;
Groebner(NewRel1);
printf "Of these, %o relations are of degree 1, defining NewRel1\n",#Newbb1;
printf "The Groebner basis for NewRel1 has %o elements.\n",#Basis(NewRel1);
printf "These are %o\n",Basis(NewRel1);


g := hom<R->R | [NormalForm(f(R.i),NewRel1) : i in [1..N]]>;

/* Replace elements of XSq by their normal forms w.r.t. the
linear relations, i.e., reducing to the 19 remaining variables
*/

for i in [0..#XSq-1] do
for j in [0..#XSq[i+1]-1] do
   XSq[i+1,j+1] := Matrix(R,#A_bas(j),#A_bas(i+j),
                   [[ g(XSq[i+1,j+1][ii,jj])
		      : jj in [1..#A_bas(i+j)]]
		      : ii in [1..#A_bas(j)]]);
end for;
end for;

/*
Redefine Sq
*/

function Sq(i,j)
  return XSq[i+1,j+1];
end function;

rel2 := {g(x) : x in Basis(NewRel)};
printf "Almost final relations:\n%o\n\n",rel2;

/*
Stop execution here, and export the results (the remaining
19 variables and 19 relations into the file A2Asym, where the
remaining reductions take place.   Then, import those results here
to finish computing the squaring operations in terms of the
final 9 variables, so that moddef files can be written for
each, or selected, combinations of variables.
*/

/* Now import the solutions found in A2Asym, rewriting these last 19 variables
in terms of the remaining 9 which determine the others.
This h : A --> A rewrites the 10 that must change and send
other variables to themselves, then compose this with g.
*/

/* define the 10 out of the penultimate 19 in terms
of the final 9 variables:
*/

a1 := R.119;
a2 := R.118;
a13 := R.107;
a23 := R.97;
b1 := R.91;
c1 := R.25;
d1 := R.24;
d2 := R.23;
d3 := R.22;

img := [
R.1,
R.2,
R.3,
a1*a13 + a2*b1 + b1 + c1 + d1 + 1,	/* d21	*/
R.5,
R.6,
R.7,
R.8,
R.9,
R.10,
a1 + a2*a13*a23 + a2*a13*b1 + a2*a13 + a2*b1 +
     a2 + a13*a23 + a13*b1 + a13 + b1 + d1 + d3, 	/* d14	*/
a1*a13 + a2*a13*a23 + a2*a13*b1 + a2*a13 + a2*b1 +
      a13*a23 + a13 + c1 + d2 + 1,			/* d13	*/
R.13,
R.14,
R.15,
R.16,
R.17,
R.18,
R.19,
R.20,
R.21,
R.22,
R.23,
R.24,
R.25,
R.26,
R.27,
R.28,
R.29,
R.30,
R.31,
R.32,
R.33,
R.34,
R.35,
R.36,
R.37,
R.38,
R.39,
a2*a13*a23 + a2*a13*b1 + a2*b1 + a2 + a13*a23 + a13*b1 + b1 + 1,	/* b52	*/
R.41,
R.42,
R.43,
R.44,
R.45,
R.46,
R.47,
R.48,
R.49,
R.50,
R.51,
R.52,
R.53,
R.54,
R.55,
R.56,
R.57,
R.58,
R.59,
R.60,
R.61,
R.62,
R.63,
R.64,
a1*a13 + a1 + a2*a13*b1 + a2*b1 + a2 + a13*a23 + a23,	/* b27	*/
a1*a2*a13 + a1*a13 + a1 + a2*a13*a23 + a2*a13 +
     a2*a23 + a2 + a13*b1 + a13 + a23 + 1,		/* b26	*/
a1*a2*a13 + a1*a13 + a2*a13*b1 + a2 + a13*b1 + a13 + 1,	/* b25	*/
R.68,
R.69,
R.70,
R.71,
R.72,
R.73,
R.74,
R.75,
R.76,
R.77,
R.78,
R.79,
R.80,
R.81,
a1*a2*a13 + a1*a13 + a2*a13 + 1,			/* b10	*/
a1*a2*a13 + a2*a13 + a2*a23 + a2*b1 + a2 + a13 + 1,	/* b9	*/
R.84,
R.85,
R.86,
R.87,
R.88,
R.89,
R.90,
R.91,
R.92,
R.93,
R.94,
R.95,
R.96,
R.97,
R.98,
R.99,
R.100,
R.101,
R.102,
R.103,
R.104,
R.105,
a1*a13 + a1 + a2*a23 + a13,			/* a14	*/
R.107,
R.108,
R.109,
R.110,
R.111,
R.112,
R.113,
R.114,
R.115,
R.116,
R.117,
R.118,
R.119
];

h := hom<R->R | img>;

for i in [0..#XSq-1] do
for j in [0..#XSq[i+1]-1] do
   XSq[i+1,j+1] := Matrix(R,#A_bas(j),#A_bas(i+j),
                   [[ h(XSq[i+1,j+1][ii,jj])
		      : jj in [1..#A_bas(i+j)]]
		      : ii in [1..#A_bas(j)]]);
end for;
end for;


I :=  ideal<R | [R.i^2-R.i : i in [1..119]]>;
rrr := {h(NormalForm(x,I)) : x in rel2};   
rrr := {h(NormalForm(x,I)) : x in rrr };
rrr := {x : x in rrr | not IsZero(x)};   
Rel := ideal<R | rrr>;

for i in [0..#XSq-1] do
for j in [0..#XSq[i+1]-1] do
   XSq[i+1,j+1] := Matrix(R,#A_bas(j),#A_bas(i+j),
                   [[ NormalForm(XSq[i+1,j+1][ii,jj],Rel)
                      : jj in [1..#A_bas(i+j)]]
                      : ii in [1..#A_bas(j)]]);
end for;
end for;



/*
Redefine Sq
*/

function Sq(i,j)
  return XSq[i+1,j+1];
end function;


/* Set numerical values for these last 9 and write out
a moddef file
*/

val := [
R.1,
R.2,
R.3,
R.4,
R.5,
R.6,
R.7,
R.8,
R.9,
R.10,
R.11,
R.12,
R.13,
R.14,
R.15,
R.16,
R.17,
R.18,
R.19,
R.20,
R.21,
R!0,		/* d3	*/
R!0,		/* d2	*/
R!0,		/* d1	*/
R!0,		/* c1	*/
R.26,
R.27,
R.28,
R.29,
R.30,
R.31,
R.32,
R.33,
R.34,
R.35,
R.36,
R.37,
R.38,
R.39,
R.40,
R.41,
R.42,
R.43,
R.44,
R.45,
R.46,
R.47,
R.48,
R.49,
R.50,
R.51,
R.52,
R.53,
R.54,
R.55,
R.56,
R.57,
R.58,
R.59,
R.60,
R.61,
R.62,
R.63,
R.64,
R.65,
R.66,
R.67,
R.68,
R.69,
R.70,
R.71,
R.72,
R.73,
R.74,
R.75,
R.76,
R.77,
R.78,
R.79,
R.80,
R.81,
R.82,
R.83,
R.84,
R.85,
R.86,
R.87,
R.88,
R.89,
R.90,
R!0,			/* b1	*/
R.92,
R.93,
R.94,
R.95,
R.96,
R!0,			/* a23	*/
R.98,
R.99,
R.100,
R.101,
R.102,
R.103,
R.104,
R.105,
R.106,
R!0,			/* a13	*/
R.108,
R.109,
R.110,
R.111,
R.112,
R.113,
R.114,
R.115,
R.116,
R.117,
R!0,			/* a2	*/
R!0 			/* a1	*/
];

for aa in [0,1] do
for bb in [0,1] do
for cc in [0,1] do
for dd in [0,1] do
for ee in [0,1] do
val[119] := aa;
val[118] := bb;
val[107] := cc;
val[97] := dd;
val[91] := ee;

if IsEven(aa*(bb+cc) + bb*(cc*(dd + ee) + cc + dd + ee + 1) + ee) then

for ff in [0,1] do
for gg in [0,1] do
for hh in [0,1] do
for ii in [0,1] do

val[25] := ii;
val[24] := hh;
val[23] := gg;
val[22] := ff;

k := hom<R->R | val>;

VSq := [ [* Matrix(R,#A_bas(j),#A_bas(i+j),
                  [ R!0 : kk in [1..#A_bas(j)*#A_bas(i+j) ]])
         : j in [0..24] *]
	 : i in [0..23] ];

for i in [0..#XSq-1] do
for j in [0..#XSq[i+1]-1] do
   VSq[i+1,j+1] := Matrix(R,#A_bas(j),#A_bas(i+j),
                   [[ k(XSq[i+1,j+1][ii,jj])
		      : jj in [1..#A_bas(i+j)]]
		      : ii in [1..#A_bas(j)]]);
end for;
end for;

filename := "Aof2-" cat 
            &cat[IntegerToString(Integers()!k(R.x)) 
		: x in [119,118,107,97,91,25,24,23,22]];
SetOutputFile(filename);

&+[#A_bas(n) : n in [0..23]];
degs := &cat[[n : i in [1..#A_bas(n) ]]: n in [0..23]];
count := 0;
for i in [1..#degs] do
  printf "%o",degs[i];
  count +:= 1;
  if count ge 10 then
     printf "\n";
     count := 0;
  else
     printf " ";
  end if;
end for;
printf "\n\n";

mons := &cat[ A_bas(n) : n in [0..23]];

for gg in [1..#mons] do
d := degs[gg];
for i in [1..23-d] do
  vv := VSq[i+1,d+1][Index(A_bas(d),mons[gg])];
  if not IsZero(vv) then
     mm := [Index(mons,A_bas(d+i)[ii]) 
             : ii in [1..#A_bas(d+i)] | not IsZero(vv[ii])];
     printf "%o %o %o",gg-1,i,#mm;
     for xx in mm do
       printf " %o",xx-1;
     end for;
     printf "\n";
  end if;
end for;
printf "\n";
end for;
printf "\n";



UnsetOutputFile();


end for; /* ii */
end for; /* hh */
end for; /* gg */
end for; /* ff */
end if;
end for; /* ee */
end for; /* dd */
end for; /* cc */
end for; /* bb */
end for; /* aa */

\end{verbatim}

\section{MAGMA code, symmetric case, third step}
\label{codesym3}

This code was built and executed `one variable at a time'.
That is, the code was run repeatedly, each time looking for
the last variable which could be replaced (as described
in Subsection~\ref{sub:thirdstep}) and then adding the code to replace
that variable.  Thus, in the first run, no variables were replaced.  
Only the 
(now commented out) code looking for the last variable which could be replaced
was run. That was determined and the code to replace that variable was added.
This new version was then run, with the (now commented out) code
looking for the next to last variable which could be replaced.
The code to replace that variable was then added.
This was continued until no further such reductions were possible.

The code at the end then writes out all $16$ of the points in the 
variety defining the $Sq^8$ acion in the symmetric case which is
found in Appendix~\ref{allsymSq8}.

\begin{verbatim}
/*
Symmetric case.
Third step:  use the remaining relations to continue eliminating variables 
as far as possible.
*/


S<a1,a2,a13,a14,a23,b1,b9,b10,b25,b26,b27,b52,
   c1,d1,d2,d3,d13,d14,d21>
   := PolynomialRing(GF(2),19);

vars :=
[a1,a2,a13,a14,a23,b1,b9,b10,b25,b26,b27,b52,
   c1,d1,d2,d3,d13,d14,d21];

/*
This list of relations was generated by computing
 {g(x) : x in Basis(NewRel)};
in the file A2sym.
*/

rels :=
  [ d14 + d13 + d3 + d2 + d1 + c1 + b1*a13 + b1 + a13*a1 + a2 + a1 + 1,
    b27*a13 + b27 + b26 + b25*a13 + b25 + b10 + b9*a13 + b9 + b1*a13 + a23*a13 + a13^2 +
        a13*a2 + a13*a1,
    b27*a2 + b10 + b1 + 1,
    b52 + b27 + b26 + b9 + b1*a2 + b1 + 1,
    d13 + d2 + c1 + b27*a2 + b27 + b25*a2 + b25 + b9*a2 + b9 + b1*a13 + b1*a2 + a23*a2 +
        a23 + a13*a2 + a13*a1 + a13 + a2^2 + a2*a1 + a2 + a1 + 1,
    b26*a2 + b25*a2 + b1*a2 + b1 + a23*a2 + a13*a1 + a2,
    b27*a2 + b9 + b1*a2 + b1 + a14 + a2 + a1 + 1,
    d14 + d3 + d1 + b27*a2 + b27 + b25*a2 + b25 + b9*a2 + b9 + b1*a2 + b1 + a23*a2 + a23 +
        a13*a2 + a13 + a2^2 + a2*a1,
    d13 + d2 + c1 + b27*a2 + b27 + b25*a2 + b25 + b9*a2 + b9 + b1*a13 + b1*a2 + a23 + a14
        + a13*a2 + a2^2 + a2*a1 + a2 + 1,
    b26 + b25 + b10 + b9*a13 + a23 + a14 + a13 + 1,
    b27*a2 + b9 + b1*a2 + b1 + a23*a2 + a13*a1 + a13 + a2 + 1,
    d21 + d14 + d13 + d3 + d2 + b10 + b9 + b1*a13 + a14,
    b27*a13 + b27 + b25*a13 + b9 + b1*a13 + a23*a13 + a23 + a14 + a13^2 + a13*a2 + a13*a1
        + a13 + 1,
    b52 + b27 + b25 + b9*a13 + b1 + a23 + a13 + a2 + a1,
    b10 + b9 + b1*a2 + a14 + a2 + a1,
    b52 + b26*a13 + b26 + b25 + b9 + b1 + a23 + a13 + a1,
    b26*a2 + b25*a2 + b10 + b9 + b1 + a13,
    a23*a2 + a14 + a13*a1 + a13 + a1,
    b27*a2 + b26*a2 + b25*a2 + b9 + a13 + 1
];

/* Since the variables will all be either 0 or 1, each is
equal to its own square.   Hence we reduce mod x=x^2
*/

I := ideal<S | [x-x^2 : x in vars]>;
rels := [NormalForm(x,I) : x in rels];

rels;


/*
test whether v is a term of r, 
and v is not a factor of any higher degree term.
If so then the relation r can be used to eliminate v.
*/

safe := function(v,r)	
  tt := Terms(r);
  if v in tt then
    tt2 := {x : x in tt | Degree(x) gt 1};
    f2 := &cat [Factorization(x) : x in tt2];
    if v in {x[1] : x in f2} then
      return false;
    else
      return true;
    end if;
  else
    return false;
  end if;
end function;

/*
Next few commands just check where we stand
*/

terms := &cat [ Terms(x) : x in rels];
terms := {x : x in terms};
#terms," terms";

t1 := {x : x in terms | Degree(x) eq 1};
tt2 := {x : x in terms | Degree(x) eq 2};

t2 := &cat[Factorization(x) : x in tt2];
t2 := { x[1] : x in t2};

#t1,#t2,#(t1 meet t2),"\n";


t1 := Sort([x : x in t1]);
t2 := Sort([x : x in t2]);

/*
Now, to work
*/

/* Each of the following steps arises by examining, by hand,
the chance to use the relations to eliminate variables,
starting at the end (d_24) and working forward toward a_1.
If a variable occurs alone as a degree 1 term in a relation
(checked by the command named 'safe'), we can use that relation
to eliminate it from the remaining relations.

The examination is done by executing

[<j, [<i,safe(S.j,rels[i]),rels[i]> 
       : i in [1..#rels] | S.j in Terms(rels[i])]> 
       : j in [1..#vars]];

For each variable S.j, and for each relation rels[i] which contains S.j
as a term, show the index i of the relation, whether the relation
is safe to use to eliminate S.j, and the relation.  Among these,
we choose the largest j which has a safe relation, and among the
relations, we use the shortest.

This code is commented out of the run whose output is in the paper,
for brevity, but can easily be executed by moving the end comment
marker following it to the beginning of the paragraph.
*/

/* Summary of final result:
   eliminate d21 using rels[12] (S.19)
             d14 using rels[1] (S.18)
             d13 using rels[9] (S.17)
             b52 using rels[4] (S.12)
             b27 using rels[16] (S.11)
             b26 using rels[10] (S.10)
             b25 using rels[2] (S.9)
             b10 using rels[14] (S.8)
             b9 using rels[19] (S.7)
             b1 using rels[6] (S.6)
             a14 using rels[5] (S.4)
             a2 using rels[3] (S.2)


[<S.j, [<i,safe(S.j,rels[i]),rels[i]> 
       : i in [1..#rels] | S.j in Terms(rels[i])]> 
       : j in [1..#vars]];
*/

safe(S.19,rels[12]);
g := hom<S->S | [S.i : i in [1..18]] cat
                [S.19+rels[12]]
		>;
rels := [NormalForm(g(x),I) : x in rels];
S.19;
g eq g*g;		/* sanity check	*/

/* -------- S.19 = d21 replaced ---------------------

[<S.j, [<i,safe(S.j,rels[i]),rels[i]> 
       : i in [1..#rels] | S.j in Terms(rels[i])]> 
       : j in [1..#vars]];
*/

safe(S.18,rels[1]);
g1:= hom<S->S | [S.i : i in [1..17]] cat
                [S.18+rels[1]] cat
                [S.19]
		>;
g := g*g1;		/* First apply g, then g1	*/
g := hom<S->S | [NormalForm(g(S.i),I) : i in [1..19]]>;
rels := [NormalForm(g(x),I) : x in rels];
S.18;
g eq g*g;		/* sanity check	*/

/* -------- S.18 = d14 replaced ---------------------

[<S.j, [<i,safe(S.j,rels[i]),rels[i]> 
       : i in [1..#rels] | S.j in Terms(rels[i])]> 
       : j in [1..#vars]];
*/


safe(S.17,rels[9]);
g1:= hom<S->S | [S.i : i in [1..16]] cat
                [S.17+rels[9]] cat
                [S.18,S.19]
		>;
g := g*g1;
g := hom<S->S | [NormalForm(g(S.i),I) : i in [1..19]]>;
rels := [NormalForm(g(x),I) : x in rels];
S.17;
g eq g*g;		/* sanity check	*/

/* -------- S.17 = d13 replaced ---------------------

[<S.j, [<i,safe(S.j,rels[i]),rels[i]> 
       : i in [1..#rels] | S.j in Terms(rels[i])]> 
       : j in [1..#vars]];
*/

safe(S.12,rels[4]);
g1:= hom<S->S | [S.i : i in [1..11]] cat
                [S.12+rels[4]] cat
                [S.i : i in [13..19]]
		>;
g := g*g1;
g := hom<S->S | [NormalForm(g(S.i),I) : i in [1..19]]>;
rels := [NormalForm(g(x),I) : x in rels];
S.12;
g eq g*g;		/* sanity check	*/

/* -------- S.12 = b52 replaced ---------------------

[<S.j, [<i,safe(S.j,rels[i]),rels[i]> 
       : i in [1..#rels] | S.j in Terms(rels[i])]> 
       : j in [1..#vars]];
*/


safe(S.11,rels[16]);
g1:= hom<S->S | [S.i : i in [1..10]] cat
                [S.11+rels[16]] cat
                [S.i : i in [12..19]]
		>;
g := g*g1;
g := hom<S->S | [NormalForm(g(S.i),I) : i in [1..19]]>;
rels := [NormalForm(g(x),I) : x in rels];
S.11;
g eq g*g;		/* sanity check	*/

/* -------- S.11 = b27 replaced ---------------------

[<S.j, [<i,safe(S.j,rels[i]),rels[i]> 
       : i in [1..#rels] | S.j in Terms(rels[i])]> 
       : j in [1..#vars]];
*/


safe(S.10,rels[10]);
g1:= hom<S->S | [S.i : i in [1..9]] cat
                [S.10+rels[10]] cat
                [S.i : i in [11..19]]
		>;
g := g*g1;
g := hom<S->S | [NormalForm(g(S.i),I) : i in [1..19]]>;
rels := [NormalForm(g(x),I) : x in rels];
S.10;
g eq g*g;		/* sanity check	*/

/* -------- S.10 = b26 replaced ---------------------

[<S.j, [<i,safe(S.j,rels[i]),rels[i]> 
       : i in [1..#rels] | S.j in Terms(rels[i])]> 
       : j in [1..#vars]];
*/


safe(S.9,rels[2]);
g1:= hom<S->S | [S.i : i in [1..8]] cat
                [S.9+rels[2]] cat
                [S.i : i in [10..19]]
		>;
g := g*g1;
g := hom<S->S | [NormalForm(g(S.i),I) : i in [1..19]]>;
rels := [NormalForm(g(x),I) : x in rels];
S.9;
g eq g*g;		/* sanity check	*/

/* -------- S.9 = b25 replaced ---------------------

[<S.j, [<i,safe(S.j,rels[i]),rels[i]> 
       : i in [1..#rels] | S.j in Terms(rels[i])]> 
       : j in [1..#vars]];
*/

safe(S.8,rels[14]);
g1:= hom<S->S | [S.i : i in [1..7]] cat
                [S.8+rels[14]] cat
                [S.i : i in [9..19]]
		>;
g := g*g1;
g := hom<S->S | [NormalForm(g(S.i),I) : i in [1..19]]>;
rels := [NormalForm(g(x),I) : x in rels];
S.8;
g eq g*g;		/* sanity check	*/

/* -------- S.8 = b10 replaced ---------------------

[<S.j, [<i,safe(S.j,rels[i]),rels[i]> 
       : i in [1..#rels] | S.j in Terms(rels[i])]> 
       : j in [1..#vars]];
*/


safe(S.7,rels[19]);
g1:= hom<S->S | [S.i : i in [1..6]] cat
                [S.7+rels[19]] cat
                [S.i : i in [8..19]]
		>;
g := g*g1;
g := hom<S->S | [NormalForm(g(S.i),I) : i in [1..19]]>;
rels := [NormalForm(g(x),I) : x in rels];
S.7;
g eq g*g;		/* sanity check	*/

/* -------- S.7 = b9 replaced ---------------------

[<S.j, [<i,safe(S.j,rels[i]),rels[i]> 
       : i in [1..#rels] | S.j in Terms(rels[i])]> 
       : j in [1..#vars]];
*/


safe(S.4,rels[5]);
g1:= hom<S->S | [S.i : i in [1..3]] cat
                [S.4+rels[5]] cat
                [S.i : i in [5..19]]
		>;
g := g*g1;
g := hom<S->S | [NormalForm(g(S.i),I) : i in [1..19]]>;
rels := [NormalForm(g(x),I) : x in rels];
S.4;
g eq g*g;		/* sanity check	*/

/* -------- S.4 = a14 replaced ---------------------

[<S.j, [<i,safe(S.j,rels[i]),rels[i]> 
       : i in [1..#rels] | S.j in Terms(rels[i])]> 
       : j in [1..#vars]];
*/



/* 
Now look for all solutions
over GF(2).   There are 16 of them.
*/


U<a1, a2, a13, a23, b1> := PolynomialRing(GF(2),5);
newrel :=
[
a1*a2 + a1*a13 + a2*a13*a23 + a2*a13*b1 + a2*a13 + a2*a23 + a2*b1 + a2 + b1
];


count := 0;
for a in GF(2) do
for b in GF(2) do
for c in GF(2) do
for d in GF(2) do
for e in GF(2) do
f := b*(a*c + c + d + e + 1) + c + 1;
g := a*b*c + a*c + b*c*e + b + c*e + c + 1;
h := hom<U->U | [a,b,c,d,e]>;
if &and[IsZero(h(x)) : x in newrel] then
  count +:= 1;
  printf "%4o : %o\n",count,[a,b,c,d,e,f,g];
end if;
end for;
end for;
end for;
end for;
end for;

\end{verbatim}

\section{MAGMA code, general case}
\label{codegen}

\begin{verbatim}
/* 
This is the version that computes all A-module structures on A(2),
not just the symmetric ones.
*/

/* Reversing the order of the variables so that relations will 
get reduced to the earliest instance of each.
This is accomplished by reversing the 'AssignNames list of names,
and by reversing the assignment of R.i's to matrix entries by
using R.(N+1-i) instead.
*/

/* 
A = A(2).   Write as B + BQ_2, with the v.s. B spanned by
the (r1,r2) and BQ_2 spanned by the (r1,r2,1).
Write A_n, B_n, Q_n = B_{n-7}Q_2 for the degree n subspaces.
*/

/* Compute Sq^16 separately, later, after using degree 1 relations
implied by Adem relations among the first N1 variables to simplify
the Sq^i for i < 15.
*/

XBbas := [&cat[[ [n-3*j-7*k,j,k]
             : j in [0..3] | n-3*j-7*k ge 0 and n-3*j-7*k le 7]
             : k in [0..0]]
             : n in [0..39]];

XQbas := [&cat[[ [n-3*j-7*k,j,k]
             : j in [0..3] | n-3*j-7*k ge 0 and n-3*j-7*k le 7]
             : k in [1..1]]
             : n in [0..39]];


function B_bas(j)
  if j ge 0 and j+1 le #XBbas then
    return XBbas[j+1];
  else
    return [];
  end if;
end function;

function Q_bas(j)
  if j ge 0 and j+1 le #XQbas then
    return XQbas[j+1];
  else
    return [];
  end if;
end function;

function A_bas(j)
  return B_bas(j) cat Q_bas(j);
end function;

N1 :=  &+[#A_bas(j)*#A_bas(j+8) : j in [0..24-8]];
N2 :=  &+[#A_bas(j)*#A_bas(j+16) : j in [0..24-16]];

N := N1+N2;

/*
There are N1 = 124 and N2 = 26 indeterminates required to
describe Sq^8 and Sq^16, resp.
*/


R := PolynomialRing(GF(2),N);
AssignNames(~R,
        Reverse(["a" cat IntegerToString(i) : i in [1..N1]] cat
                ["b" cat IntegerToString(i-N1) : i in [N1+1..N1+N2]] 
		));

/* Define Sq^a action on Milnor basis element Sq(r1,r2,r3)
assuming that a < 8.)
*/

function MSq(a,r)
  return &cat [ [ [a+r[1]-3*i-4*j, r[2]+i-j, r[3]+j]
                 : i in [0..Min(r[1],Truncate((a-4*j)/2))] 
		   | a-2*i-4*j ge 0 and
		     IsOdd(Binomial(a+r[1]-3*i-4*j,r[1]-i)) and
		     IsOdd(Binomial(r[2]+i-j,i)) and
		     IsOdd(Binomial(r[3]+j,j)) ]
		 : j in [0..Min(r[2],Truncate(a/4))] ];
end function;

function In(x,L)
  if x in L then return 1; else return 0; end if;
end function;


/* 
Sq(i,j) = XSq[i+1,j+1] is Sq^i from degree j to i+j
*/

XSq := [ [* Matrix(R,#A_bas(j),#A_bas(i+j),
                  [ R!0 : k in [1..#A_bas(j)*#A_bas(i+j) ]])
         : j in [0..24] *]
	 : i in [0..23] ];

for i in [0..7] do
for j in [0..23] do
   XSq[i+1,j+1] := Matrix(R,#A_bas(j),#A_bas(i+j),
                   [[In(b,MSq(i,r)) : b in A_bas(i+j)] : r in A_bas(j)] );
end for;
end for;

/*
Define this AFTER computing the entries
*/

function Sq(i,j)
  return XSq[i+1,j+1];
end function;

/*
Test Adem relations in this range
Expect no output
*/

for n in [0..23] do
for b in [1..7] do
for a in [1..2*b-1] do
/* check that all the terms needed are defined	*/
if n+a+b le 23  and
   &and [ IsEven(Binomial(b-j-1,a-2*j)) or
             (a+b-j le 7 and j le 7)
	   : j in [0..Truncate(a/2)] ] then

  M := Sq(b,n)*Sq(a,n+b);
  for j in [0..Truncate(a/2)] do
     if IsOdd(Binomial(b-j-1,a-2*j)) then
        M +:= Sq(j,n)*Sq(a+b-j,n+j);
     end if;
  end for;
  if not IsZero(M) then
    print "Wrong: ",a,b,M;
  end if;

end if;
end for;
end for;
end for;


/* Define Sq^8
*/

XSq[9] := [* Matrix(R,#A_bas(j),#A_bas(8+j),
                  [ R!0 : k in [1..#A_bas(j)*#A_bas(8+j) ]])
         : j in [0..24] *];

last := 0;
for j in [0..24-8] do
  next := last+#A_bas(j)*#A_bas(8+j);
  XSq[9][j+1] := Matrix(R,#A_bas(j),#A_bas(8+j),
                  [ R.(N+1-i) : i in [last+1..next] ]);
  last := next;
end for;

printf "\nInitial Sq^8:\n%o\n",XSq[9][1..17];



/*
Redefine Sq
*/

function Sq(i,j)
  return XSq[i+1,j+1];
end function;


/* Sq^9 = Sq^1 Sq^8	*/

for j in [0..14] do
  XSq[10,j+1] := Sq(8,j)*Sq(1,j+8);
end for;

/*
Redefine Sq
*/

function Sq(i,j)
  return XSq[i+1,j+1];
end function;


/* Sq^10 = Sq^2 Sq^8 + Sq^9 Sq^1	*/

for j in [0..13] do
  XSq[11,j+1] := Sq(8,j)*Sq(2,j+8) + Sq(1,j)*Sq(9,j+1);
end for;

/*
Redefine Sq
*/

function Sq(i,j)
  return XSq[i+1,j+1];
end function;


/* Sq^11 = Sq^1 Sq^10	*/

for j in [0..12] do
  XSq[12,j+1] := Sq(10,j)*Sq(1,j+10);
end for;

/*
Redefine Sq
*/

function Sq(i,j)
  return XSq[i+1,j+1];
end function;

/* Sq^12 = Sq^4 Sq^8 + Sq^11 Sq^1 + Sq^10 Sq^2	*/

for j in [0..11] do
  XSq[13,j+1] := Sq(8,j)*Sq(4,j+8) + Sq(1,j)*Sq(11,j+1) + Sq(2,j)*Sq(10,j+2);
end for;

/*
Redefine Sq
*/

function Sq(i,j)
  return XSq[i+1,j+1];
end function;


/* Sq^13 = Sq^1 Sq^12	*/

for j in [0..10] do
  XSq[14,j+1] := Sq(12,j)*Sq(1,j+12);
end for;

/*
Redefine Sq
*/

function Sq(i,j)
  return XSq[i+1,j+1];
end function;


/* Sq^14 = Sq^2 Sq^12 + Sq^13 Sq^1	*/

for j in [0..9] do
  XSq[15,j+1] := Sq(12,j)*Sq(2,j+12) + Sq(1,j)*Sq(13,j+1);
end for;

/*
Redefine Sq
*/

function Sq(i,j)
  return XSq[i+1,j+1];
end function;



/* Sq^15 = Sq^1 Sq^14	*/

for j in [0..8] do
  XSq[16,j+1] := Sq(14,j)*Sq(1,j+14);
end for;

/*
Redefine Sq
*/

function Sq(i,j)
  return XSq[i+1,j+1];
end function;

/* ---------------  Now use Adem relations to determine all relations
*/

Rel := ideal<R | 0>;

printf "\nComputing relations for Sq^8 action only on A(2)\n";

for b in [1..15] do
for a in [1..Min(15,2*b-1)] do
for j in [0..23-a-b] do
if &and [ IsEven(Binomial(b-j-1,a-2*j)) or
             (a+b-j le 15 and j le 15)
	   : j in [0..Truncate(a/2)] ] then

  M := Sq(b,j)*Sq(a,j+b);
  for k in [0..Truncate(a/2)] do
     if IsOdd(Binomial(b-k-1,a-2*k)) then
        M -:= Sq(k,j)*Sq(a+b-k,j+k);
     end if;
  end for;
  if not IsZero(M) then
    /* printf "Relation from Sq^%o Sq%o:\n%o\n\n ",a,b,M; */
    Rel +:= ideal<R | &cat[[M[i,j] 
                      : i in [1..#Rows(M)]] 
		      : j in [1..#Rows(Transpose(M))]]>;
  end if;

end if;
end for;
end for;
end for;


bb := { x : x in Basis(Rel) | not IsZero(x)};
printf "\nThere are %o relations defining Rel.\n",#bb;

/* Basis(Rel) contains 564 elements.   Separate out those of degree 1
to reduce the number of first variables from 124 to 19.

Use the linear relations to reduce the number of variables.
Then define Sq^16 and compute the complete ideal of all relations.
*/

bb1 := { x : x in bb | Degree(x) eq 1};
Rel1 := ideal<R | bb1>;
Groebner(Rel1);
printf "Of these, %o relations are of degree 1, defining Rel1\n",#bb1;
printf "The Groebner basis for Rel1 has %o elements.\n",#Basis(Rel1);

f := hom<R->R | [NormalForm(R.i,Rel1) : i in [1..N]]>;

/*
Replace Sq^i entries by their normal forms to simplify
the relations produced by the remaining Adem relations
*/

for i in [0..#XSq-1] do
for j in [0..#XSq[i+1]-1] do
   XSq[i+1,j+1] := Matrix(R,#A_bas(j),#A_bas(i+j),
                   [[ f(XSq[i+1,j+1][ii,jj])
		      : jj in [1..#A_bas(i+j)]]
		      : ii in [1..#A_bas(j)]]);
end for;
end for;

/*
Redefine Sq
*/

function Sq(i,j)
  return XSq[i+1,j+1];
end function;



/*
printf "\nFirst round Sq^8:\n%o\n",XSq[9][1..17];
printf "Dictionary:\n%o\n",[<i,R.i,f(R.i)> : i in [1..N]];
*/


/* 
Now we have reduced the use of the first 124 variables down
to the 19 given the linear relations the others must satisfy.
Those are:  
    a1, a2, a3, a21, a22, a23, a24, a47, a48, a49,
    a50, a60, a61, a62, a89, a90, a101, a102, a117
*/

/*
Recompute Rel now, with the smaller set of variables
*/

Rel := ideal<R | 0>;

printf "\nRecomputing relations for Sq^8 action only on A(2)\n";

for b in [1..15] do
for a in [1..Min(15,2*b-1)] do
for j in [0..23-a-b] do
if &and [ IsEven(Binomial(b-j-1,a-2*j)) or
             (a+b-j le 15 and j le 15)
	   : j in [0..Truncate(a/2)] ] then

  M := Sq(b,j)*Sq(a,j+b);
  for k in [0..Truncate(a/2)] do
     if IsOdd(Binomial(b-k-1,a-2*k)) then
        M -:= Sq(k,j)*Sq(a+b-k,j+k);
     end if;
  end for;
  if not IsZero(M) then
    /* printf "Relation from Sq^%o Sq%o:\n%o\n\n ",a,b,M; */
    Rel +:= ideal<R | &cat[[M[i,j] 
                      : i in [1..#Rows(M)]] 
		      : j in [1..#Rows(Transpose(M))]]>;
  end if;

end if;
end for;
end for;
end for;


bb := { x : x in Basis(Rel) | not IsZero(x)};
printf "\nThere are now %o relations defining Rel.\n",#bb;

/* Basis(Rel) now contains 22 elements.   Separate out those of degree 1
to reduce the number of first variables from 124 to 16 now.  The
remaining 16 are

    a1, a2, a3, a21, a22, a23, a24, a47,
    a48, a49, a50, a60, a61, a62, a90, a102

Use these linear relations to reduce the number of variables.
Then define Sq^16 and compute the complete ideal of all relations.
*/

bb1 := { x : x in Basis(Rel) | Degree(x) eq 1};
Rel1 := ideal<R | bb1>;
Groebner(Rel1);
printf "Of these, %o relations are of degree 1, defining Rel1\n",#bb1;
printf "The Groebner basis for Rel1 has %o elements.\n",#Basis(Rel1);

ff := hom<R->R | [NormalForm(R.i,Rel1) : i in [1..N]]>;

/*
Replace Sq^i entries by their normal forms to simplify
the relations produced by the remaining Adem relations
*/

for i in [0..#XSq-1] do
for j in [0..#XSq[i+1]-1] do
   XSq[i+1,j+1] := Matrix(R,#A_bas(j),#A_bas(i+j),
                   [[ ff(f(XSq[i+1,j+1][ii,jj]))
		      : jj in [1..#A_bas(i+j)]]
		      : ii in [1..#A_bas(j)]]);
end for;
end for;

/*
Redefine Sq
*/

function Sq(i,j)
  return XSq[i+1,j+1];
end function;


/*
printf "\nSecond round Sq^8:\n%o\n",XSq[9][1..17];
printf "Dictionary:\n%o\n",[<i,R.i,ff(f(R.i))> : i in [1..N]];
*/

/*
Recompute Rel a third time, with the smaller set of variables
*/

Rel := ideal<R | 0>;

printf "\nRecomputing relations for Sq^8 action only on A(2)\n";

for b in [1..15] do
for a in [1..Min(15,2*b-1)] do
for j in [0..23-a-b] do
if &and [ IsEven(Binomial(b-j-1,a-2*j)) or
             (a+b-j le 15 and j le 15)
	   : j in [0..Truncate(a/2)] ] then

  M := Sq(b,j)*Sq(a,j+b);
  for k in [0..Truncate(a/2)] do
     if IsOdd(Binomial(b-k-1,a-2*k)) then
        M -:= Sq(k,j)*Sq(a+b-k,j+k);
     end if;
  end for;
  if not IsZero(M) then
    /* printf "Relation from Sq^%o Sq%o:\n%o\n\n ",a,b,M; */
    Rel +:= ideal<R | &cat[[M[i,j] 
                      : i in [1..#Rows(M)]] 
		      : j in [1..#Rows(Transpose(M))]]>;
  end if;

end if;
end for;
end for;
end for;


bb := { x : x in Basis(Rel) | not IsZero(x)};
printf "\nThere are now %o relations defining Rel.\n",#bb;

/* Basis(Rel) now contains 17 elements.   None are of degree 1.
The remaining variables are still these 16:

    a1, a2, a3, a21, a22, a23, a24, a47,
    a48, a49, a50, a60, a61, a62, a90, a102

58 variables are reduced to constants.

Use the linear relations to reduce the number of variables.
Then define Sq^16 and compute the complete ideal of all relations.
*/

bb1 := { x : x in Basis(Rel) | Degree(x) eq 1};
Rel1 := ideal<R | bb1>;
Groebner(Rel1);
printf "Of these, %o relations are of degree 1, defining Rel1\n",#bb1;
printf "The Groebner basis for Rel1 has %o elements.\n",#Basis(Rel1);

fff := hom<R->R | [NormalForm(R.i,Rel1) : i in [1..N]]>;


/*
Replace Sq^i entries by their normal forms to simplify
the relations produced by the remaining Adem relations
*/

for i in [0..#XSq-1] do
for j in [0..#XSq[i+1]-1] do
   XSq[i+1,j+1] := Matrix(R,#A_bas(j),#A_bas(i+j),
                   [[ fff(ff(f(XSq[i+1,j+1][ii,jj])))
		      : jj in [1..#A_bas(i+j)]]
		      : ii in [1..#A_bas(j)]]);
end for;
end for;

/*
Redefine Sq
*/

function Sq(i,j)
  return XSq[i+1,j+1];
end function;

/*
printf "\nThird round Sq^8:\n%o\n",XSq[9][1..17];
printf "Dictionary:\n%o\n",[<i,R.i,fff(ff(f(R.i)))> : i in [1..N]];
*/

/*
Consolidate all the linear relations into one ideal and one hom
*/

Rel1 := ideal<R | [R.i - fff(ff(f(R.i))) : i in [1..N]]>;
f := hom<R->R | [NormalForm(R.i,Rel1) : i in [1..N]]>;

printf "\nRemaining variables:\n";
for i in [1..124] do
if R.(151-i) eq f(R.(151-i)) then
  printf " %o\n",<151-i,R.(151-i)>;
end if;
end for;
printf "\n";


/*
Proceed to define Sq^16 and compute the remaining relations
*/


/*  Sq^16
*/

XSq[17] := [* Matrix(R,#A_bas(j),#A_bas(16+j),
                  [ R!0 : k in [1..#A_bas(j)*#A_bas(16+j) ]])
         : j in [0..24] *];

last := N1;
for j in [0..24-16] do
  next := last+#A_bas(j)*#A_bas(16+j);
  XSq[17][j+1] := Matrix(R,#A_bas(j),#A_bas(16+j),
                  [ R.(N+1-i) : i in [last+1..next] ]);
  last := next;
end for;

printf "\nInitial Sq^16:\n%o\n",XSq[17][1..9];


/*
Redefine Sq
*/

function Sq(i,j)
  return XSq[i+1,j+1];
end function;


/* Sq^17 = Sq^1 Sq^16
*/

for j in [0..6] do
  XSq[18,j+1] := Sq(16,j)*Sq(1,j+16);
end for;

/*
Redefine Sq
*/

function Sq(i,j)
  return XSq[i+1,j+1];
end function;


/* Sq^18 = Sq^2 Sq^16 + Sq^17 Sq^1
*/

for j in [0..5] do
  XSq[19,j+1] := Sq(16,j)*Sq(2,j+16) + Sq(1,j)*Sq(17,j+1);
end for;

/*
Redefine Sq
*/

function Sq(i,j)
  return XSq[i+1,j+1];
end function;


/* Sq^19 = Sq^1 Sq^18
*/

for j in [0..4] do
  XSq[20,j+1] := Sq(18,j)*Sq(1,j+18);
end for;

/*
Redefine Sq
*/

function Sq(i,j)
  return XSq[i+1,j+1];
end function;

/* Sq^20 = Sq^4 Sq^16 + Sq^19 Sq^1 + Sq^18 Sq^2
*/

for j in [0..3] do
  XSq[21,j+1] := Sq(16,j)*Sq(4,j+16) + Sq(1,j)*Sq(19,j+1) + Sq(2,j)*Sq(18,j+2);
end for;

/*
Redefine Sq
*/

function Sq(i,j)
  return XSq[i+1,j+1];
end function;


/* Sq^21 = Sq^1 Sq^20
*/

for j in [0..2] do
  XSq[22,j+1] := Sq(20,j)*Sq(1,j+20);
end for;

/*
Redefine Sq
*/

function Sq(i,j)
  return XSq[i+1,j+1];
end function;


/* Sq^22 = Sq^2 Sq^20 + Sq^21 Sq^1
*/

for j in [0..1] do
  XSq[23,j+1] := Sq(20,j)*Sq(2,j+20) + Sq(1,j)*Sq(21,j+1);
end for;

/*
Redefine Sq
*/

function Sq(i,j)
  return XSq[i+1,j+1];
end function;



/* Sq^23 = Sq^1 Sq^22
*/

for j in [0..0] do
  XSq[24,j+1] := Sq(22,j)*Sq(1,j+22);
end for;

/*
Redefine Sq
*/

function Sq(i,j)
  return XSq[i+1,j+1];
end function;


/* ---------------  Now use Adem relations to determine all relations
*/

NewRel := ideal<R | 0>;

printf "\nRelations for Sq^8 and Sq^16 action on A(2)\n\n";

for b in [1..23] do
for a in [1..2*b-1] do
for j in [0..23-a-b] do
  M := Sq(b,j)*Sq(a,j+b);
  for k in [0..Truncate(a/2)] do
     if IsOdd(Binomial(b-k-1,a-2*k)) then
        M -:= Sq(k,j)*Sq(a+b-k,j+k);
     end if;
  end for;
  if not IsZero(M) then
    /* printf "Relation from Sq^%o Sq%o:\n%o\n\n ",a,b,M; */
    NewRel +:= ideal<R | &cat[[M[i,j] 
                      : i in [1..#Rows(M)]] 
		      : j in [1..#Rows(Transpose(M))]]>;
  end if;

end for;
end for;
end for;

bb := { x : x in Basis(NewRel) | not IsZero(x)};
printf "\nThere are %o relations defining NewRel.\n",#bb;

/* Now eliminate variables using the new degree 1 relations
*/

Newbb1 := { x : x in bb | Degree(x) eq 1};
NewRel1 := ideal<R | Newbb1>;
Groebner(NewRel1);
printf "Of these, %o relations are of degree 1, defining NewRel1\n",#Newbb1;
printf "The Groebner basis for NewRel1 has %o elements.\n",#Basis(NewRel1);

g := hom<R->R | [NormalForm(f(R.i),NewRel1) : i in [1..N]]>;

/* 92 relations, 45 of degree 1, Groebnerized is 18, leaving 8 b_i's
*/

/* Replace elements of XSq by their normal forms w.r.t. the
linear relations, i.e., reducing to the 16 + 8 remaining variables
*/

for i in [0..#XSq-1] do
for j in [0..#XSq[i+1]-1] do
   XSq[i+1,j+1] := Matrix(R,#A_bas(j),#A_bas(i+j),
                   [[ g(XSq[i+1,j+1][ii,jj])
		      : jj in [1..#A_bas(i+j)]]
		      : ii in [1..#A_bas(j)]]);
end for;
end for;

/*
Redefine Sq
*/

function Sq(i,j)
  return XSq[i+1,j+1];
end function;


/*
Recompute using the remaining 24 variables
*/


NewRel := ideal<R | 0>;

printf "\nRelations for Sq^8 and Sq^16 action on A(2)\n\n";

for b in [1..23] do
for a in [1..2*b-1] do
for j in [0..23-a-b] do
  M := Sq(b,j)*Sq(a,j+b);
  for k in [0..Truncate(a/2)] do
     if IsOdd(Binomial(b-k-1,a-2*k)) then
        M -:= Sq(k,j)*Sq(a+b-k,j+k);
     end if;
  end for;
  if not IsZero(M) then
    /* printf "Relation from Sq^%o Sq%o:\n%o\n\n ",a,b,M; */
    NewRel +:= ideal<R | &cat[[M[i,j] 
                      : i in [1..#Rows(M)]] 
		      : j in [1..#Rows(Transpose(M))]]>;
  end if;

end for;
end for;
end for;

bb := { x : x in Basis(NewRel) | not IsZero(x)};
printf "\nThere are now %o relations defining NewRel.\n",#bb;

/* Now eliminate variables using the new degree 1 relations
(no such relations!)
*/

Newbb1 := { x : x in bb | Degree(x) eq 1};
NewRel1 := ideal<R | Newbb1>;
Groebner(NewRel1);
printf "Of these, %o relations are of degree 1, defining NewRel1\n",#Newbb1;
printf "The Groebner basis for NewRel1 has %o elements.\n",#Basis(NewRel1);

/*
22 nonzero relations now, none linear, so NewRel can stay
as is and no new hom to replace variables is needed.
*/

/*
We are now reduced to the 24 variables
    a1, a2, a3, a21, a22, a23, a24,
    a47, a48, a49, a50, a60, a61, a62,
    a90, a102,
    b1, b2, b3, b4, b14, b15, b22, b26
*/

/* There are 22 relations between them 
*/


HERE ------;


/* Now import the solutions found in AA, rewriting these last 19 variables
in terms of the remaining 9 which determine the others.
Do this h : A --> A which rewrites the 10 that must change and sending
other variables to themselves, then compose this with g.
*/

/* define the 10 out of the penultimate 19 in terms
of the final 9 variables:
*/

a1 := R.150;
a2 := R.149;
a3 := R.148;
a21 := R.130;
a47 := R.104;
a48 := R.103;
a60 := R.91;
a61 := R.90;
a62 := R.89;
b1 := R.26;
b2 := R.25;
b3 := R.24;
b4 := R.23;

img := [

a1*a60 + a1*a62 + a2*a60 + a2*a61 + a3*a60 + b1,				/* b26	*/
R.2,
R.3,
R.4,
a1*a2 + a1*a21 + a1*a60 + a1*a62 + a3*a62 + a3 + b1 + b2 + 1,			/* b22	*/
R.6,
R.7,
R.8,
R.9,
R.10,
R.11,
a1*a2*a60 + a1*a2*a62 + a1*a60 + a1*a62 + a1 + a2*a3*a21*a62 + 
        a2*a3*a60*a62 + a2*a3*a60 + a2*a3*a62 + a2*a21*a47 + a2*a21*a48 + 
        a2*a21*a60 + a2*a21*a62 + a2*a21 + a2*a47*a60 + a2*a47*a62 + a2*a48*a60 
        + a2*a48*a62 + a2*a60 + a2 + a3*a21*a62 + a3*a60*a62 + a3*a60 + a3*a62 +
        a21*a47 + a21*a48 + a21*a60 + a21*a62 + a21 + a47*a60 + a47*a62 + 
        a48*a60 + a48*a62 + a60 + b2 + b4,					/* b15	*/
a1*a2*a60 + a1*a2*a62 + a1*a2 + a1*a21 + a2*a3*a21*a62 + 
        a2*a3*a60*a62 + a2*a3*a60 + a2*a3*a62 + a2*a3 + a2*a21*a47 + a2*a21*a48 
        + a2*a21*a60 + a2*a21*a62 + a2*a21 + a2*a47*a60 + a2*a47*a62 + 
        a2*a48*a60 + a2*a48*a62 + a2*a60 + a3*a21*a62 + a3*a21 + a3*a60*a62 + a3
        + a21*a47 + a21*a48 + a21*a60 + a21*a62 + a21 + a47*a60 + a47*a62 + 
        a48*a60 + a48*a62 + a60 + b1 + b3 + 1,					/* b14	*/
R.14,
R.15,
R.16,
R.17,
R.18,
R.19,
R.20,
R.21,
R.22,
R.23,
R.24,
R.25,
R.26,
R.27,
R.28,
R.29,
R.30,
R.31,
R.32,
R.33,
R.34,
R.35,
R.36,
R.37,
R.38,
R.39,
R.40,
R.41,
R.42,
R.43,
R.44,
R.45,
R.46,
R.47,
R.48,
a1*a21 + a2*a47 + a2 + a3*a60 + a21 + a61 + a62,				/* a102	*/
R.50,
R.51,
R.52,
R.53,
R.54,
R.55,
R.56,
R.57,
R.58,
R.59,
R.60,
 a1*a2*a21 + a1*a2 + a1*a21*a60 + a1*a21*a62 + a1*a60 + a1*a62 + 
        a2*a3*a60 + a2*a3 + a2*a21*a47 + a2*a21*a62 + a2*a47*a60 + a2*a47*a62 + 
        a2*a48 + a2*a60*a62 + a3*a21*a60 + a3*a21 + a3*a60*a62 + a3 + a21*a47 + 
        a21*a48 + a47*a60 + a47*a62 + a48*a60 + a48*a62 + a62 + 1,		/* a90	*/
R.62,
R.63,
R.64,
R.65,
R.66,
R.67,
R.68,
R.69,
R.70,
R.71,
R.72,
R.73,
R.74,
R.75,
R.76,
R.77,
R.78,
R.79,
R.80,
R.81,
R.82,
R.83,
R.84,
R.85,
R.86,
R.87,
R.88,
R.89,
R.90,
R.91,
R.92,
R.93,
R.94,
R.95,
R.96,
R.97,
R.98,
R.99,
R.100,
a1*a2 + a1*a21 + a1*a60 + a1*a62 + a1 + a2*a3*a62 + a2*a47 + 
        a2*a48 + a2*a60 + a2*a62 + a2 + a3*a21*a62 + a3*a60*a62 + a21*a47 + 
        a21*a48 + a21*a60 + a21*a62 + a21 + a47*a60 + a47*a62 + a47 + a48*a60 + 
        a48*a62 + a48 + a60 + 1,					/* a50	*/
a1*a2*a60 + a1*a2*a62 + a1*a21*a60 + a1*a21*a62 + a1 + 
        a2*a3*a21*a62 + a2*a3*a60*a62 + a2*a3*a60 + a2*a3 + a2*a21*a48 + 
        a2*a21*a60 + a2*a21 + a2*a48*a60 + a2*a48*a62 + a2*a60*a62 + a2*a60 + 
        a3*a21*a60 + a3*a21*a62 + a3*a21 + a3*a60 + a3 + a21*a60 + a21*a62 + a47
        + a48 + a60 + a62,						/* a49	*/
R.103,
R.104,
R.105,
R.106,
R.107,
R.108,
R.109,
R.110,
R.111,
R.112,
R.113,
R.114,
R.115,
R.116,
R.117,
R.118,
R.119,
R.120,
R.121,
R.122,
R.123,
R.124,
R.125,
R.126,
a1*a2*a21 + a1*a2*a60 + a1*a2*a62 + a2*a3*a21*a62 + a2*a3*a60*a62 +
        a2*a3*a62 + a2*a21*a47 + a2*a21*a48 + a2*a21*a60 + a2*a21*a62 + a2*a21 +
        a2*a47*a60 + a2*a47*a62 + a2*a48*a60 + a2*a48*a62 + a2*a62 + a3 + 1,	/*a24	*/
a1*a2*a21 + a1*a2*a60 + a1*a2*a62 + a1*a21 + a2*a3*a21*a62 + 
        a2*a3*a60*a62 + a2*a3*a62 + a2*a21*a47 + a2*a21*a48 + a2*a21*a60 + 
        a2*a21*a62 + a2*a21 + a2*a47*a60 + a2*a47*a62 + a2*a47 + a2*a48*a60 + 
        a2*a48*a62 + a2*a62 + a2 + a3*a60 + a3*a62 + a3 + a21 + 1,		/*a23	*/
a1*a21 + a1 + a2*a47 + a3*a60 + a21,						/*a22	*/
R.130,
R.131,
R.132,
R.133,
R.134,
R.135,
R.136,
R.137,
R.138,
R.139,
R.140,
R.141,
R.142,
R.143,
R.144,
R.145,
R.146,
R.147,
R.148,
R.149,
R.150

];

h := hom<R->R | img>;

for i in [0..#XSq-1] do
for j in [0..#XSq[i+1]-1] do
   XSq[i+1,j+1] := Matrix(R,#A_bas(j),#A_bas(i+j),
                   [[ h(XSq[i+1,j+1][ii,jj])
		      : jj in [1..#A_bas(i+j)]]
		      : ii in [1..#A_bas(j)]]);
end for;
end for;

printf  "XSq applied h\n";

I :=  ideal<R | [R.i^2-R.i : i in [1..N]]>;
rrr := {h(NormalForm(x,I)) : x in rel2};

printf  "First rrr\n";

rrr := {h(NormalForm(x,I)) : x in rrr };
printf  "Second rrr\n";


rrr := {NormalForm(x,I) : x in rrr | not IsZero(x)};
printf  "Third rrr\n";

printf "%o\n",rrr;

Rel := ideal<R | rrr>;
printf  "Rel\n";
/* Groebner(Rel); */

printf "Groebner\n";

for i in [0..#XSq-1] do
for j in [0..#XSq[i+1]-1] do
   XSq[i+1,j+1] := Matrix(R,#A_bas(j),#A_bas(i+j),
                   [[ NormalForm(XSq[i+1,j+1][ii,jj],Rel)
                      : jj in [1..#A_bas(i+j)]]
                      : ii in [1..#A_bas(j)]]);
end for;
end for;

printf "Reduced XSq mod Rel\n";



/*
Redefine Sq
*/

function Sq(i,j)
  return XSq[i+1,j+1];
end function;



/*
Here, we want to say the lower left blocks are 0 to find the actions which preserve the
short exact sequence coming from the Q_2 action.

Oh, there are only three such actions, with three entries, all equal to a60.
So a60 = 0 iff the ses is preserved!

Next, find the equations for the difference between the action on Im Q_2 and
on the quotient by it.
*/

function eqq(i)
if i eq 0 then return R!1; else return R!0; end if;
end function;

XIn := [* Matrix(R,#B_bas(n),#A_bas(n),
        [[eqq(i-j) : j in [1..#A_bas(n)]]
	              : i in [1..#B_bas(n)]]) : n in [0..55] *];

XPn := [* Matrix(R,#A_bas(n),#B_bas(n),
        [[eqq(i-j) : j in [1..#B_bas(n)]]
	              : i in [1..#A_bas(n)]]) : n in [0..55] *];

XJn := [* Matrix(R,#B_bas(n-7),#A_bas(n),
        [[eqq(i-j+#B_bas(n)) : j in [1..#A_bas(n)]]
	              : i in [1..#B_bas(n-7)]]) : n in [0..55] *];

XQn := [* Matrix(R,#A_bas(n),#B_bas(n-7),
        [[eqq(i-j-#B_bas(n)) : j in [1..#B_bas(n-7)]]
	              : i in [1..#A_bas(n)]]) : n in [0..55] *];

Diffs := [ [* XIn[n+1] * Sq(k,n) * XPn[n+k+1]
              - XJn[n+7+1] * Sq(k,n+7) * XQn[n+k+7+1]
             : n in [0..16] *]
	     : k in [0..23] ];



HERE			This is here to be easy to find.
--------------------------;   This is here to stop magma from executing any more commands

/* Sq^8 and Sq^16 in this final version was written out in the
files Sq8 and Sq16
*/


/* Set numerical values for these last 9 and write out
a moddef file
*/

val := [
R.1,
R.2,
R.3,
R.4,
R.5,
R.6,
R.7,
R.8,
R.9,
R.10,
R.11,
R.12,
R.13,
R.14,
R.15,
R.16,
R.17,
R.18,
R.19,
R.20,
R.21,
R!0,		/* d3	*/
R!0,		/* d2	*/
R!0,		/* d1	*/
R!0,		/* c1	*/
R.26,
R.27,
R.28,
R.29,
R.30,
R.31,
R.32,
R.33,
R.34,
R.35,
R.36,
R.37,
R.38,
R.39,
R.40,
R.41,
R.42,
R.43,
R.44,
R.45,
R.46,
R.47,
R.48,
R.49,
R.50,
R.51,
R.52,
R.53,
R.54,
R.55,
R.56,
R.57,
R.58,
R.59,
R.60,
R.61,
R.62,
R.63,
R.64,
R.65,
R.66,
R.67,
R.68,
R.69,
R.70,
R.71,
R.72,
R.73,
R.74,
R.75,
R.76,
R.77,
R.78,
R.79,
R.80,
R.81,
R.82,
R.83,
R.84,
R.85,
R.86,
R.87,
R.88,
R.89,
R.90,
R!0,			/* b1	*/
R.92,
R.93,
R.94,
R.95,
R.96,
R!0,			/* a23	*/
R.98,
R.99,
R.100,
R.101,
R.102,
R.103,
R.104,
R.105,
R.106,
R!0,			/* a13	*/
R.108,
R.109,
R.110,
R.111,
R.112,
R.113,
R.114,
R.115,
R.116,
R.117,
R!0,			/* a2	*/
R!0 			/* a1	*/
];

for aa in [0,1] do
for bb in [0,1] do
for cc in [0,1] do
for dd in [0,1] do
for ee in [0,1] do
val[119] := aa;
val[118] := bb;
val[107] := cc;
val[97] := dd;
val[91] := ee;

if IsEven(aa*(bb+cc) + bb*(cc*(dd + ee) + cc + dd + ee + 1) + ee) then

for ff in [0,1] do
for gg in [0,1] do
for hh in [0,1] do
for ii in [0,1] do

val[25] := ii;
val[24] := hh;
val[23] := gg;
val[22] := ff;

k := hom<R->R | val>;

VSq := [ [* Matrix(R,#A_bas(j),#A_bas(i+j),
                  [ R!0 : k in [1..#A_bas(j)*#A_bas(i+j) ]])
         : j in [0..24] *]
	 : i in [0..23] ];

for i in [0..#XSq-1] do
for j in [0..#XSq[i+1]-1] do
   VSq[i+1,j+1] := Matrix(R,#A_bas(j),#A_bas(i+j),
                   [[ k(XSq[i+1,j+1][ii,jj])
		      : jj in [1..#A_bas(i+j)]]
		      : ii in [1..#A_bas(j)]]);
end for;
end for;

filename := "Aof2-" cat 
            &cat[IntegerToString(Integers()!k(R.x)) 
		: x in [119,118,107,97,91,25,24,23,22]];
SetOutputFile(filename);

&+[#A_bas(n) : n in [0..23]];
degs := &cat[[n : i in [1..#A_bas(n) ]]: n in [0..23]];
count := 0;
for i in [1..#degs] do
  printf "%o",degs[i];
  count +:= 1;
  if count ge 10 then
     printf "\n";
     count := 0;
  else
     printf " ";
  end if;
end for;
printf "\n\n";

mons := &cat[ A_bas(n) : n in [0..23]];

for gg in [1..#mons] do
d := degs[gg];
for i in [1..23-d] do
  vv := VSq[i+1,d+1][Index(A_bas(d),mons[gg])];
  if not IsZero(vv) then
     mm := [Index(mons,A_bas(d+i)[ii]) 
             : ii in [1..#A_bas(d+i)] | not IsZero(vv[ii])];
     printf "%o %o %o",gg-1,i,#mm;
     for xx in mm do
       printf " %o",xx-1;
     end for;
     printf "\n";
  end if;
end for;
printf "\n";
end for;
printf "\n";



UnsetOutputFile();


end for; /* ii */
end for; /* hh */
end for; /* gg */
end for; /* ff */
end if;
end for; /* ee */
end for; /* dd */
end for; /* cc */
end for; /* bb */
end for; /* aa */


/* The rest of this was written to apply to Z.
It checks that the Adem relations are now satisfied.

Then computes CSq, the conjugate squaring ops, 
aND LOOKS FOR THE EQUATIONS which tell the dual
structure.

Do that for A(2).

Then compare the results for Z =A(2)/Q_2 A(2) and
Q_2 A(2).
*/




f := hom<R->R |
   [ R.1, R.2, R!1, R.2, R!0, R.1, R!1, R.2, R!1, R!1,
     R!0, R.2, R.13, (R.1+R.13+R.1*R.13+R.2*R.23), R!0,
       R.2, R.13, R!0, R!0, R!1,
     R.13+1, R!1, R.23, R!1, R.13, R!1, R.13, R.23, R.29]>;


/* Apply f should make all the relations zero
*/


for b in [1..16] do
for a in [1..2*b-1] do
printf "%o,%o; ",a,b;
for j in [0..16-a-b] do
  M := Sq(b,j)*Sq(a,j+b);
  for k in [0..Truncate(a/2)] do
     if IsOdd(Binomial(b-k-1,a-2*k)) then
        M -:= Sq(k,j)*Sq(a+b-k,j+k);
     end if;
  end for;
  if not IsZero(M) then
    d := Dim(M);
    MM := Matrix(R,d[1],d[2],[[ f(M[i,j]) : j in [1..d[2]]] : i in [1..d[1]]]);
    if not IsZero(MM) then
       printf "Stiil nonzero Relation from Sq^%o Sq%o:\n%o\n\n ",a,b,MM;
    end if;
  end if;

end for;
end for;
end for;

/* It does.   Record the new Sq^i
*/

for i in [1..#XSq] do
for j in [1..#XSq[i]] do
if not IsZero(XSq[i,j]) then
  d := Dim(XSq[i,j]);
  XSq[i,j] := Matrix(R,d[1],d[2],
              [[ f(XSq[i,j][ii,jj]) : jj in [1..d[2]]] 
	         : ii in [1..d[1]]]);
end if;
end for;
end for;

/*
Redefine Sq
*/

function Sq(i,j)
  return XSq[i+1,j+1];
end function;



/* Compute dual, by computing the conjugate action:
   CSq[i+1][j+1] is chi(Sq^i) : Z_j ---> Z_i+j
*/

CSq := XSq;

for a in [3..16] do
  for j in [0..#CSq[a+1]-a-1] do
     CSq[a+1][j+1] +:=  &+[  CSq[k+1,j+1]* XSq[a-k+1,j+k+1]
                            : k in [1..a-1] ];
  end for;
end for;

rev := func< M | Transpose(ReverseRows(ReverseColumns(M))) >;
\end{verbatim}

\section{MAGMA code, general case, third step}
\label{codegen3}

\begin{verbatim}

/*
Third step:  use the quadratic relations to continue eliminating variables 
as far as possible.
*/


S<a1,a2,a3,a21,a22,a23,a24,a47,a48,a49,a50,a60,a61,a62,a90,a102,b1,b2,b3,b4,b14,b15,b22,b26>
   := PolynomialRing(GF(2),24);

vars :=
[ a1, a2, a3, a21, a22, a23, a24,
    a47, a48, a49, a50, a60, a61, a62,
    a90, a102,
    b1, b2, b3, b4, b14, b15, b22, b26
    ];

/*
This list of relations was generated by computing
 {g(x) : x in Basis(NewRel)};
above.
*/

rels :=
  [ a60*a3 + a47*a2 + a22 + a21*a1 + a21 + a1,
    a90 + a62*a3 + a50 + a49 + a23 + a3 + 1,
    a49*a2 + a48*a2 + a24 + a23 + a21 + a3,
    b15 + b4 + b2 + a50*a2 + a50 + a48*a2 + a48 + a47*a2 + a47 + a23*a2 + a23 +
        a21*a2 + a21 + a3*a2 + a3 + a2^2 + a2*a1,
    a50*a21 + a50 + a49 + a48*a21 + a48 + a47*a21 + a24 + a23*a21 + a23 + a21^2
        + a21*a3 + a21*a2 + a21*a1,
    a90 + a62*a23 + a60*a23 + a50 + a48 + a47 + a23*a21 + a23*a2 + a21 + a3 + a2
        + a1,
    a102 + a62*a3 + a62 + a61 + a50*a2 + a23 + a3 + 1,
    a62*a3 + a60*a3 + a49*a2 + a48*a2 + a47*a2 + a21*a1 + a3 + a2,
    b14 + b3 + b1 + a62*a3 + a62*a1 + a60*a3 + a60*a1 + a50*a2 + a50 + a48*a2 +
        a48 + a47*a2 + a47 + a23*a2 + a23 + a21*a3 + a21*a2 + a21*a1 + a21 +
        a2^2 + a2 + a1 + 1,
    a50*a2 + a49*a2 + a48*a2 + a23 + a21 + 1,
    a102 + a62*a60 + a62*a47 + a62*a1 + a62 + a61*a60 + a61*a21 + a61*a2 + a61 +
        a60*a50 + a60*a48 + a60*a47 + a60*a23 + a60*a21 + a60*a3 + a60*a2 +
        a60*a1 + a60 + a21 + a2,
    a62*a3 + a60*a3 + a50*a2 + a47*a2 + a23 + a21*a1 + a21 + a3 + a2 + 1,
    a62*a23 + a60*a23 + a50*a21 + a50 + a48*a21 + a47*a21 + a47 + a23*a2 + a23 +
        a22 + a21^2 + a21*a3 + a21*a2 + a21*a1 + a21 + 1,
    a90 + a62*a49 + a62 + a60*a49 + a49*a21 + a49*a2 + a49 + a48 + a47 + a23 +
        a21 + a3 + a2 + a1,
    a50*a2 + a24 + a3 + 1,
    a62*a3 + a24 + a23 + a22 + a2 + a1,
    a62*a23 + a60*a23 + a49 + a48 + a47 + a24 + a23*a21 + a23*a2 + a22 + a21 +
        1,
    a62*a3 + a50*a2 + a23 + a22 + a3 + a2 + a1 + 1,
    b26 + b14 + b3 + a62*a3 + a61*a2 + a60*a3 + a60*a2 + a50*a2 + a50 + a48*a2 +
        a48 + a47 + a23*a2 + a23 + a22 + a21*a3 + a21*a2 + a2^2 + a2 + 1,
    a102 + a62*a23 + a62 + a61 + a60*a23 + a50*a21 + a50 + a48*a21 + a47*a21 +
        a47 + a23*a2 + a23 + a21^2 + a21*a3 + a21*a2 + a21*a1 + a21 + a2 + a1 +
        1,
    b22 + b15 + b14 + b4 + b3 + a62*a3 + a60*a3 + a24 + a23 + a22 + a21*a3 +
        a3*a2,
    b15 + b14 + b4 + b3 + b2 + b1 + a62*a3 + a62*a1 + a60*a3 + a60*a1 + a21*a3 +
        a21*a1 + a3*a2 + a3 + a2*a1 + a2 + a1 + 1
];

/* Since the variables will all be either 0 or 1, each is
equal to its own square.   Hence we reduce mod x=x^2
*/

I := ideal<S | [x-x^2 : x in vars]>;
rels := [NormalForm(x,I) : x in rels];

rels;


/*
test whether v is a term of r, 
and v is not a factor of any higher degree term.
If so then the relation r can be used to eliminate v.
*/

safe := function(v,r)	
  tt := Terms(r);
  if v in tt then
    tt2 := {x : x in tt | Degree(x) gt 1};
    f2 := &cat [Factorization(x) : x in tt2];
    if v in {x[1] : x in f2} then
      return false;
    else
      return true;
    end if;
  else
    return false;
  end if;
end function;

/*
Next few commands just check where we stand
*/

terms := &cat [ Terms(x) : x in rels];
terms := {x : x in terms};
#terms," terms";

t1 := {x : x in terms | Degree(x) eq 1};
tt2 := {x : x in terms | Degree(x) eq 2};

t2 := &cat[Factorization(x) : x in tt2];
t2 := { x[1] : x in t2};

#t1,#t2,#(t1 meet t2),"\n";


t1 := Sort([x : x in t1]);
t2 := Sort([x : x in t2]);

/*
Now, to work
*/

/* Each of the following steps arises by examining, by hand,
the chance to use the relations to eliminate variables,
starting at the end (d_24) and workiing forward toward a_1.
If a variable occurs alone as a degree 1 term in a relation
(checked by the command named 'safe'), we can use that relation
to eliminate it from the remaining relations.

The examination is done by executing

[<j, [<i,safe(S.j,rels[i]),rels[i]> 
       : i in [1..#rels] | S.j in Terms(rels[i])]> 
       : j in [1..#vars]];

For each var S.j, and for each relation rels[i] which contains S.j
as a term, show the index i of the relation, whether the relation
is safe to use to eliminate S.j, and the relation.  Among these,
we choose the largest j which has a safe relation, and among the
relations, we use the shortest.

*/

/* eliminate 
 b26 = S.24 using rels[19]
 b22 = S.23 using rels[21]
 b15 = S.22 using rels[4]
 b14 = S.21 using rels[9]
 a102 = S.16 using rels[7]
 a90 = S.15 using rels[2]
 a50 = S.11 using rels[14]
 a49 = S.10 using rels[17]
 a24 = S.7 using rels[16]
 a23 = S.6 using rels[15]
 a22 = S.5 using rels[1]
*/

/* Move this paragraph forward as the calculation progresses	*/
[<j,S.j, [<i,safe(S.j,rels[i]),rels[i]> 
       : i in [1..#rels] | S.j in Terms(rels[i])]> 
       : j in [1..#vars]];
HERE ------;   This syntax error stops execution

safe(S.24,rels[19]);
g := hom<S->S | [S.i : i in [1..23]] cat
                [S.24+rels[19]]
		>;
rels := [NormalForm(g(x),I) : x in rels];
S.24;
g eq g*g;		/* sanity check	*/


safe(S.23,rels[21]);
g1:= hom<S->S | [S.i : i in [1..22]] cat
                [S.23+rels[21]] cat
                [S.24]
		>;
g := g*g1;		/* First apply g, then g1	*/
g := hom<S->S | [NormalForm(g(S.i),I) : i in [1..24]]>;
rels := [NormalForm(g(x),I) : x in rels];
S.23;
g eq g*g;		/* sanity check	*/

safe(S.22,rels[4]);
g1:= hom<S->S | [S.i : i in [1..21]] cat
                [S.22+rels[4]] cat
                [S.i : i in [23..24]]
		>;
g := g*g1;		/* First apply g, then g1	*/
g := hom<S->S | [NormalForm(g(S.i),I) : i in [1..24]]>;
rels := [NormalForm(g(x),I) : x in rels];
S.22;
g eq g*g;		/* sanity check	*/


safe(S.21,rels[9]);
g1:= hom<S->S | [S.i : i in [1..20]] cat
                [S.21+rels[9]] cat
                [S.i : i in [22..24]]
		>;
g := g*g1;		/* First apply g, then g1	*/
g := hom<S->S | [NormalForm(g(S.i),I) : i in [1..24]]>;
rels := [NormalForm(g(x),I) : x in rels];
S.21;
g eq g*g;		/* sanity check	*/



safe(S.16,rels[7]);
g1:= hom<S->S | [S.i : i in [1..15]] cat
                [S.16+rels[7]] cat
                [S.i : i in [17..24]]
		>;
g := g*g1;		/* First apply g, then g1	*/
g := hom<S->S | [NormalForm(g(S.i),I) : i in [1..24]]>;
rels := [NormalForm(g(x),I) : x in rels];
S.16;
g eq g*g;		/* sanity check	*/


safe(S.15,rels[2]);
g1:= hom<S->S | [S.i : i in [1..14]] cat
                [S.15+rels[2]] cat
                [S.i : i in [16..24]]
		>;
g := g*g1;		/* First apply g, then g1	*/
g := hom<S->S | [NormalForm(g(S.i),I) : i in [1..24]]>;
rels := [NormalForm(g(x),I) : x in rels];
S.15;
g eq g*g;		/* sanity check	*/


safe(S.11,rels[14]);
g1:= hom<S->S | [S.i : i in [1..10]] cat
                [S.11+rels[14]] cat
                [S.i : i in [12..24]]
		>;
g := g*g1;		/* First apply g, then g1	*/
g := hom<S->S | [NormalForm(g(S.i),I) : i in [1..24]]>;
rels := [NormalForm(g(x),I) : x in rels];
S.11;
g eq g*g;		/* sanity check	*/


safe(S.10,rels[17]);
g1:= hom<S->S | [S.i : i in [1..9]] cat
                [S.10+rels[17]] cat
                [S.i : i in [11..24]]
		>;
g := g*g1;		/* First apply g, then g1	*/
g := hom<S->S | [NormalForm(g(S.i),I) : i in [1..24]]>;
rels := [NormalForm(g(x),I) : x in rels];
S.10;
g eq g*g;		/* sanity check	*/


safe(S.7,rels[16]);
g1:= hom<S->S | [S.i : i in [1..6]] cat
                [S.7+rels[16]] cat
                [S.i : i in [8..24]]
		>;
g := g*g1;		/* First apply g, then g1	*/
g := hom<S->S | [NormalForm(g(S.i),I) : i in [1..24]]>;
rels := [NormalForm(g(x),I) : x in rels];
S.7;
g eq g*g;		/* sanity check	*/


safe(S.6,rels[15]);
g1:= hom<S->S | [S.i : i in [1..5]] cat
                [S.6+rels[15]] cat
                [S.i : i in [7..24]]
		>;
g := g*g1;		/* First apply g, then g1	*/
g := hom<S->S | [NormalForm(g(S.i),I) : i in [1..24]]>;
rels := [NormalForm(g(x),I) : x in rels];
S.6;
g eq g*g;		/* sanity check	*/


safe(S.5,rels[1]);
g1:= hom<S->S | [S.i : i in [1..4]] cat
                [S.5+rels[1]] cat
                [S.i : i in [6..24]]
		>;
g := g*g1;		/* First apply g, then g1	*/
g := hom<S->S | [NormalForm(g(S.i),I) : i in [1..24]]>;
rels := [NormalForm(g(x),I) : x in rels];
S.5;
g eq g*g;		/* sanity check	*/

/* Now count the number of solutions.
Since the b_i do not occur in the remaining 3 relations,
we need only consider the subring generated by the a_i
*/

/* 
 Now reduce to the 9+4=13 vars and 3 rels above.
 The rels do not involve the b_i, so there are 16 = 2^4 values of Sq^16
 determined by b1,b2,b3,b4 for each Sq^8 (determined by a1,...,a62).
*/

T<a1,a2,a3,a21,a47,a48,a60,a61,a62,b1,b2,b3,b4> :=
      PolynomialRing(GF(2),13);


newrels := 
[
    a1*a21 + a1*a60 + a1*a62 + a2*a3*a21*a60*a62 + a2*a3*a21*a60 + a2*a3*a21 + a2*a3*a60 + a2*a3*a62 + a2*a3 + a2*a21*a47*a60 + a2*a21*a47*a62 + a2*a21*a47 + a2*a21*a48*a60 + 
        a2*a21*a48*a62 + a2*a21*a60 + a2*a48 + a2 + a3*a21*a60*a62 + a3*a21 + a3*a60*a62 + a3*a60 + a21*a47*a60 + a21*a47*a62 + a21*a48*a60 + a21*a48*a62 + a21*a62 + a21 + a47*a60 + 
        a47*a62 + a48*a60 + a48*a62 + a48 + a62 + 1,
    a1*a2*a21*a60 + a1*a2*a60*a62 + a1*a21 + a1*a60*a62 + a1*a60 + a1*a62 + a2*a3*a21*a60*a62 + a2*a3*a60*a62 + a2*a21*a47*a60 + a2*a21*a48*a60 + a2*a21*a60*a62 + a2*a47*a60*a62 + 
        a2*a47*a60 + a2*a47 + a2*a48*a60*a62 + a2*a61 + a3*a21*a60*a62 + a21*a47*a60 + a21*a48*a60 + a21*a60*a62 + a21*a61 + a47*a60*a62 + a47*a60 + a47*a62 + a48*a60*a62 + a48*a60 + 
        a60*a61 + a60*a62,
    a1*a2*a21*a60 + a1*a2*a21*a62 + a1*a2*a60 + a1*a2*a62 + a1*a2 + a1*a21 + a2*a3*a21*a60 + a2*a3*a21 + a2*a3*a60*a62 + a2*a21*a48 + a2*a21*a62 + a2*a21 + a2*a48*a60 + a2*a48*a62 + 
        a2*a60*a62 + a2*a62 + a2 + a3*a60 + a3*a62 + a3

];


count := 0;
for r1 in GF(2) do
for r2 in GF(2) do
for r3 in GF(2) do
for r4 in GF(2) do
for r5 in GF(2) do
for r6 in GF(2) do
for r7 in GF(2) do
for r8 in GF(2) do
for r9 in GF(2) do
h := hom<T->T | [r1,r2,r3,r4,r5,r6,r7,r8,r9,0,0,0,0]>;
if &and[IsZero(h(x)) : x in newrels] then
  count +:= 1;
  printf "%4o : %o\n",count, [r1,r2,r3,r4,r5,r6,r7,r8,r9];
end if;
end for;
end for;
end for;
end for;
end for;
end for;
end for;
end for;
end for;


/*
Print out the expressions for the variables we have replaced
in terms of those which remain for use in the last step of
the main MAGMA code A2gen.
*/

printf "\n\nExpressions for the variables we have replaced in\n";
printf "terms of those which remain:\n";

[<S.i,g(S.i)> : i in [1..#vars] | not S.i eq g(S.i)];

\end{verbatim}

\section{MAGMA code for computing duals}
\label{codedual}

We must first compute
\[
\chi(Sq^a) = \sum_{k=1}^a Sq^k \chi(Sq^{a-k}).
\]

The MAGMA code for the conjugate is short:
\begin{verbatim}

CSq := XSq;

for a in [3..16] do
  for j in [0..#CSq[a+1]-a-1] do
     CSq[a+1][j+1] +:=  &+[  CSq[k+1,j+1]* XSq[a-k+1,j+k+1]
                            : k in [1..a-1] ];
  end for;
end for;

\end{verbatim}

Starting by setting {\tt CSq} equal to {\tt XSq} 
initializes {\tt CSq[a+1]}, which will represent the linear transformations
$\chi(Sq^a)$, by including the $k=a$ term of the sum.
This is the whole of $\chi(Sq^a)$ for $a < 3$, so our loop
adding the remaining terms runs from $a=3$ to $a=16$ (the higher
$Sq^a$ are not needed in order to identify the $\cA$-module structure).

We then use the data in the Tables {\tt Coeffs} below to 
compute the coefficients
recorded in parts (2) and (3) of Theorems~\ref{thm:symmetric},
\ref{thm:general}, and \ref{thm:Baction}.\\

\begin{verbatim}
/*  
Now compute D : R -> R expressing the duality 
*/

/* Table of coefficients for V_sym

Coeffs := [
 < 119, a1, 8, [0,0,0], [5,1,0]>,
 < 118, a2, 8, [0,0,0], [2,2,0]>,
 <  91, b1, 8, [0,0,0], [1,0,1]>,
 < 107, a13, 8, [4,0,0], [6,2,0]>,
 <  97, a23, 8, [0,2,0], [5,3,0]>,
 <  25, c1, 16, [0,0,0], [7,3,0]>,
 <  24, d1, 16, [0,0,0], [6,1,1]>,
 <  23, d2, 16, [0,0,0], [3,2,1]>,
 <  22, d3, 16, [0,0,0], [0,3,1]>
 ];
*/

/* Table of coefficients for V_gen
*/

Coeffs := [
 < 150, a1, 8, [0,0,0], [5,1,0]>,
 < 149, a2, 8, [0,0,0], [2,2,0]>,
 < 148, a3, 8, [0,0,0], [1,0,1]>,
 < 130, a21, 8, [4,0,0], [6,2,0]>,
 < 104, a47, 8, [0,2,0], [5,3,0]>,
 < 103, a48, 8, [0,2,0], [7,0,1]>,
 <  91, a60, 8, [0,0,1], [6,3,0]>,
 <  90, a61, 8, [0,0,1], [5,1,1]>,
 <  89, a62, 8, [0,0,1], [2,2,1]>,
 <  26, b1, 16, [0,0,0], [7,3,0]>,
 <  25, b2, 16, [0,0,0], [6,1,1]>,
 <  24, b3, 16, [0,0,0], [3,2,1]>,
 <  23, b4, 16, [0,0,0], [0,3,1]>
 ];

theta := func< r | [7-r[1], 3-r[2], 1-r[3]]>;
deg := func< r | r[1] + 3*r[2] + 7*r[3]>;

Dimg := [ R.i : i in [1..N] ]; 	/* initial duality hom images	*/

for c in Coeffs do
  k := deg(c[4]);
  i := c[3];
  x := c[4];
  y := c[5];
  Dimg[c[1]] := CSq[i+1,24-k-i]
                  [Index(A_bas(23-k-i),theta(y))]
		  [Index(A_bas(23-k),theta(x))];
end for;

D := hom<R->R | Dimg>;
\end{verbatim}

The three Theorems in Section~\ref{sec:duality} then record the values of the
homomorphism $D$ on the generators.
The modifications to the code above for $\cB(2)$ should be evident.

\section{MAGMA output from first two steps}
\label{firstout}

\begin{verbatim}

Loading "xA8"

Initial Sq^8:
[*
[a1 a2 b1],

[a3 a4 a5 b2],

[a6 a7 a8 b3 b4],

[ a9 a10  b5  b6]
[a11 a12  b7  b8],

[a13 a14  b9 b10]
[a15 a16 b11 b12],

[a17 a18 b13 b14 b15]
[a19 a20 b16 b17 b18],

[a21 b19 b20 b21]
[a22 b22 b23 b24]
[a23 b25 b26 b27],

[a24 b28 b29]
[a25 b30 b31]
[a26 b32 b33]
[  0  a1  a2],

[a27 b34 b35 b36]
[a28 b37 b38 b39]
[  0  a3  a4  a5],

[b40 b41 b42]
[b43 b44 b45]
[b46 b47 b48]
[ a6  a7  a8],

[b49 b50]
[b51 b52]
[b53 b54]
[ a9 a10]
[a11 a12],

[b55 b56]
[b57 b58]
[a13 a14]
[a15 a16],

[b59 b60]
[b61 b62]
[a17 a18]
[a19 a20],

[b63]
[b64]
[a21]
[a22]
[a23],

[b65]
[a24]
[a25]
[a26],

[b66]
[a27]
[a28],

Matrix with 4 rows and 0 columns,

Matrix with 3 rows and 0 columns,

Matrix with 2 rows and 0 columns,

Matrix with 2 rows and 0 columns,

Matrix with 2 rows and 0 columns,

Matrix with 1 row and 0 columns,

Matrix with 1 row and 0 columns,

Matrix with 1 row and 0 columns,

Matrix with 0 rows and 0 columns
*]

Computing relations for Sq^8 action only on A(2)

There are 2220 relations defining Rel.

Of these, 452 relations are of degree 1, defining Rel1

The Groebner basis for Rel1 has 81 elements.

Initial Sq^16:
[*
[c1 d1 d2 d3],

[d4 d5 d6],

[d7 d8],

[ d9 d10]
[d11 d12],

[d13 d14]
[d15 d16],

[d17]
[d18],

[d19]
[d20]
[d21],

[d22]
[d23]
[d24]
[ c1],

Matrix with 3 rows and 0 columns,

Matrix with 4 rows and 0 columns,

Matrix with 5 rows and 0 columns,

Matrix with 4 rows and 0 columns,

Matrix with 4 rows and 0 columns,

Matrix with 5 rows and 0 columns,

Matrix with 4 rows and 0 columns,

Matrix with 3 rows and 0 columns,

Matrix with 4 rows and 0 columns,

Matrix with 3 rows and 0 columns,

Matrix with 2 rows and 0 columns,

Matrix with 2 rows and 0 columns,

Matrix with 2 rows and 0 columns,

Matrix with 1 row and 0 columns,

Matrix with 1 row and 0 columns,

Matrix with 1 row and 0 columns,

Matrix with 0 rows and 0 columns
*]

Relations for Sq^8 and Sq^16 action on A(2)


There are 318 relations defining NewRel.

Of these, 50 relations are of degree 1, defining NewRel1

The Groebner basis for NewRel1 has 19 elements.

Almost final relations:
{
0,
d14 + d13 + d3 + d2 + d1 + c1 + b1*a13 + b1 + a13*a1 + a2 + a1 + 1,
b27*a13 + b27 + b26 + b25*a13 + b25 + b10 + b9*a13 + b9 + b1*a13 + a23*a13 + 
a13^2 + a13*a2 + a13*a1,
b27*a2 + b10 + b1 + 1,
b52 + b27 + b26 + b9 + b1*a2 + b1 + 1,
d13 + d2 + c1 + b27*a2 + b27 + b25*a2 + b25 + b9*a2 + b9 + b1*a13 + b1*a2 + 
a23*a2 + a23 + a13*a2 + a13*a1 + a13 + a2^2 + a2*a1 + a2 + a1 + 1,
b26*a2 + b25*a2 + b1*a2 + b1 + a23*a2 + a13*a1 + a2,
b27*a2 + b9 + b1*a2 + b1 + a14 + a2 + a1 + 1,
d14 + d3 + d1 + b27*a2 + b27 + b25*a2 + b25 + b9*a2 + b9 + b1*a2 + b1 + a23*a2 +
a23 + a13*a2 + a13 + a2^2 + a2*a1,
d13 + d2 + c1 + b27*a2 + b27 + b25*a2 + b25 + b9*a2 + b9 + b1*a13 + b1*a2 + a23 
+ a14 + a13*a2 + a2^2 + a2*a1 + a2 + 1,
b26 + b25 + b10 + b9*a13 + a23 + a14 + a13 + 1,
b27*a2 + b9 + b1*a2 + b1 + a23*a2 + a13*a1 + a13 + a2 + 1,
d21 + d14 + d13 + d3 + d2 + b10 + b9 + b1*a13 + a14,
b27*a13 + b27 + b25*a13 + b9 + b1*a13 + a23*a13 + a23 + a14 + a13^2 + a13*a2 + 
a13*a1 + a13 + 1,
b52 + b27 + b25 + b9*a13 + b1 + a23 + a13 + a2 + a1,
b10 + b9 + b1*a2 + a14 + a2 + a1,
b52 + b26*a13 + b26 + b25 + b9 + b1 + a23 + a13 + a1,
b26*a2 + b25*a2 + b10 + b9 + b1 + a13,
a23*a2 + a14 + a13*a1 + a13 + a1,
b27*a2 + b26*a2 + b25*a2 + b9 + a13 + 1
}


In file "xA8", line 781, column 27:
>> ---------------------;   This is here to stop magma from executing any further
                             ^
User error: bad syntax

\end{verbatim}

\section{MAGMA output from third step}
\label{thirdstep}

\begin{verbatim}

Loading "AA"
[
    a1*a13 + a1 + a2 + a13*b1 + b1 + c1 + d1 + d2 + d3 + d13 + d14 + 1,
    a1*a13 + a2*a13 + a13*a23 + a13*b1 + a13*b9 + a13*b25 + a13*b27 + a13 + b9 + b10 + b25
        + b26 + b27,
    a2*b27 + b1 + b10 + 1,
    a2*b1 + b1 + b9 + b26 + b27 + b52 + 1,
    a1*a2 + a1*a13 + a1 + a2*a13 + a2*a23 + a2*b1 + a2*b9 + a2*b25 + a2*b27 + a13*b1 + a13
        + a23 + b9 + b25 + b27 + c1 + d2 + d13 + 1,
    a1*a13 + a2*a23 + a2*b1 + a2*b25 + a2*b26 + a2 + b1,
    a1 + a2*b1 + a2*b27 + a2 + a14 + b1 + b9 + 1,
    a1*a2 + a2*a13 + a2*a23 + a2*b1 + a2*b9 + a2*b25 + a2*b27 + a2 + a13 + a23 + b1 + b9 +
        b25 + b27 + d1 + d3 + d14,
    a1*a2 + a2*a13 + a2*b1 + a2*b9 + a2*b25 + a2*b27 + a13*b1 + a14 + a23 + b9 + b25 + b27
        + c1 + d2 + d13 + 1,
    a13*b9 + a13 + a14 + a23 + b10 + b25 + b26 + 1,
    a1*a13 + a2*a23 + a2*b1 + a2*b27 + a2 + a13 + b1 + b9 + 1,
    a13*b1 + a14 + b9 + b10 + d2 + d3 + d13 + d14 + d21,
    a1*a13 + a2*a13 + a13*a23 + a13*b1 + a13*b25 + a13*b27 + a14 + a23 + b9 + b27 + 1,
    a1 + a2 + a13*b9 + a13 + a23 + b1 + b25 + b27 + b52,
    a1 + a2*b1 + a2 + a14 + b9 + b10,
    a1 + a13*b26 + a13 + a23 + b1 + b9 + b25 + b26 + b52,
    a2*b25 + a2*b26 + a13 + b1 + b9 + b10,
    a1*a13 + a1 + a2*a23 + a13 + a14,
    a2*b25 + a2*b26 + a2*b27 + a13 + b9 + 1
]
35  terms
19 9 9 

true
d21
true
true
d14
true
true
d13
true
true
b52
true
true
b27
true
true
b26
true
true
b25
true
true
b10
true
true
b9
true
true
a14
true

In file "AA", line 256, column 14:
>> HERE --------;
                ^
User error: bad syntax
\end{verbatim}

\end{document}